\documentclass[mnsc]{informs3}
\usepackage{lscape}

\usepackage{setspace}
\usepackage[final,colorlinks]{hyperref}
\usepackage{graphicx}
\usepackage{tabularx}
\usepackage{booktabs}
\newcolumntype{Y}{>{\centering\arraybackslash}X}

\usepackage{epsf}
\usepackage{chngcntr}

\usepackage{color}
\iftrue 
\newcommand{\comments}[1]{\textcolor{green}{\textit{#1}}}

\else
\newcommand{\comments}[1]{}

\fi
\newcommand{\removed}[1]{}

\def\min{\mathop{\rm min}}
\def\max{\mathop{\rm max}}

\def\sup{\mathop{\rm sup}}
\def\inf{\mathop{\rm inf}}

\newcommand{\R}{\mathcal{R}}

\renewcommand{\Re}{\mathbb{R}}

\newcommand{\bars}[1]{\overline{\overline{#1}}}

\OneAndAHalfSpacedXI

\usepackage{natbib}
 \bibpunct[, ]{(}{)}{,}{a}{}{,}%
 \def\bibfont{\small}%
\TheoremsNumberedThrough     
\ECRepeatTheorems

\EquationsNumberedThrough    

\MANUSCRIPTNO{}

\begin{document}


\RUNAUTHOR{Li}

\RUNTITLE{Inverse Optimization of Convex Risk Functions}

\TITLE{Inverse Optimization of Convex Risk Functions}

\ARTICLEAUTHORS{%
\AUTHOR{Jonathan Yu-Meng Li}
\AFF{Telfer School of Management, University of Ottawa, Ottawa, Ontario K1N6N5, Canada,
\EMAIL{Jonathan.Li@telfer.uOttawa.ca}} 
} 

\ABSTRACT{%
The theory of convex risk functions has now been well established as the basis for identifying the families of risk functions that should be used in risk-averse optimization problems. Despite its theoretical appeal, the implementation of a convex risk function remains difficult, as there is little guidance regarding how a convex risk function should be chosen so that it also well represents a decision maker's subjective risk preference. In this paper, we address this issue through the lens of inverse optimization. Specifically, given solution data from some (forward) risk-averse optimization problem (i.e., a risk minimization problem with known constraints), we develop an inverse optimization framework that generates a risk function that renders the solutions optimal for the forward problem. The framework incorporates the well-known properties of convex risk functions---namely, monotonicity, convexity, translation invariance, and law invariance---as the general information about candidate risk functions, as well as feedback from individuals---which include an initial estimate of the risk function and pairwise comparisons among random losses---as the more specific information. Our framework is particularly novel in that unlike classical inverse optimization, it does not require making any parametric assumption about the risk function (i.e., it is non-parametric). We show how the resulting inverse optimization problems can be reformulated as convex programs and are polynomially solvable if the corresponding forward problems are polynomially solvable. We illustrate the imputed risk functions in a portfolio selection problem and demonstrate their practical value using real-life data.

}%

\maketitle

\section{Introduction} 
The theory of convex risk functions, established since the work of \cite{artzner:coherentRM} and later generalized by \cite{follmer02:cvxRiskMeas}, \cite{Ruszczynski:2006aa}, and others, has played a central role in the development of modern risk-averse optimization models. The work of \cite{Ruszczynski:2006aa} in particular brings to light the intimate relationship between convex risk functions and optimization theory and provides necessary tools for analyzing the tractability of risk-averse optimization problems involving convex risk functions. The unified scheme that \cite{Ruszczynski:2006aa} provided through convex analysis also explains the success of several convex risk functions that have now been widely applied for risk minimization, among which the most well known is arguably Conditional Value-at-Risk (CVaR) (\cite{Rockafellar:optimizationCVAR}). 

It was however not for the purpose of optimization (i.e., risk minimization), at least not solely, that the theory was first established. Rather, the motivation lay in the need for alternative measures of risk that could better characterize how individuals perceive risk. For example, the property of convexity, which led to the term ``convex" risk function, was postulated by the theory as an essential and universal characteristic of how risk-averse individuals would perceive risk, namely that diversification should not increase risk. The industry-standard measure of risk, Value-at-Risk (VaR), unfortunately does not satisfy the property of convexity, whereas CVaR, as its counterpart,  which does satisfy convexity, has become a popular theory-supported alternative. Other properties of the theory that have also been widely referenced in justifying the choice of a measure for risk include monotonicity and translation invariance (\cite{follmer02:cvxRiskMeas}), law invariance (\cite{kusuoka:licrm}), positive homogeneity (\cite{artzner:coherentRM}), and comonotonicity (\cite{Acerbi:smmrcrsra}), among others. Each of these properties represents a certain well-grounded rationale for how risk might be perceived over random variables. Some are applicable fairly generally (e.g., monotonicity and law invariance), whereas some others can be domain dependent (e.g., positive homogeneity and comonotonicity).

However, despite the general attractive features of convex risk functions from the point of view of both optimization and risk modelling, very little guidance has been provided to date regarding how to choose a convex risk function that can also well represent a decision maker's subjective perception of risk. In current practice, the choice of a convex risk function is mostly ad hoc and involves very little knowledge of decision makers' true risk preferences. This raises the question of how ones' risk preferences may be observed and how to generate a convex risk function that complies with the observed preferences. \cite{Delage:2015aa} appear to be the first to address this question, by proposing a means to construct a convex risk function from the assessments provided by the decision maker, who compares pairs of risky random losses. Their work is closely related to the scheme of preference (or utility) elicitation (see, e.g., \cite{RT:2014aa}), where queries are considered for extracting users' preferences in establishing their utility functions. One of the main challenges facing this line of inquiry is that in reality decision makers may only be able to provide limited responses because of potential time and cognitive constraints, and thus the elicited preference information is often incomplete. This situation is formulated in \cite{Delage:2015aa}
as a preference robust optimization problem where a worst-case risk
measure is sought that complies with a finite number of pairwise preference relations elicited from the decision maker. Similar ideas can be found also in the context of expected utility theory. \cite{armbruster:dmupii} and \cite{jianhu:rdmosrt}
consider the formulation of a worst-case expected utility function based on limited preference information, whereas \cite{Boutilier06:couemdc} considers a worst-case regret criterion over utility functions. 

In this paper, we attempt to provide an alternative perspective on
the search of a convex risk function that takes into account decision makers'
true risk preferences, namely through the lens of inverse optimization. 
The motivation is that in many current applications, it becomes possible to have
access to the record of the decisions made by individuals, and the past decisions, if optimal, provide useful preference information. Such kinds of preference information may be viewed as a special form of pairwise preference relations, where the random variable chosen according to a made decision is considered preferable to the random variables that could be chosen by alternative decisions. In the case that the alternative decisions are finite, the pairwise preference relations are also finite, which can then be handled by existing frameworks such as \cite{Delage:2015aa}. However, this work emphasizes the case where the alternative decisions may be described through a convex set, which leads to infinitely many pairwise relations that existing frameworks cannot handle.
Moreover, we also recognize that, in practice, even though individuals may perceive risk differently, they often start by agreeing upon some seemingly reasonable risk measure. They then adjust their measurement of risk after receiving more precise preference information. One such example is that many investors tend to follow the principle of {\it safety first} (\cite{Roy:1952aa}), which states that the top concern of an investor is to avoid a possible catastrophic event. They would thus naturally start by choosing a downside risk measure that they feel safe enough to apply (e.g., \cite{Fabozzi:2015aa}).  For example, this could be the CVaR risk measure that has now been widely applied in various areas. Although staying aligned with the downside risk measure is desirable, decisions made by individuals are often inconsistent with (e.g., more aggressive than) what the risk measure prescribes. In this paper, we refer to such a risk measure as a reference risk function, which should be followed closely before more precise preference information can be revealed.  A natural framework to address the above issues is the setup of inverse optimization; namely, given the solutions for some forward problem (i.e., a decision optimization problem with known constraints), the inverse problem seeks a risk function that renders the solutions optimal for the forward problem by minimally deviating from the reference risk function. Our formulation of the inverse problem will allow for 
incorporating preference information in both the forms of pairwise relations and ``most preferable" decisions in convex sets of alternatives, and also the important properties of convex risk functions, namely the monotonicity, convexity, translation invariance, and law invariance. We show how the resulting inverse optimization problems can be tractably analyzed by applying conjugate duality theory (\cite{Rockafellar:1974aa}).

To the best of our knowledge, little has been discussed in the literature
about inverse optimization for convex risk functions. \cite{Bertsimas:2012aa} 
considered inverse optimization for a financial application involving the use of coherent risk measures, but they assumed that the measure is given a priori and focused instead on the
estimation of parameters characterizing random returns and risk budgets.
\cite{Iyengar:2005aa} also applied inverse optimization
to estimate parameters of expected returns in a financial problem.
More generally, inverse optimization methods have been developed for
linear programs (\cite{Ahuja:2001aa}, 
\cite{Dempe:2006aa}), conic programs (\cite{Iyengar:2005aa}),
and convex separable programs (\cite{Zhang:2010aa})
for estimating the parameters that characterize the programs. Early
works also include \cite{Burton:1992aa},  \cite{Zhang:1996aa}, and  \cite{Hochbaum:2003aa}, who focused on network and combinatorial optimization problems (see 
\cite{Heuberger:2004aa} for a survey), whereas more recent works
include  \cite{Schaefer:2009aa} on integer programs,  \cite{Chan:2014aa} on multi-objective programs, \cite{Ghate:2015aa}
on countably infinite linear programs, \cite{ChanLeeTe2018} on the issue of sub-optimality of an observed solution, and \cite{Keshavarz:2011aa}, \cite{Bertsimas:2014aa},  \cite{Aswani:2015aa}, and  \cite{Mohajerin-Esfahani:2015aa} on various issues related to the observations of multiple responses from an agent solving a parametric optimization problem.

In much of the literature referenced above, the problems are structured in a
parametric fashion, and the goal is to estimate the parameters that
characterize the forward problems from observed decisions. However, the
parametric assumption is too limiting for the purpose of identifying
a decision maker's true risk function because it restricts the class of functions
to which the true risk function may belong. It also provides no guarantee
regarding the convergence to the true risk function even if some elicited
information, such as pairwise preference relations, is available. In contrast,
the inverse optimization formulations presented in this paper are
parameter-free and search over the entire space of convex risk functions
for the true risk function. With the collection of more elicited information, their
solutions can converge to the true
risk function, if it is a convex risk function. In this sense, our work broadens the scope
of inverse optimization and opens the door for nonparametric approaches
to function estimation through inverse optimization. However, we should note that \cite{Bertsimas:2014aa} provide a kernel method for inverse optimization, which can also be considered non-parametric. While their method focuses on estimating a function characterizing all the subgradients associated with an unknown function, the method developed in this paper addresses directly the estimation of the unknown function. More detailed discussions along this line are provided later in Section \ref{iconv}. We should also mention that the inverse problem considered in this paper generally falls into the class of inverse problems that focus on estimating the objective function of an optimization problem. Although this class of inverse problems is known to be tractable in a parametric setting, the class of inverse problems that seek to impute parameters defining the feasible region of an optimization problem is generally much less tractable (see, e.g., \cite{Birge:2017aa}). This paper's focus on the former may help explain why tractably solving the inverse problem, even in a non-parametric setting, is a reasonable hope.

Our formulation of the inverse problem does require, albeit implicitly, an assumption that may affect the scope of application of our inverse models, and we should point it out here. Namely, our formulation assumes that in the case where decisions are based on some probabilistic views (or beliefs) of the decision makers, these views (i.e., the probabilistic assessments of random outcomes), either stay unchanged over time or, if not, can always be disclosed together with the decisions made. This may not always be possible because some decision makers may only have a vague sense of probability driven by their intuition and find it cognitively too demanding to articulate how their views change over time. As extensively discussed in this paper, while it is possible to address the case where decision makers take no view or a constant (personal) view on the likelihood of outcomes,  namely by applying non-law-invariant risk functions, it remains an open question as to how a (law-invariant) risk function can possibly be learned from past decisions when the decision makers did hold different probabilistic views over time but were not able to disclose them. This question is important because the scenario, as described, could happen in practice; and from a statistical point of view, the use of law-invariant risk function can be necessary. As a preliminary step, we have conducted some experiments, which we describe in this paper, to examine first how the performances of the solutions optimized based on imputed risk functions may be affected by the misspecification of probability distributions in our inverse models (i.e., distributions applied in the models are inconsistent with the actual views of the decision maker). We observe that, despite the misspecification, the performances can still be noticeably improved towards the optimal performances as more decision data are incorporated into the inverse models. However, it still requires the full development of a formal and rigorous theory to satisfactorily answer the question, and we leave this for future study. 

One natural application of our inverse optimization framework is to identify a risk function that captures the risk preference of an investor from his/her past investment decisions (see, e.g.,  \cite{Delage:2015aa}). While this is the application that we primarily focus on in this paper, we should point out that our framework can also be naturally applied to other settings that involve budget allocation decisions under uncertainty. For instance, \cite{Haskell:2018aa} consider a setting of homeland security and the problem of learning the risk preference of the Department of Homeland Security (DHS), which makes decisions to allocate budget across a number of cities so as to protect them from potential terrorist attacks. While covering these other applications in depth may go beyond the scope of this paper, the framework 
established in this paper shall provide the basis for further exploring these other applications.

We briefly summarize our main contributions below:
\begin{enumerate}
\item We develop for the first time an inverse-optimization framework for
convex risk functions that generates a risk function incorporating
the following information: 1) the properties of monotonicity, convexity, translation invariance, and law invariance promoted in the theory of convex risk functions, 2) observable optimal solutions from forward problems, 3) a reference risk function, and 4) elicited pairwise preference relations.

\item We formulate the inverse optimization problem in a non-parametric
fashion and show that for a large number of cases, the computational tractability of the inverse problem is largely determined by the forward problem: namely that the former is polynomially solvable if the latter is so.

\item Methodologically speaking, we show that to solve the inverse problem it suffices to search over a particular class of risk functions for an optimal solution. This class of risk functions is representable in terms of the random variables resulting from the observed decisions and a fixed set of parameters. Based on this representation, the inverse problem reduces to determining the values of the parameters in the representation, and these values correspond to the amounts of risk estimated for the chosen random variables. The number of parameters needed in the representation is set by the number of observed decisions. This offers an intuitive and computationally tractable interpretation of the (non-parametric) inverse problem. 

\item We demonstrate the application of our framework in a portfolio selection
problem and provide computational evidence that the imputed risk functions utilize well 
the preference information contained in observable
solutions and a reference risk function. This leads to a solution that
can be well justified in terms of both its performance evaluated based
on the true risk function and the reference risk function. We also demonstrate how quickly the performances of the solutions optimized based on the imputed risk functions can converge to the performances of the solutions optimized based on the true risk function, as the number of observed decisions increases. 
\end{enumerate}

\section{Forward and Inverse Problem of Risk Minimization}
We begin by formalizing the forward problem of risk minimization and characterizing the problem using the theory of convex risk functions. We then proceed to the formulation of the inverse problem. 

\subsection{Forward problem of risk minimization} \label{sec21}
 Our general setup of the forward problem follows closely the setup in the literature of choice over acts (i.e., random variables $Z: \Omega \rightarrow \Re$) (\cite{Savage:1954aa}). In this setup, it is assumed that a decision maker's preference over random variables can be specified and that the forward problem seeks the most preferable random variables. In the special case where a probability measure $\mathbb{P}$ can be identified over the sigma-algebra $\Sigma$ of sample space $\Omega$, the preference can be alternatively defined over the distributions, denoted by $F_Z$, of the random variables. This special case is considered in the literature of choice over lotteries (i.e., distributions) (\cite{Von-Neumann:1944aa}). 

We will first continue formalizing the forward problem under the general setting, and the special case will follow naturally as we proceed. Without loss of generality, we assume that any random variable $Z$ represents some form of loss, by which we mean that it has the interpretation that for any $\omega \in \Omega$, the larger the value of $Z(\omega)$ is, the worse it is. We denote by $\succeq$ a system of preference relations that is complete, transitive, and continuous \footnote{These are the conditions required to ensure that there exists a function $\rho$ that captures the system of preference relations (\cite{debreu:rponf}).}, where $Z_1 \succeq Z_2$ denotes that $Z_1$ is preferred to $Z_2$. A risk function $\rho$ is a numerical representation that captures the preference relation $\succeq$ in terms of the riskiness of random losses  (i.e., a random loss $Z$ is preferable if it is perceived less risky). In this paper, we will focus on the case where a risk function $\rho$ is defined over random losses based on a sample space with finitely many outcomes $\Omega:=\{\omega_{i}\}_{i=1}^{M}$. In this setting, any random loss $Z$ can be represented also by a vector $\vec{Z} \in \Re^{|\Omega|}$, where $(\vec{Z})_i = Z(\omega_i)$, and the random loss resulting from a decision $x \in \Re^{n}$ can be written by $\vec{Z}(x) : = (Z(x,\omega_1),...,Z(x,\omega_M))^\top \in\Re^{|\Omega|}$. If a random loss $Z_2$ is perceived at least as risky as $Z_1$  (i.e., $Z_1 \succeq Z_2$), the risk function $\rho:\Re^{|\Omega|}\rightarrow \Re$
should satisfy $\rho(\vec{Z}_{1})\leq\rho(\vec{Z}_{2})$. Accordingly, a solution $x^*$ is optimal if and only if it satisfies
$\rho(\vec{Z}(x^*))\leq\rho(\vec{Z}(x)$), $\forall x\in{\cal X}$,
and a risk minimization problem can be formulated as
\begin{equation}
\min_{x\in{\cal X}}\rho(\vec{Z}(x)). \label{forward}
\end{equation}
Throughout this paper, we assume that the function $\vec{Z}(x)$ is convex in $x$ (i.e., $Z(x,\omega_i)$ is convex in $x$ for all $\omega_i \in \Omega$), and the feasible set ${\cal X}\subseteq\Re^{n}$ is a convex set.

It is hypothesized in the theory of convex risk functions (\cite{follmer02:cvxRiskMeas}) that any risk function $\rho$ that represents a reasonable or ``rational" 
preference system would satisfy certain axioms. The most widely known ones are the following three:
\begin{enumerate} 
\item (Monotonicity) $\rho(\vec{Z}_{1})\leq\rho(\vec{Z}_{2})$ for any $Z_{1}(\omega)\leq Z_{2}(\omega),\;\forall \omega \in \Omega$, 
\item (Convexity) $\rho(\lambda\vec{Z}_{1}+(1-\lambda)\vec{Z}_{2})\leq \lambda\rho(\vec{Z}_{1})+(1-\lambda)\rho(\vec{Z}_{2})$, where $0\leq\lambda\leq1$, and
\item (Translation Invariance) $\rho(\vec{Z}+c)=\rho(\vec{Z})+c$.
\end{enumerate}

The first axiom, monotonicity, captures the fact that any reasonable preference system would never prefer a random loss that is known to have higher loss for any possible outcome (i.e., any $Z_1$, $Z_2$ such that $Z_1(\omega)\leq Z_2(\omega)$, $\forall \omega \in \Omega$ must lead to $Z_1 \succeq Z_2$). The axiom of convexity describes the diversification preference, namely that any convex combination (diversification) $\lambda Z_1 + (1-\lambda) Z_2$ must be (at least equally) preferable to non-diversified counterparts (i.e., $Z_1$ or $Z_2$). Lastly, translation invariance is necessary when a monetary interpretation of risk is required, which is the case in finance where a deterministic amount such as cash can always be used to offset the risk by the same amount. The corresponding preference system is insensitive to any constant amount $c$ added to or subtracted from all random losses (i.e., $Z_1 \succeq Z_2 \Rightarrow Z_1+c \succeq Z_2+c$). Based on these three axioms, we define the following class of functions that capture risk-averse preferences.
\vspace{5pt}
\begin{definition} ({\it Risk-averse functions}) Let ${\cal R}$ denote the set of functions that 
satisfy axioms (1)--(3) and $\rho(\vec{0})=0$.
\end{definition}
\vspace{5pt}
The condition $\rho(\vec{0})=0$ is imposed for the purpose of normalization so that risk estimated based on different risk functions is comparable. We now also formalize the special case where a probability measure $\mathbb{P}$ can be defined over $(\Omega, \Sigma)$. This naturally leads to the consideration of the following axiom for a risk function (see, e.g., \cite{kusuoka:licrm}). 
\begin{enumerate} \setcounter{enumi}{3}
\item (Law Invariance) $\rho(\vec{Z}_{1})=\rho(\vec{Z}_{2})$, for any $Z_{1}\sim_{\mathbb{P}}Z_{2}$ (distributionally equivalent).
\end{enumerate}
The above axiom immediately implies that the risk function in this case is essentially a function of distributions. We define the following class of functions when the distributions for all random variables are available.
\vspace{5pt}
\begin{definition} ({\it Law-invariant risk-averse functions}) 
Let ${\cal R}_F \subset {\cal R}$ denote the set of risk-averse functions that are law invariant. Without loss of generality, we can equivalently write $\rho(\vec{Z})$ as $\rho(F_Z)$  (i.e., a function of distributions $F_Z$).
\end{definition}
\vspace{5pt}

Depending on the decision maker's knowledge about the distribution, one may decide which class of risk functions (i.e., risk-averse versus law-invariant risk-averse) is more appropriate to assume in defining the forward problem. In particular, we provide the following examples, which cover three possible cases: the case with no distribution, the case with ambiguous distributions, and the case with a specific distribution. The first two cases can be treated as special cases of our general setup (i.e., preference over acts), whereas the third case invokes the property of law invariance.

\begin{example} (General case) \label{exx1}
If a decision maker's choices are made by knowing only that there are $|\Omega|$ possible outcomes for the uncertain losses, we may assume that the individual solves a forward problem based on a certain risk-averse function $\rho \in {\cal R}$.
\end{example} 

\begin{example} (Distributional ambiguity) \label{exx2}
If a decision maker actually has in mind a certain distribution-based convex risk measure $\hat{\rho}(\cdot ; q)$, where $q$ denotes the distribution, for evaluating risk but is concerned about the estimation error associated with an empirical distribution $\hat{p}$, we may assume that she is ambiguity-averse and that her choices follow a distributionally robust version of the risk measure. Namely, it can take the following form with uncertain probability $q$:
$$ \rho^{\uparrow}(Z) = \sup_{q} \left\{ \hat{\rho}(Z; q)\;\middle |\; D(q, \hat{p}) \leq d,\; \vec{1}^\top q = 1, \; q \geq 0 \right\}, $$
where $D$ measures the difference between two distributions and is usually defined based on some $\phi$-divergence function (see \cite{Ben-Tal:2013aa} for more details). It is not hard to confirm that in this case $\rho^{\uparrow} \in {\cal R}$. Of course, in reality we do not know the structure of $\hat{\rho}$ and $D$ and may assume only that the true risk function is a risk-averse function $\rho \in {\cal R}$. 
\end{example}

\begin{example} (Known Distribution) \label{exx3}
Suppose that a decision maker can identify which distribution $q$ to use in $\hat{\rho}(\cdot ; q)$ and that the information about the distribution is available. In this case, we may assume that the individual solves a certain law-invariant risk-averse function $\rho \in {\cal R}_F$ based on the distribution $q$. 
\end{example}

\begin{remark} \label{remm1}
One scenario that can also be of practical interest but is not covered above is the case where a decision maker solves some distribution-based risk measure $\hat{\rho}(\cdot ; q)$ but the individual does not reveal which distribution $q$ is used. While it remains possible to address this case if $q$ stays unchanged over time, namely by assuming that the true risk function is a risk-averse function $\rho \in {\cal R}$ because $\hat{\rho}(\cdot ; q) \in {\cal R}$ for a fixed $q$, it becomes less clear how to address the case when the distribution $q$ may differ from one time point to another. In the latter, the observed decisions can actually appear inconsistent with the assumption $\rho \in {\cal R}$. 
This can be easily seen by considering for instance the case where the decision maker's true risk function is the simple expected value function. 
We may observe between two random variables $X_i$ and $X_j$ that at one point $\mathbb{E}_{q_1}[X_i] > \mathbb{E}_{q_1}[{X_j}]$ and at the other point  $\mathbb{E}_{q_2}[X_i] < \mathbb{E}_{q_2}[{X_j}]$ based on two different distributions $q_1$ and $q_2$. Clearly, there exists no $\rho \in {\cal R}$ that can capture such preferences. For this reason, we find it necessary to assume throughout this paper that in the case where the decision makers applied different distributions, these distributions can always be disclosed together with the decisions made. Later in the numerical section, Section \ref{52}, we will revisit this assumption and address the case where the disclosed distributions may not be fully accurate.
\end{remark}


\subsection{Inverse problem of risk minimization}
In the inverse problem, the risk function $\rho$ is unknown, but one has access to decisions made according to forward problems as defined in the previous section.
The goal is to generate a risk function $\rho$ that renders the observed decisions as optimal as possible in the forward problem. 
Specifically, let $(x^t, \vec{Z}^t(\cdot), {\cal X}^t)$ represent each observation, which denotes that $x^t$ was made with respect to the random vector $\vec{Z}^t(\cdot)$ and the feasible region ${\cal X}^t$. We can write down the following optimality condition that characterizes the risk function $\rho$ through the observed decisions. 
\vspace{5pt}
\begin{description}
\item[Optimality Condition:] Given a list of observations $\left\{(x^t, \vec{Z}^t(\cdot), {\cal X}^t)\right\}_{t \in {\cal T}}$, where $|{\cal T}|<\infty$, the set of risk functions that render the decisions optimal admits 
$$ {\cal R}_{inv}  := \left\{ \rho \;\;\middle |\;\; \rho(\vec{Z}^{t}(x^t)) \leq \rho(\vec{Z}^{t}(x)),\; \forall x \in {\cal X}^{t},\; t \in {\cal T}\right\}. $$
\end{description}

Note that the risk function $\rho$ does not depend on $t$  (i.e., the decision maker's risk preference is assumed to stay constant when the past decisions were made). In the case where $|{\cal T}|$ is small, the above set may not be sufficient to build a meaningful inverse problem, since it can contain some degenerate form of functions. It would thus be necessary to assume that some ``prior" knowledge about the risk function can be acquired. The most direct way to acquire such knowledge is through preference elicitation (e.g., \cite{RT:2014aa}) where the decision maker would be asked to make comparisons among a selective list of random variables.  We can also borrow the concept of reference solution from the literature of inverse optimization, which stands for a solution that can be used as a reference while searching for an alternative better solution. 
\vspace{5pt}
\begin{description}
\item[Elicited Preference Relations:] Given a list of pairs of random losses $\left\{(L_k,U_k)\right\}_{k\in {\cal K}}$, where $|{\cal K}|<\infty$, that satisfy $L_k \succeq U_k$ for $k \in {\cal K}$, we define the set
$$ {\cal R}_{el}(\{(L_k,U_k)\}_{k\in {\cal K}}) := \left\{ \rho \;\;\middle |\;\; \rho(\vec{L}_k) \leq \rho(\vec{U}_k),\;\forall k\in {\cal K}\right\}.$$
\end{description}
\vspace{5pt}
\begin{description}
\item[Reference Risk Function:] Given a reference risk function $\tilde{\rho} \in {\cal R}$ and a parameter $\epsilon \geq 0$ that describes the maximum discrepancy between the candidate risk function and the reference risk function, the following set of risk functions can be defined accordingly:
$$ {\cal R}_{ref}(\epsilon) := \left\{ \rho \;\;\middle |\;\; || \rho - \tilde{\rho} ||_{\infty} \leq \epsilon \right\}, $$
where $||\cdot||_{\infty}$ stands for the infinity norm defined over functions $\rho:\Re^{|\Omega|}\rightarrow \Re$.
\end{description}
\vspace{5pt}

\begin{remark}
The constraints defining ${\cal R}_{inv}$ can also be viewed as a special case of the constraints inferred from preference elicitation. While the random losses $\vec{Z}^{t}(x)$ considered in ${\cal R}_{inv}$ are ``forward-problem" dependent, in preference elicitation any pair of random losses may be considered for comparison. We take such a unified perspective in formulating the inverse problem. 
\end{remark}

\begin{remark} \label{rm3}
It is worth mentioning here that it is also possible to consider norms other than the infinity norm in the definition of ${\cal R}_{ref}$. For instance, one might consider $L^2$ norm $||\rho||_{\phi}:=(\int \rho(z)^2 \phi(dz))^{1/2}$ for some probability measure $\phi$. However, it may not be clear whether this is practically useful because one may find it hard to specify the measure $\phi$ and to interpret the norm. In the case where $\phi$ is discrete, which requires only comparing $\rho$ and $\tilde{\rho}$ over finite points, the analysis presented in this paper can easily accommodate such a case.
\end{remark}

With the above definitions of ${\cal R}_{inv}$, ${\cal R}_{el}$, and ${\cal R}_{ref}$, one may consider different criteria to determine which risk function described by the above sets is the optimal choice. In particular, we consider the following four criteria. The first three are primarily concerned about the fitting of risk levels, whereas the fourth one addresses the fitting of prescribed decisions. The first criterion follows most closely the spirit of classical inverse optimization.
\begin{enumerate}
\item Minimizing the deviation from the reference risk function $\tilde{\rho}$:
\begin{eqnarray}
\inf_{\rho,\epsilon \in \Re} && \epsilon \nonumber \\
\text{subject to} &&  \rho \in {\cal R}\; ({\rm or}\; {\cal R}_F \text{ in the case of law invariance}), \label{eq:inv1} \\
&& \rho \in {\cal R}_{ref}(\epsilon) \cap {\cal R}_{inv} \cap {\cal R}_{el}(\{(L_k,U_k)\}_{k\in {\cal K}}).  \nonumber
\end{eqnarray}
\end{enumerate}

This model assumes that in addition to capturing the preferences implied by the observed decisions, there is a practical need to stay aligned with the reference risk function whenever possible. As mentioned in the introduction, the reference risk function can for example be the CVaR risk measure used to implement the safety-first principle. In general, one can apply the model to generate an alternative risk function that not only renders the observed decisions optimal but also maximally aligns with the chosen downside risk measure.

One may argue that in some cases the observed decisions might not be ``rigorously" optimal given that human beings are not perfectly rational. We can accommodate such a possibility by replacing the set ${\cal R}_{inv}$ with the following set based on sub-optimality:
$$ {\cal R}_{inv}(\gamma) = \left\{ \rho \;\;\middle |\;\; \rho(\vec{Z}^{t}(x^t)) \leq \rho(\vec{Z}^{t}(x)) + \gamma_t ,\; \forall x \in {\cal X}^{t},\; t \in {\cal T}\right\}. $$
Thus, if observed decisions are known to be sub-optimal, the following model may be considered that seeks to close the optimality gap.
\begin{enumerate} \setcounter{enumi}{1}
\item Minimizing the sub-optimality of observed decisions: 
\begin{eqnarray}
\inf_{\rho, \gamma \in \Re^{|{\cal T}|}} && \sum_{t \in {\cal T}} \gamma_t \nonumber \\
\text{subject to} &&  \rho \in {\cal R}\; ({\rm or}\; {\cal R}_F \text{ in the case of law invariance}),  \label{eq:inv2} \\
&& \rho \in {\cal R}_{inv}(\gamma) \cap {\cal R}_{el}(\{(L_k,U_k)\}_{k\in {\cal K}}) \cap {\cal R}_{ref}(\epsilon^*), \nonumber
\end{eqnarray}
where $\epsilon^* \in (0, \infty]$ is fixed beforehand.
\end{enumerate}

Here, one's priority is to ensure that the imputed risk function will render the observed decisions as favorable as possible. One may choose to set $\epsilon^*=\infty$ if needed. The central idea behind this model is that the observed decisions, albeit sub-optimal, still closely follow the decision maker's true preference. However, if the accuracy of the decision data is in doubt, one may instead fix $\gamma = \gamma^*$ to some value $\gamma^*$ larger than the optimal solution in (\ref{eq:inv2}). Moreover, if the concern is about potential risk underestimation, the following model provides a means to examine the worst possible risk.
\begin{enumerate} \setcounter{enumi}{2}
\item For any $\vec{Z}$, seek a worst-case risk function
\footnote{Although it may not be immediately obvious, one can actually confirm that the worst-case function $\rho$ is itself a feasible solution to the constraints in (\ref{eq:inv3}) (see Lemma \ref{lemo}) and hence the problem can be equivalently stated as: seek a function 
$\rho \in {\cal R} (\text{or }{\cal R}_F) \cap {\cal R}_{inv} (\gamma^*) \cap  {\cal R}_{el}(\{(L_k,U_k)\}_{k\in {\cal K}}) \cap {\cal R}_{ref}(\epsilon^*) $ such that
$ \rho \geq \rho' ,\;\; \forall \rho' \in {\cal R} (\text{or }{\cal R}_F) \cap {\cal R}_{inv}(\gamma^*) \cap {\cal R}_{el}(\{(L_k,U_k)\}_{k\in {\cal K}}) \cap {\cal R}_{ref}(\epsilon^*) .
$}
: 
\begin{eqnarray}
\rho(\vec{Z}) :=  &\;\;\;\;\;\;\;\;\;\;\;\;\;\; \sup_{\rho'} &   \rho'(\vec{Z}) \nonumber \\
& \text{subject to} & \rho' \in {\cal R}\; ({\rm or}\; {\cal R}_F \text{ in the case of law invariance}),  \label{eq:inv3} \\
&& \rho' \in {\cal R}_{inv} (\gamma^*) \cap {\cal R}_{el}(\{(L_k,U_k)\}_{k\in {\cal K}}) \cap {\cal R}_{ref}(\epsilon^*),  \nonumber
\end{eqnarray}
where $\gamma^* \in [0,\infty)^{|{\cal T}|}$ and $\epsilon^* \in (0, \infty]$ are fixed beforehand.
\end{enumerate}

The last criterion we consider is closely related to the idea behind the second criterion, where the goal is to reconcile the decision data. The difference is that not only do we like to ensure the optimal solution generated from the imputed risk function is aligned with the decision maker's risk preference, but we also want the solution itself to be close to the observed decision. To formalize this, we define the following set parameterized by a decision $x'$:
$$ \tilde{{\cal R}}_{inv}(x')  = \left\{ \rho \;\;\middle |\;\; \rho(\vec{Z}^T(x')) \leq \rho(\vec{Z}^T(x)),\;\forall x \in {\cal X}^T\ \right\}.$$
Note that to facilitate our later discussion in Section \ref{sec4}, we consider here only the case of single observation indexed by $T$ (indicating the most recent observation). Clearly, by fixing $x':=x^T$, the above set reduces to ${\cal R}_{inv}$. Our last criterion can be modelled as follows.
\begin{enumerate} \setcounter{enumi}{3}
\item Minimizing the deviation from an observed decision: 
\begin{eqnarray}
\inf_{\rho, x' \in {\cal X}} && ||x'-x^T|| \nonumber \\
\text{subject to} &&  \rho \in {\cal R}\; ({\rm or}\; {\cal R}_F \text{ in the case of law invariance}),  \label{eq:inv4} \\
&& \rho \in \tilde{{\cal R}}_{inv}(x') \cap {\cal R}_{ref}(\epsilon^*), \nonumber
\end{eqnarray}
where $||\cdot||$ is an arbitrary norm and $\epsilon^* \in (0,\infty]$ is fixed beforehand. 
\end{enumerate}

More detailed discussion about the case of multiple observations will be given in 
Section \ref{sec4}. The above model is particularly useful when there is a preference for the status quo decision $x^T$. For instance, an investor can prefer that the portfolio generated from the imputed risk function does not differ much from his most current portfolio.

\section{Solving the Inverse Problems} \label{iconv}
In this section, we address first the inverse problems (\ref{eq:inv1})--(\ref{eq:inv3}). In particular, we will base our discussions primarily on the inverse problem (\ref{eq:inv1}), which best highlights the fundamental complexity behind all the inverse problems. Once we walk through the steps it takes to resolve the complexity, it will also be clear how to solve inverse problems (\ref{eq:inv2})--(\ref{eq:inv3}). To put the problem into perspective, let us recast first the inverse problem (\ref{eq:inv1}) into
\begin{eqnarray}
\inf_{\rho} && ||\rho-\tilde{\rho}||_{\infty} \nonumber \\
\label{uniforminv} {\rm subject\;to}  && \rho \in {\cal R}\;\; (\text{or } {\cal R}_F \text{ in the case of law invariance}), \\
&& \rho(\vec{W}^{t}) \leq \rho(\vec{W}),\; \forall  \vec{W} \in {\cal W}^{t},\;  t \in {\cal T}, \label{invv}  \\
 &&  \rho(L_k) \leq \rho(U_k),\; \forall k\in {\cal K},  \label{inv2}
\end{eqnarray}
where $\vec{W}^{t} :=\vec{Z}^t(x^{t})$ and ${\cal W}^{t}:= \left\{\vec{Z}^t(x) \;\middle |\;x\in{\cal X}^{t} \right\}$.  

The above problem cannot be solved by traditional analysis for inverse optimization due to the  non-parametric nature of the risk function $\rho$. In a more specialized setting where one removes the optimality constraints (\ref{invv}) and replaces the norm $||\cdot||_{\infty}$ by a norm that requires only comparing $\rho$ and $\tilde{\rho}$ over finite points (see Remark \ref{rm3}), the problem might be solvable based on the non-parametric method developed in \cite{Delage:2015aa}. This method hinges on the observation that if only finitely many points need to be compared (e.g., (\ref{inv2}) involving only $|{\cal K}| < \infty$ pairs), the problem can be reduced to a finite-dimensional convex program whose size grows polynomially with the number of points needing comparison. Unfortunately, in the setting of the above inverse problem, because of the infinity norm and the optimality constraints (\ref{invv}) it necessarily involves infinitely many points (e.g., $\forall \vec{W} \in {\cal W}^t$) that must be compared and thus renders the method of \cite{Delage:2015aa} inapplicable.

While these difficulties may put in doubt the tractability of the inverse problem (\ref{uniforminv})--(\ref{inv2}), our key finding is that it is possible to bypass the difficulties by new analysis based on conjugate duality theory (see, e.g., \cite{Rockafellar:1974aa}). Our goal here is to present from a high-level perspective the key analysis steps to approach the problem. The technical details of the theory and proofs can be found in Appendix \ref{ssec00} and \ref{ssec03}.

\subsection{Imputing risk-averse functions} \label{gg}
We start by considering the general case where the observed decisions to a forward problem were made over acts, as described in Example \ref{exx1}. That is, we intend to solve the inverse problem 
(\ref{uniforminv})--(\ref{inv2}) with $\rho \in {\cal R}$. From here on, we make the following assumption about the reference risk function $\tilde{\rho}$.
\begin{definition}
A risk function $\rho$ is called a coherent risk measure if $\rho \in {\cal R}$ and it further satisfies $\rho(\lambda \vec{Z}) = \lambda \rho (\vec{Z})$ for any $\lambda \geq 0$ (scale invariance). 
\end{definition}

\begin{assumption} \label{asss1}
The reference risk function $\tilde{\rho} \in {\cal R}$ is a coherent risk measure.
\end{assumption}

The above assumption is not stringent because most risk measures applied in practice are coherent risk measures. Moreover, we apply the following well-known representation result of coherent risk measures and assume that such a representation is available (see Appendix \ref{ssec01} for the representations of several popular risk measures).

\begin{theorem} (\cite{artzner:coherentRM})
Any coherent risk measure admits the supremum representation of 
\begin{equation} \label{supcoh}
\rho(\vec{Z})=\sup_{p \in {\cal C}} p^{\top}\vec{Z}, 
\end{equation}
where ${\cal C}$ is a non-empty, closed, convex set of probability measures (i.e.,  
${\cal C} \subseteq \Delta \subseteq \Re^{|\Omega|}$ with $\Delta := \left\{p\in \Re^{|\Omega|}\;\middle |\; \vec{1}^\top p = 1,\;p\geq 0  \right\}$).
\end{theorem}

For convenience, in this paper we say that the coherent risk measure $\rho$ is supported by the set ${\cal C}$. Our first observation to solving the inverse problem (\ref{uniforminv})--(\ref{inv2}) is that it is possible to identify a subset of risk functions that is ``sufficiently" large to contain an optimal solution to the inverse problem. We rely on the following definition to characterize this subset of risk functions.
\begin{definition}
Given a set of random losses $\{\vec{Z}_j\}_{j\in {\bar {\cal J}}}$ for some ${\bar {\cal J}}$ and a set ${\bar {\cal C}} \subseteq \Delta \subseteq \Re^{|\Omega|}$, we say that a function $\rho$ is supported by the pair $(\{\vec{Z}_j\}_{j\in {\bar {\cal J}}}, {\bar {\cal C}})$ if it belongs to the following set of functions:
$$ {\cal L}(\{\vec{Z}_j\}_{j\in {\bar {\cal J}}}, {\bar {\cal C}}) := 
\left\{ \rho_{\delta}\;  \middle | \; \exists \delta \in \mathbb{R}^{|{\bar {\cal J}}|},\; \forall \vec{Z},\;  
\rho_{\delta}(\vec{Z}) =  \sup_{p \in {\bar {\cal C}}}  p^{\top}\vec{Z} - \max_{j\in {\bar {\cal J}}}\left\{p^{\top}\vec{Z}_{j}-\delta_{j}\right\}  \right\}. $$
\end{definition}

We also need the following definition that will be applied throughout the rest of this paper.
\begin{definition}
Let $\{\vec{X}_{j}\}_{j\in {\cal J}}$ be the random losses in the union $\{\vec{W}^{t}\}_{t \in {\cal T}} \cup \{\vec{L}_{k}\}_{k\in {\cal K}}\cup\{\vec{U}_{k}\}_{k\in {\cal K}} \cup \vec{0}$. Without loss of generality, we assume $\{\vec{X}_{t}\}_{t\in {\cal T}} = \{\vec{W}^{t}\}_{t \in {\cal T}}$.
\end{definition}

\begin{proposition} \label{pro4}
Given that Assumption \ref{asss1} holds and the set of optimal solutions is non-empty, there must exist a function $\rho_{\delta} \in {\cal L}(\{\vec{X}_j\}_{j\in {\cal J}}, {\cal C})$, where ${\cal C}$ is the set that supports the reference risk function $\tilde{\rho}$, which is optimal to the inverse problem (\ref{uniforminv})--(\ref{inv2}). Moreover, given any optimal solution $\rho_0$ to the inverse problem (\ref{uniforminv})--(\ref{inv2}), there always exists a function $\rho_{\delta} \in {\cal L}(\{\vec{X}_j\}_{j\in {\cal J}}, {\cal C})$ that is also optimal and bounds from above the solution $\rho_0$:  
\begin{equation}
\rho_{\delta}(\vec{Z}) \geq \rho_0(\vec{Z}),\; \forall \vec{Z},  \label{eqqq}
\end{equation}
namely, by setting $\delta_j = \rho_0(\vec{X}_j)$, $\forall j \in {\cal J}$. 
\end{proposition}
\begin{proof}{Proof of Proposition \ref{pro4}}
The general strategy of the proof is to show that if there exists a risk
function $\rho_0 \in {\cal R}$ that is optimal to the inverse problem with some optimal value $u^{*}<\infty$, there must exist a risk function $\rho_{\delta} \in {\cal L}(\{\vec{X}_j\}_{j\in {\cal J}}, {\cal C})$ that is also optimal to the problem, namely by setting \[\delta_j = \rho_0(\vec{X}_j),\; \forall j \in {\cal J}.\] 
We leave the details of confirming this claim and that it implies (\ref{eqqq}) to Appendix
\ref{ssec03}.
\Halmos
\end{proof}

The above result implies firstly that there is no loss of optimality if we restrict our search of an optimal solution to the set ${\cal L}(\{\vec{X}_j\}_{j\in {\cal J}}, {\cal C})$. This significantly reduces the complexity of solving the inverse problem (\ref{uniforminv})--(\ref{inv2}) because the search over the set ${\cal L}(\{\vec{X}_j\}_{j\in {\cal J}}, {\cal C})$ can be effectively done by searching the space of parameter $\delta$. This provides the assurance that even though one cannot identify a decision maker's risk function through any parametric form, it is still possible to learn the risk function by tuning only a finite number of parameters (i.e., $|{\cal J}|$ many parameters). Moreover, we know from the inequality \eqref{eqqq} that in the case where the optimal solution to the inverse problem is not unique, the optimal solution found in the set ${\cal L}(\{\vec{X}_j\}_{j\in {\cal J}}, {\cal C})$ would be the most ``robust" because it provides the most conservative estimate of risk. What can appear counterintuitive is that the representation of $\rho_{\delta}$ does not depend on the feasible set ${\cal W}^{t} = \left\{\vec{Z}^t(x)\;\middle |\;x\in {\cal X}^{t}\right\}$, and one may wonder how an imputed risk function $\rho_{\delta}$ then takes into account the information about the set ${\cal W}^{t}$.  The short answer is that the information would be incorporated into the risk function $\rho_{\delta}$ when it comes to the point of determining the value of the parameter $\delta$
\footnote{
To provide a better grasp of this, let us suppose for now that one is able to efficiently determine if $\delta$ belongs to the following set:
$$ \Theta: = \left\{ \delta \in \mathbb{R}^{|{\cal J}|}\; \middle |\; \exists \rho \in {\cal R}\cap {\cal R}_{inv} \cap {\cal R}_{el}(\{(L_k,U_k\}_{k\in {\cal K}}) {\rm \; such\; that\;} \rho(\vec{X}_j) = \delta_j ,\;\forall j \in {\cal J} \right\}.$$
Then, one can quickly confirm that because of (\ref{eqqq}) there exists $\delta \in \Theta$ such that $\rho_{\delta}$ will necessarily satisfy the optimality condition (i.e., $\rho_{\delta} \in {\cal R}_{inv}$). Indeed, we have $\rho_{\delta}(\vec{W}^{t}) = \rho_0(\vec{W}^{t}) \leq \rho_0(\vec{Z}) \leq \rho_{\delta}(\vec{Z}),  \;\forall \vec{Z} \in {\cal W}^t$. 
This explains why the complexity actually lies in dealing with the set $\Theta$. 
It is not necessarily clear, however how one can tractably search over the set $\Theta$. The fact that ${\cal L}(\{\vec{X}_j\}_{j\in {\cal J}}, {\cal C}) \supseteq \{ \rho_{\delta} \;|\; \delta \in \Theta\}$ suggests that we may first consider the relaxed problem of searching over the set 
${\cal L}(\{\vec{X}_j\}_{j\in {\cal J}}, {\cal C})$.}.
To show how the value of $\delta$ can be calculated, we first need the following intermediate result. The result shows that searching over the set ${\cal L}(\{\vec{X}_j\}_{j\in {\cal J}}, {\cal C})$ can be equivalently formulated as a finite-dimensional system of constraints over the parameter $\delta$ where the solution $\delta$ corresponds to how function values may be assigned over the set $\{\vec{X}_j\}_{j\in{\cal J}}$ that supports the risk function. 

\begin{proposition} \label{pro2}
Given any $\rho_{\delta} \in {\cal L}(\{\vec{Z}_j\}_{j\in {\bar {\cal J}}}, {\bar {\cal C}})$, there must exist $y_j \in {\bar {\cal C}}$, $j\in {\bar {\cal J}}$ such that $\rho_{\delta}$ satisfies the following system of constraints: 
\begin{eqnarray*}
 && \rho_{\delta}(\vec{Z}_{j})+y_{j}^{\top}(\vec{Z}_{i}-\vec{Z}_{j})\leq\rho_{\delta}(\vec{Z}_{i}),  \; \forall i\neq j,\\
  && y_j \in {\bar {\cal C}},\; \forall j\in {\bar {\cal J}}. \nonumber 
\end{eqnarray*}
Conversely, given any solution $\{y_{j}^{*}\}_{j\in {\bar {\cal J}}},\;\{\delta_{j}^{*}\}_{j\in {\bar {\cal J}}}$
that satisfies the system below:
\begin{eqnarray}
 && \delta_{j}+y_{j}^{\top}(\vec{Z}_{i}-\vec{Z}_{j})\leq\delta_{i},\;  \forall i\neq j,\label{eq:linearsystema}\\
  && y_j \in {\bar {\cal C}},\; \forall j\in {\bar {\cal J}}, \nonumber 
\end{eqnarray}
there must exist a $\rho_{\delta} \in {\cal L}(\{\vec{Z}_j\}_{j\in {\bar {\cal J}}}, {\bar {\cal C}})$ that satisfies $\rho_{\delta}(\vec{Z}_{j})=\delta_{j}^{*}$, namely the function 
\begin{equation}
\rho_{\delta^*}(\vec{Z})=\sup_{y\in {\bar {\cal C}}}y^{\top}\vec{Z}-\max_{j\in {\bar {\cal J}}} \left\{y^{\top}\vec{Z}_{j}-\delta_{j}^{*}\right\}.\label{eq:polyreppeat}
\end{equation}
\end{proposition}

We next show that by restricting ourselves to the search in the set ${\cal L}(\{\vec{X}_j\}_{j\in {\cal J}}, {\cal C})$, we can also identify how to search in the subset
${\cal L}(\{\vec{X}_j\}_{j\in {\cal J}}, {\cal C}) \cap {\cal R}_{inv} \cap {\cal R}(\{(L_k, U_k)\}_{k\in {\cal K}}).$
Namely, it is equivalent to adding additional constraints to the system (\ref{eq:linearsystema}). This final system comprises finite-dimensional convex constraints. 

\begin{proposition} \label{pro3}
Given any $\rho_{\delta} \in {\cal L}(\{\vec{X}_j\}_{j\in {\cal J}}, {\cal C})$, the function further satisfies $\rho_{\delta} \in {\cal R}_{inv} \cap {\cal R}(\{(L_k, U_k)\}_{k\in {\cal K}})$ if and only if there exists $y_j \in {\cal C}$, $j \in {\cal J}$ such that $\delta$ satisfies the following system:
\begin{eqnarray}
 && \delta_{j}+y_{j}^{\top}(\vec{X}_{i}-\vec{X}_{j}) \leq\delta_{i},\;  \forall i,j \in {\cal J},\;  i\ne j,\; \label{c1}\\
 && y_{j} \in {\cal C},\; \forall j \in {\cal J}, \label{c2} \\
 && y_{t}^{\top}\vec{X}_{t}\leq h_t(y_{t}),\; \forall t \in {\cal T}, \label{invcond} \\
 && \delta_{i}\leq\delta_{j},\; \forall(i,j)\in{\cal B},\label{precond} 
\end{eqnarray}
where 
$ h_t(y) := \min_{x} \left\{ y^\top \vec{Z}^t(x) \;\middle |\; x \in {\cal X}^t\right\}$ and ${\cal B}:= \left\{(i,j)\in \{1,2,...,{\cal J}\}^{2}\;\middle |\;(\vec{X}_{i},\vec{X}_{j})\in\{(\vec{L}_{k},\vec{U}_{k})\}_{k\in {\cal K}}\right\}$.
\end{proposition} 

One can thus resort to solving the above system \eqref{c1}--\eqref{precond} to determine more efficiently the values of the parameters $\delta_j$, $j \in {\cal J}$, so that its corresponding risk function $\rho_{\delta}$ would necessarily satisfy all the imposed conditions. It is also clear at this point that the above system indeed incorporates the information about the set ${\cal W}^{t} = \left\{\vec{Z}^t(x)\;\middle |\;x\in {\cal X}^{t}\right\}$ (i.e., in the definition of $h_t(y)$). 

As the final step, we discuss how to ensure that the risk function $\rho_{\delta}$ also minimizes the objective function $|| \rho_{\delta} - \tilde{\rho} ||_{\infty}$. It turns out that the absolute difference between $\rho_{\delta}(\vec{Z})$ and $\tilde{\rho}(\vec{Z})$ at any random loss $\vec{Z}$ can always be bounded by the difference at one of the random losses from the set $\{\vec{X}_j\}_{j\in {\cal J}}$ (which supports the risk function $\rho_{\delta}$). This is due to the piecewise linear structure embedded in the representation of $\rho_{\delta}$ (i.e., the term $\max_{j\in {\cal J}}\{p^{\top}\vec{X}_{j}-\delta_{j}\}$ \footnote{In particular, the linearity of each piece in the term has the implication that the largest difference can always be found at the support points $\vec{X}_j$, $j\in {\cal J}$.}). As a result, we need only to seek a risk function $\rho_{\delta}$ that minimizes the absolute differences over finite points (i.e., $\max_{j \in {\cal J}} |\rho_{\delta}(\vec{X}_j)-\tilde{\rho}(\vec{X}_j)|$). Combining this with the above observation that $\rho_{\delta}$ can be found by solving a finite-dimensional convex system over the parameter $\delta$, we arrive at the conclusion that to solve the inverse problem (\ref{eq:inv1}), one only needs to solve a convex program over the parameter $\delta$. This main result is presented below, and its detailed proof can be found in Appendix \ref{ssec03}. We note here that the steps we present to analyze the problem are particularly important from a methodological perspective. This for example enables us to unravel the more complicated case presented in the next section.

\begin{theorem} \label{2ndmain}
Given that Assumption \ref{asss1} holds and that the set of optimal solutions is non-empty, the inverse optimization problem (\ref{eq:inv1}) can be solved by a risk function $\rho_{\delta} \in {\cal L}(\{\vec{X}_{j}\}_{j\in {\cal J}}, {\cal C})$, where ${\cal C}$ is the support set of the reference risk function $\tilde{\rho}$ and the parameter $\delta$ is calculated by solving 
\begin{eqnarray}
\min_{\delta \in \Re^{|{\cal J}|},y_j \in \Re^{|\Omega|}} & \max_{j \in {\cal J}} |\delta_j - \tilde{\rho}(\vec{X}_j)| \label{final}\\
 {\rm subject\;to}& (\ref{c1}),\; (\ref{c2}),\; (\ref{invcond}),\; (\ref{precond}). \nonumber
\end{eqnarray}

The problem (\ref{final}) is a convex optimization problem and is polynomially solvable if 
\begin{enumerate}
\item given any $t \in {\cal T}$, the forward problem $\min_{x\in{\cal X}^{t}}\rho(\vec{Z}^t(x))$ is polynomially solvable in the case where $\rho(\vec{Z}^t(x)) := \sum_{i=1}^{|\Omega|} y_i Z^t(x,\omega_i)$ for any $y \in \Re^{|\Omega|}_+$,  and 
\item the support set ${\cal C}$ for $\tilde{\rho}$ is equipped with an oracle that can for any $p \in \Re^{|\Omega|}$ either confirm that $p \in {\cal C}$ or provide a hyperplane that separates $p$ from ${\cal C}$ in polynomial time.
\end{enumerate}
\end{theorem}

Intuitively, the above theorem summarizes a two-step procedure to identify an optimal risk function to the inverse problem (\ref{eq:inv1}). First, it computes the parameter $\delta$ by solving (\ref{final}) so that it can determine the function values $\rho_{\delta}(\vec{Z})$ over finite points $\{\vec{X}_j\}_{j\in {\cal J}}$, namely by setting $\rho_{\delta}(\vec{X}_j) = \delta_j$, $j \in {\cal J}$. Then, it interpolates (and extrapolates) other function values (i.e., $\rho_{\delta}(\vec{Z})$ for any $\vec{Z}$), based on the structure of $\rho_{\delta}$ (i.e., the definition of the set ${\cal L}(\{\vec{X}_j\}_{j\in {\cal J}}, {\cal C})$). It is guaranteed by the theorem that a risk function $\rho_{\delta}$ interpolated (and extrapolated) as such will necessarily satisfy all the imposed conditions and reach the optimal value. The main computational complexity of the procedure lies in solving the problem (\ref{final}), whose complexity, roughly speaking, is in the same order of the complexity of the forward problem. It is thus assured by the theorem that one can always efficiently learn decision makers' risk preference from their past decisions as long as the forward problems that they solved are amenable to efficient solution methods. This is the case, for example, when one tries to learn about investors' risk functions from their past investment decisions, given that many portfolio selection problems (i.e., the forward problem) can be solved efficiently. Note that the condition about the oracle is very mild, which is usually required for proving any general tractability result. Note also that the problem (\ref{final}) can be solved efficiently as a conic program (\cite{NEMIROVSKI:2007aa}), under mild regularity conditions, if the forward problem and the set ${\cal C}$ are conic representable.

To help deepen one's understanding and intuition about how the function values are interpolated (and extrapolated), we present the following two alternative formulations of $\rho_{\delta}$.
\begin{corollary} \label{linprog}
The risk function $\rho_{\delta} \in {\cal L}(\{\vec{X}_j\}_{j\in {\cal J}}, {\cal C})$ can be equivalently formulated as 
\begin{eqnarray}
 \rho_{\delta}(\vec{Z}) =  \sup_{p,s\in \Re} && p^\top \vec{Z} - s \nonumber \\
    {\rm subject\; to} && p^\top \vec{X}_j - s \leq \delta_j ,\; \forall j \in {\cal J}, \nonumber \\
                               && p \in {\cal C}. \nonumber
\end{eqnarray}
It can also be formulated as 
\begin{eqnarray}
    \rho_{\delta}(\vec{Z}) =  \inf_{t \in \Re} && t \nonumber \\
    {\rm subject\; to} && \vec{Z} - t \in {\cal A},\nonumber 
\end{eqnarray}
where $
{\cal A}:= \left\{ \vec{Z}\; \middle |\; \rho_{\delta}(\vec{Z}) \leq 0 \right\}$. In the case 
${\cal C} := \left\{ p \in \Delta \;\middle |\; \vec{R}_{\nu}^\top p \leq b_{\nu},\; \nu \in {\cal V}\right\}$ and $|{\cal V}|<\infty$, we have
\[
{\cal A} = 
\left \{ \vec{Z}\;\middle |\; \exists \theta \in \mathbb{R}^{|{\cal J}|}_{+},\; \vartheta \in \mathbb{R}^{|{\cal V}|}_+, \;
\begin{array}{cc}
 &\vec{Z} \leq \sum_{j\in {\cal J}} \theta_{j} (\vec{X}_j - \delta_j) + \sum_{\nu \in {\cal V}} \vartheta_{\nu}(\vec{R}_{\nu} - b_{\nu}), \\
  &\sum_{j\in {\cal J}} \theta_j = 1
\end{array} \right\}.
\]
\end{corollary}

The first formulation of $\rho_{\delta}$ indicates that once the parameter $\delta$ is determined, the function $\rho_{\delta}$ then evaluates any random loss $\vec{Z}$ by a linear function $p^\top \vec{Z}-s$ that bounds from below all $\delta$s (i.e., all assigned function values over $\{\vec{X}_j\}_{j \in {\cal J}}$) as tightly as possible subject to its sub-gradient $p$ bounded by ${\cal C}$. In the second formulation, the definition of the set ${\cal A}$ is well known in the risk theory (see, e.g., \cite{artzner:coherentRM}), which stands for an ``acceptance set" (i.e., the set of random losses with acceptable (non-positive) risk). In particular, here the function value $\rho_{\delta}(\vec{Z})$ can be translated into the minimum amount of cash (i.e., constant $t$) required to render the final loss $\vec{Z}-t$ acceptable (i.e., lying in the set ${\cal A}$). This sub-level set, as shown above in the second formulation of ${\cal A}$, can be viewed as a monotone convex set generated from the non-negative span of the sets $\{\vec{X}_j - \delta_j\}_{j \in {\cal J}}$ and $\{\vec{R}_{\nu}-b_{\nu}\}_{\nu \in {\cal V}}$. This sheds light on how the whole contour  $\left\{\vec{Z}\;\middle |\;\rho_{\delta}(\vec{Z}) = 0\right\}$ (which determines $\rho_{\delta}(\vec{Z})$ for any $\vec{Z}$ value) is interpolated (and extrapolated) once $\delta$ is determined: namely, having set $\rho_{\delta}(\vec{X}_j) = \delta_j$, $\forall j \in {\cal J}$ over the finite points is equivalent to fixing first the points $\{\vec{X}_j - \delta_j\}_{j \in {\cal J}}$ on the boundary of ${\cal A}$. Then, the rest of the boundary of ${\cal A}$ is interpolated by spanning a cone from $\{\vec{R}_{\nu}-b_{\nu}\}_{\nu \in {\cal V}}$ at every point in the monotone convex hull of $\{\vec{X}_j - \delta_j\}_{j \in {\cal J}}$. 

Figure \ref{update} provides an illustration, in two states, of the updates of $\rho_{\delta}$ as more information is acquired. In each of the plots, the level set $\left\{ \vec{Z}\; \middle |\; \rho_{\delta}(\vec{Z}) = 0\right\}$ is drawn over losses $\vec{Z}$ bounded between -1 and 1. Note that because of the property of translation invariance, this level set completely characterizes the risk function $\rho_{\delta}$
\footnote{That is, $\rho_{\delta}(\vec{X}) = c \Rightarrow \rho_{\delta}(\vec{X}-c) = 0$, which implies that any other level set is just the zero-level set shifted along $\vec{1}$ by $c$.}. In all three plots, we consider a decision maker optimizing a linear function $\vec{Z}^{t}(x) = Z^t x \in \mathbb{R}^2$, where $Z^t \in \mathbb{R}^{2\times 5}$, subject to budget constraints ${\cal X}:=\left\{x \in \mathbb{R}^5\;\middle |\; \vec{1}^\top x = 1,\; x \geq 0\right\}$. In the first plot, we have 
$Z^1 = \left(\begin{array}{ccccc}
0.5 & 0.225 & 0.275 & 0.5 & 0.65 \\
0.15 & 0.05 & -0.1 & -0.15& -0.1
\end{array}\right)$
and the observed decision $x^1(3)=1$ and 0 otherwise.
Here, the risk function $\rho_{\delta}$ is updated according to the largest possible $\delta_1$ that is feasible. To verify this, one can see from the dash lines the range of feasible subgradients that render the decision $x^1$ optimal. For any other convex function that assigns $Z^1x^1$ a value larger than $\delta_1$ (i.e., the level set is ``farther" from $Z^1x^1$), its subgradient at the point $Z^1x^1- \delta_1$ will fall outside the range of feasible subgradients. In the next plot, an additional observation is considered. We have 
$Z^2 = \left(\begin{array}{ccccc}
0.55 & 0.425 & 0.5 & 0.75 & 0.90 \\
-0.3 & -0.35 & -0.5 & -0.55& -0.5
\end{array}\right)$ and the observed decision $x^2(3)=1$ and 0 otherwise.
One can see that $\delta_1$, which is feasible in the first plot, is no longer feasible. This is because to make the subgradients at $Z^2x^2 - \delta_2$ feasible (i.e., rendering $x^2$ optimal) and retain the convexity of the sublevel set ${\cal A}$, we must reduce $\delta_1$. One would not be able to find any other level set that is farther from $Z^1x^1$ (and resp. $Z^2x^2$) than $\delta^1$ (and resp. $\delta^2$) and retains the convexity of ${\cal A}$. In the last plot, we assume that a ${\rm CVaR}_{\beta}$ measure with $\beta = 0.4$ is provided as the reference risk function $\tilde{\rho}$, which is supported by the set $ {\cal C}: = \left\{p\;\middle |\; \vec{1}^\top p = 1,\; 0 \leq p \leq \beta \right\}$. Now the values of $\delta_1$ and $\delta_2$ can only be feasible if the subgradients at $Z^1x^1 - \delta_1$ and $Z^2x^2 - \delta_2$ are ``confined" by ${\cal C}$. Moreover, the set ${\cal A}$ is ``extrapolated" by ${\cal C}$ also (e.g., over the points $\left\{ \vec{Z} \; \middle |\; Z(\omega_1) < 0 \right\}$).

\begin{figure}[h]
\begin{center}
\includegraphics[scale=0.45]{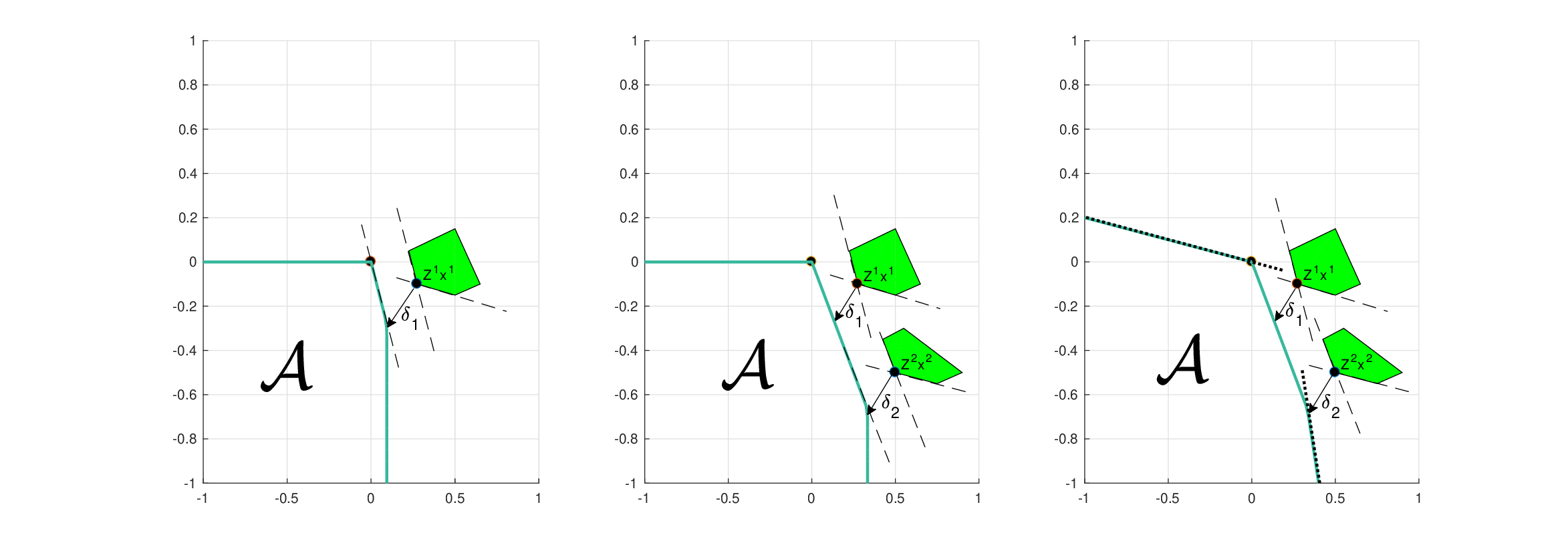}\caption{The updates of the set ${\cal A}= \left\{\vec{Z}\;\middle |\; \rho_{\delta}(\vec{Z}) \leq 0\right\}$: with one observed decision (Left), with two observed decisions (Center), and with ${\rm CVaR}_{0.4}$ additionally as the reference risk function (Right) 
}
\label{update}
\end{center}
\end{figure}

It is valuable to point out here the similarities and differences between our approach and other non-parametric approaches. In particular, \cite{Bertsimas:2014aa} also consider a non-parametric inverse problem, albeit in a very different setting from ours, and  
show that it is possible to reduce the problem by restricting the search of an optimal solution to a smaller class of functions supported by finite points only. More specifically, they show that to seek a subgradient function $y : \Re^{|\Omega|} \rightarrow \Re^{|\Omega|}$ that satisfies the optimality conditions 
\begin{equation}
y(\vec{W}^t)^\top (\vec{W} - \vec{W}^t ) \geq 0, \; \forall \vec{W} \in {\cal W}^{t},\; t \in {\cal T} ,\label{ber}
\end{equation}
it suffices to search in the set 
$\left\{ y : \Re^{|\Omega|} \rightarrow \Re^{|\Omega|} \;\middle |\; \exists \alpha_{i,t} \in \mathbb{R},\;
y_i(\vec{Z}) = \sum_{t\in {\cal T}} \alpha_{i,t} k(\vec{W}^t,\vec{Z}),\; i =1,...,|\Omega|\right\}$,
where $k : \Re^{|\Omega|} \times \Re^{|\Omega|} \rightarrow \mathbb{R}$ is a kernel function. The insight here is that such a class of functions provides sufficient flexibility to locally fit any given values $y(\vec{W}^t)$, $t \in {\cal T}$. There is however a fundamental difficulty to apply this method in our setting: namely that it is not amenable to incorporating global properties such as convexity. On the other hand, the class of functions we consider (i.e., ${\cal L}(\{\vec{X}_j\}_{j\in {\cal J}}, {\cal C})$) can both provide the flexibility for local fitting and incorporate the global properties of convex risk functions. We should note, however, that their kernel approach does not rely on convex analysis as we do in this paper and hence might be useful to handle the case of non-convexity where local and global optimality do not coincide. 

Although in a different context there is another non-parametric approach for convex interpolation, which has been successfully applied for instance in regression under the name of convex regression (\cite{Boyd:2004:CO:993483}). This approach seeks a convex function that best fits observed function values over finite points $\{\vec{X}_j\}_{j\in {\cal J}}$. The approach involves solving the constraints (\ref{c1}) also to determine the function values over finite points, and it uses a piecewise linear function to interpolate other function values. In this vein, closer to the context of this paper is the work of \cite{Delage:2015aa} mentioned earlier in the section (also \cite{armbruster:dmupii}), which in fact involves also solving the constraints (\ref{c1}) to determine function values over finite points and from there to identify the worst-case convex function. Indeed, as these works and our work are both concerned about convexity, the constraint (\ref{c1}) naturally arises. Our work in this sense can be viewed as a generalization of these works: namely that our work further addresses the comparison of function values over infinite points by expanding the system (\ref{c1}) and interpolating via a richer class of functions ${\cal L}(\{\vec{X}_j\}_{j\in {\cal J}}, {\cal C})$.

The discussions up to this point have laid enough ground work for solving other inverse problems. In particular, to solve the inverse problem (\ref{eq:inv2}), one needs only note that the property of translation invariance enables us to formulate the constraint in ${\cal R}_{inv}(\gamma)$ equivalently as 
$$ {\cal R}_{inv}(\gamma) = \left\{ \rho \;\;\middle |\;\; \rho(\vec{Z}^{t}(x^t) - \gamma_t) \leq \rho(\vec{Z}^{t}(x)),\; \forall x \in {\cal X}^{t},\; t \in {\cal T}\right\}.$$ 
Based on the same analysis, we arrive at the following. 
\begin{corollary}
Under the same assumption in Theorem \ref{2ndmain}, the inverse optimization problem (\ref{eq:inv2}) can be solved by a risk function $\rho_{\delta} \in {\cal L}(\{\vec{X}_j\}_{j\in {\cal J}}, {\cal C})$, where ${\cal C}$ is the support set of the reference risk function $\tilde{\rho}$ and the parameter $\delta$ is calculated by solving 
\begin{eqnarray}
\min_{\delta \in \Re^{|{\cal J}|},y_{j}\in \Re^{|\Omega|}, \gamma \in \Re^{|{\cal T}|}} && \sum_{t \in {\cal T}} \gamma_t \nonumber\\
{\rm subject\; to} && y_{t}^{\top}\vec{X}_{t}\leq h_t(y_{t}) + \gamma_t,\;\;\; t \in {\cal T}, \nonumber \\
                          && |\delta_j - \tilde{\rho}(\vec{X}_j)| \leq \epsilon^*,\;\;\; \forall  j \in {\cal J}, \nonumber \\
                           &&  (\ref{c1}),\; (\ref{c2}),\; (\ref{precond}). \nonumber
\end{eqnarray}
\end{corollary}

To solve the problem (\ref{eq:inv3}), we know already from Lemma \ref{lemo} that there exists a feasible function $\rho$ that bounds from above all other feasible functions. Following the proof of Proposition \ref{pro4}, one can further confirm  that there must exist $\rho_{\delta} \in {\cal L}(\{\vec{X}_j\}_{j\in {\cal J}},{\cal C})$ that bounds from above the risk function $\rho$. We can thus conclude that it 
suffices to search in the set ${\cal L}(\{\vec{X}_j\}_{j\in {\cal J}},{\cal C})$ for a worst-case risk function by maximizing all $\delta$s. This can be formulated equivalently as the following problem.

\begin{corollary}
Under the same assumption in Theorem \ref{2ndmain}, the inverse optimization problem (\ref{eq:inv3}) can be solved by a risk function $\rho_{\delta} \in {\cal L}(\{\vec{X}_j\}_{j\in {\cal J}}, {\cal C})$, where ${\cal C}$ is the support set of the reference risk function $\tilde{\rho}$ and the parameter $\delta$ is calculated by solving 
\begin{eqnarray}
\max_{\delta \in \Re^{|{\cal J}|},y_{j}\in \Re^{|\Omega|}} && \sum_{j\in {\cal J}} \delta_j \nonumber\\
 {\rm subject\; to}&& y_{t}^{\top}\vec{X}_{t}\leq h_t(y_{t}) + \gamma^*_t,\;\;\; \forall t \in {\cal T}, \nonumber \\
 && |\delta_j - \tilde{\rho}(\vec{X}_j)| \leq \epsilon^*,\;\;\; \forall  j \in {\cal J}, \nonumber \\
  && (\ref{c1}),\; (\ref{c2}),\; (\ref{precond}). \nonumber
\end{eqnarray}
\end{corollary}

\subsection{Imputing permutation-invariant risk-averse functions} \label{permrisk}
We introduce in this section the notion of permutation invariance, which will enable us to identify a subclass of risk-averse functions that can be of practical interest. The notion also provides the basis for the discussion of law invariance in Section \ref{sec33}. We say that an operator $\sigma:\Re^{|\Omega|}\rightarrow\Re^{|\Omega|}$ is a permutation operator over $\vec{Z} \in \Re^{|\Omega|}$ if it satisfies $(\sigma(\vec{Z}))_{i}=(\vec{Z})_{g^{-1}(i)}$ for any $\vec{Z}\in\Re^{|\Omega|}$, where $g:\{1,...,|\Omega|\}\rightarrow\{1,...,|\Omega|\}$ is a bijective function that permutes over $|\Omega|$ elements. We denote by $\Sigma$ the set of all permutation operators. 

\begin{definition} ({\it Permutation-invariant risk-averse functions}) Let $\bars{{\cal R}}$ denote the set of permutation-invariant risk-averse functions defined by
$$ \bars{{\cal R}}:= \left\{ \rho \in {\cal R}\;\middle |\; \rho(\vec{Z})=\rho(\sigma(\vec{Z})),\;\forall\sigma\in\Sigma,\; \forall \vec{Z} \in \mathbb{R}^{|\Omega|} \right\}.$$
\end{definition}

Here we should revisit Example \ref{exx2} in Section \ref{sec21}, which provides an important class of permutation-invariant risk-averse functions. Namely, any distributionally robust risk measure $\rho^{\uparrow}$ defined based on a phi-divergence criteria and an empirical distribution $\hat{p} = \frac{1}{M} \vec{1}$ satisfies $\rho^{\uparrow} \in \bars{{\cal R}}$. This is not hard to confirm once one recognizes that the function $D(q,\hat{p})$ satisfies $D(q,\hat{p}) = D(\sigma(q), \hat{p}),\; \forall \sigma \in \Sigma$. Indeed, although a distributionally robust risk measure provides more conservative estimates of risk, in principle it should not be sensitive to the ordering of $\vec{Z}$ given that $\rho^{\uparrow}$ are built based on samples. More generally, in any case where the decision maker is found to be insensitive to the ordering, one may consider solving the inverse problems by replacing $\rho \in {\cal R}$ with $\rho \in \bars{{\cal R}}$. Moreover, we should assume in the inverse problems that the reference risk function $\tilde{\rho}$ employed is also permutation invariant. As shown in Lemma \ref{lem2}, it is equivalent to making the following assumption.
\begin{assumption} \label{asperm}
The support set of the reference risk function $\tilde{\rho}$, now denoted by $\bars{{\cal C}}$, satisfies $p \in \bars{{\cal C}} \Leftrightarrow \sigma(p) \in \bars{{\cal C}}$. 
\end{assumption}

It is technically involved however to solve the inverse problems that take into account all the possible permutations. In particular, difficulty arises when one seeks a risk function satisfying the optimality condition that now takes the form 
$$  \rho(\sigma(\vec{W}^t)) \leq \rho(\sigma'(\vec{W})),\; \forall \vec{W} \in {\cal W}^t,\;  \forall \sigma, \sigma' \in \Sigma,\; t \in {\cal T},$$
where the set $\left\{\sigma'(\vec{W})\;\middle | \sigma' \in \Sigma,\; \vec{W}\in {\cal W}^t \right\}$ is non-convex. However, as detailed in Appendix \ref{ssec03}, Proposition \ref{pro9}, one can resolve this difficulty by following closely the steps presented in the previous section. Namely, one can apply the following definition to search over the subset $\bars{{\cal L}}(\{\vec{X}_j\}_{j\in {\cal J}}, \bars{{\cal C}})$, which necessarily contains an optimal solution to the inverse problem, and there again the conjugate duality theory comes to our rescue. 

\begin{definition}
Given a support set $\bar{{\cal C}}$ that is permutation invariant, we say that a function $\bars{\rho}_{\delta}$ is permutation invariant and supported by the pair $(\{\vec{Z}_j\}_{j\in {\bar {\cal J}}}, \bar{{\cal C}})$ if it belongs to the following set of functions
$$ \bars{{\cal L}}(\{\vec{Z}_j\}_{j\in {\bar {\cal J}}}, \bar{{\cal C}}) := 
\left\{ \bars{\rho}_{\delta}\;  \middle | \; \exists \delta \in \mathbb{\R}^{| {\bar {\cal J}}|},\; 
\bars{\rho}_{\delta}(\vec{Z}) =  \sup_{p \in \bar{{\cal C}}}  p^{\top}\vec{Z} - \max_{\sigma\in \Sigma, j\in {\bar {\cal J}}} \left\{p^{\top}\sigma(\vec{Z}_{j})-\delta_{j} \right\}  \right\}. $$
\end{definition}

The figure below provides some intuition of the above functions, where we continue the examples presented in Figure \ref{update}. In particular, we see that the sub-level set ${\cal A}$ now takes a symmetric shape.

\begin{figure}[h]
\begin{center}
\includegraphics[scale=0.45]{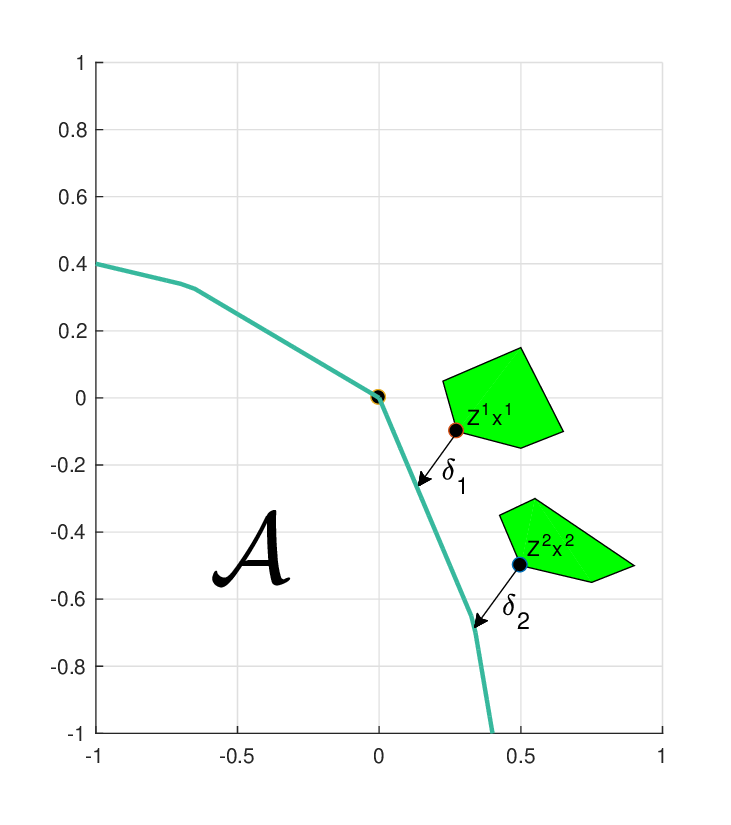}\caption{The update of the set ${\cal A}= \left\{\vec{Z}\;\middle |\; \bars{\rho}_{\delta}(\vec{Z}) \leq 0\right\}$ in Figure 1: with $\bars{\rho}_{\delta}$ now further satisfying permutation-invariance}
\label{update2}
\end{center}
\end{figure}

However, there is also the difficulty of handling the size of the inverse problems, which grows exponentially with respect to the input data of $\vec{X}_j$ because of the need to take into account all the permutations (e.g. $\sigma(\vec{X}_j)$, $\forall \sigma \in \Sigma$). We detail also in Appendix \ref{ssec03} how to reduce the problems to programs that grow only polynomially in the size of the input data (e.g., $|{\cal J}|$).

\begin{proposition} \label{prop_final}
Given that Assumption \ref{asperm} holds and that the set of optimal solutions is non-empty, the inverse problem (\ref{uniforminv})--(\ref{inv2}) with ${\cal R}:=\bars{{\cal R}}$ can be solved by a risk function $\bars{\rho}_{\delta} \in \bars{{\cal L}}(\{\vec{X}_j\}_{j \in {\cal J}}, \bars{{\cal C}})$, where $\bars{{\cal C}}$ is the support set of the reference risk function $\tilde{\rho}$ and the parameter $\delta$ is calculated by solving 
\begin{eqnarray}
\min_{\delta,y_{j},v_{i,j},w_{i,j}} && \max_{j \in {\cal J}} |\delta_j - \tilde{\rho}(\vec{X}_j)| \label{modd} \\
 {\rm subject\; to}&& \vec{1}^{\top}v_{i,j}+\vec{1}^{\top}w_{i,j}\leq\delta_{i}-\delta_{j}+y_{j}^{\top}\vec{X}_{j},\;\;\; \forall j \in {\cal J},\;\forall i\neq j, \label{mod2} \\
 && \vec{X}_{i}y_{j}^{\top}-v_{i,j}\vec{1}^{\top}-\vec{1}w_{i,j}^{\top}\leq 0,\;\;\; \forall  j \in {\cal J},\; \forall i\neq j, \nonumber \\
 && y_j \in \bars{{\cal C}}, \;\;\; \forall j \in {\cal J}, \nonumber \\
 && y_{t}^{\top}\vec{X}_{t}\leq h_t(y_{t}),\;\;\; \forall t \in {\cal T}, \label{opttt} \\
 && \delta_{i}\leq\delta_{j},\;\;\; \forall(i,j)\in{\cal B}, \nonumber
\end{eqnarray}
where $\delta \in \Re^{|{\cal J}|}$, $y_{j}\in\Re^{|\Omega|}$, $v_{i,j}\in\Re^{|\Omega|}$, $w_{i,j}\in\Re^{|\Omega|}$, the set ${\cal B}:=\left\{(i,j)\in\{1,2,...,{\cal J}\}^{2}\;\middle |\;(\vec{X}_{i},\vec{X}_{j})\in\{(\vec{L}_{k},\vec{U}_{k})\}_{k\in {\cal K}}\right\}$, 
and $h_t$ denotes the function $h_t(y):=\min_{x}\left\{y^{\top}\vec{Z}^t(x) \;\middle |\;x\in{\cal X}^{t}\right\}$. 

Moreover, the supremum representation in $\bars{{\cal L}}(\{\vec{X}_j\}_{j \in {\cal J}}, \bars{{\cal C}})$ can be reduced to 
\begin{eqnarray}
 \sup_{p \in \Re^{|\Omega|}, t \in \Re, v_j \in \Re^{|\Omega|}, w_j \in \Re^{|\Omega|}}     && p^\top \vec{Z} - t  \label{mod111}\\
{\rm subject \ to}    && \vec{1}^{\top}v_j+\vec{1}^{\top}w_j \leq t +  \delta_j,\;\;\; \forall j\in {\cal J},   \nonumber\\
                              && \vec{X_j}p^{\top}-v_j\vec{1}^{\top}-\vec{1}w_j^{\top}\leq 0,\;\;\; \forall  j\in {\cal J},  \nonumber\\
                              && p \in \bars{{\cal C}}. \nonumber
\end{eqnarray}

\end{proposition}

The fact that the size of the above programs \eqref{modd} and \eqref{mod111} grows only polynomially, rather than exponentially, with respect to the input data should provide a strong incentive to consider applying the above models. In particular, we should emphasize here that the above models offer an opportunity to learn a decision maker's risk function potentially much faster, in terms of requiring fewer observations of past decisions to reach certain learning performance, than the models presented in the previous section. Indeed, recall from the representation of the risk function $\bars{\rho}_{\delta}$ that for each random loss $\vec{X}_j$ chosen by an observed decision, the representation will automatically incorporate exponentially many more points (i.e., all the permutations of $\vec{X}_j$) and carry over the preference relations learned from $\vec{X}_j$ to all its permutations. In other words, the above models allow for incorporating ``exponentially" more preference information with the ``cost" of taking only polynomially longer time to solve the models. This should stress the importance of always checking first whether a decision maker is sensitive or not to the ordering (i.e., $\sigma(Z)$ for different $\sigma$). Following the same analysis, we can derive similar results for the inverse problems (\ref{eq:inv2})--(\ref{eq:inv3}). For brevity, we defer them to Appendix \ref{ssec02}, Corollary \ref{inv3a} and \ref{inv3b}.

\subsection{Imputing law-invariant risk-averse functions} \label{sec33}
As mentioned in Section \ref{sec21}, imposing the condition of law invariance on a risk-averse function is equivalent to considering it as a function of distributions. That is, we necessarily assume by default that the distributions of all random variables are available, as discussed also in Example \ref{exx3}, Section \ref{sec21}. For this reason, we should start by assuming that all elements involved in the inverse problems are distribution-based.

\begin{assumption} \label{ass_obj}
Each entry of the random loss  $\vec{Z}^t(x)$ admits the form of $(\vec{Z}^t(x))_i = Z(x,\xi^t (\omega_i))$, where $\xi^t : \Omega \rightarrow \Re^m$. The random vector $\xi^t$ has finite support $\left\{\xi^t_1,...,\xi^t_{\tau_0^t}\right\}$ and a probability distribution $F_{\xi^t}$ that satisfies $\mathbb{P}(\xi^t = \xi^t_o) = \bar{p}_o^{\xi^t}$ for $o=1,...,\tau_0^t$. \end{assumption}

We also make the following assumption about the reference risk function $\tilde{\rho}$.
\begin{assumption} \label{ref_law}
The reference risk function $\tilde{\rho}$ is law invariant (i.e., $\tilde{\rho} \in {\cal R}_F$).
\end{assumption}

The key to solving the inverse problems that account for distributions lies in identifying the connection between the condition of law invariance and permutation invariance discussed in the previous section. Namely, supposing that the probability measure $\mathbb{P}$ is uniform, one can observe that for any two random losses $Z_1$, $Z_2$ that share the same distribution, their vector representations
$\vec{Z}_{1}$, $\vec{Z}_{2}\in\Re^{|\Omega|}$ must satisfy $\vec{Z}_{1}=\sigma(\vec{Z}_{2})$
for some $\sigma\in\Sigma$. That is, in the case of uniform probability measure, a law-invariant risk measure must satisfy $\rho \in \bars{{\cal R}}$. To exploit this fact further, we make the following mild assumption.
\begin{assumption} \label{rational} 
All probability distributions of random losses take rational numbers as probability values.
\end{assumption}

In this case, given any discrete probability distribution $F_Z$ specified by a pair of support and probability vector $(\vec{S},\bar{p}) \in \Re^{\tau} \times \Re^{\tau}$ (i.e., 
$F_Z= \sum_{o=1}^{\tau} \bar{p}_o {\bf Dirac}((\vec{S})_o)$, where ${\bf Dirac}$ is the Dirac measure with all its weight on $(\vec{S})_o$), one can always equivalently express the probability value $\bar{p}_o$, $o=1,...,\tau$ by a ratio $n_{o}/M$, $n_o\in \{1,...,M\}$ for some $M\in\mathbb{Z}^{+}$. The random loss $Z \sim F_Z$ can thus be equivalently defined as a mapping from an outcome space $\Omega$ with $M$ uniformly distributed outcomes to $\Re$ that satisfies $Z(\omega) \in \left\{(\vec{S})_1,...,(\vec{S})_{\tau} \right\}$ and $| \left\{ \omega \in \Omega\; \middle |\;Z(\omega)=(\vec{S})_o \right\}| = \bar{p}_o M$, $ o=1,...,\tau$. However, it might be costly to implement such a procedure because the constant $M$ might need to be large and thus significantly increases the size of the optimization problems (\ref{mod2}) and (\ref{mod111}). In the following proposition, we show that the optimization problems can always be further reduced to programs whose sizes depend (almost) only on the size of the supports of distributions (i.e., $|\text{supp}(F_Z)| = \tau$), rather than the size of the outcome space (i.e., $M$). The proof is deferred to Appendix \ref{ssec03}.  

\begin{proposition} \label{lastpros}
Let $\{F_j\}_{j \in {\cal J}}$ be the distributions of the random losses in the support set $\{\vec{X}_{j}\}_{j\in {\cal J}}$, and each distribution $F_j$ be specified by a pair of $(\vec{S}_j,  \bar{p}^j) \in \Re^{\tau_j} \times \Re^{\tau_j}$ such that $F_j = \sum_{o=1}^{\tau_j} \bar{p}_o^j {\bf Dirac}((\vec{S}_j)_o)$. Given that Assumption \ref{ass_obj}, \ref{ref_law}, and \ref{rational} hold and that the set of optimal solutions is non-empty, the inverse problem (\ref{uniforminv})--(\ref{inv2}) with ${\cal R}_{F}$ can be solved by a risk function $\rho_{\delta}^F \in {\cal L}_{F}(\{F_j\}_{j \in {\cal J}}, {\cal C}_0)$, where 
$$  {\cal L}_{F}(\{F_j\}_{j \in {\cal J}}, {\cal C}_0):= 
\left\{ \rho_{\delta}^F\;  \middle | \; \exists \delta \in \mathbb{\R}^{|{\cal J}|},\; 
\rho_{\delta}^F(F_Z) =  (\ref{pro6law})  \right\}, $$
\begin{eqnarray}
     \sup_{p \in \Re^{\tau}, v_j \in \Re^{\tau_j}, w_j \in \Re^{\tau}, t \in \Re}  && 
     p^\top \vec{S} - t
    \label{pro6law}\\
     {\rm subject \ to} && \vec{1}^{\top}v_j+\vec{1}^{\top}w_j \leq t +  \delta_j,\;\;\; \forall  j \in {\cal J} , \nonumber \\
                                && \vec{S_j}p^{\top}-\Lambda_j \circ (v_j \vec{1}^{\top})-\vec{1}w_j^{\top}\leq 0, \;\;\; \forall  j \in {\cal J}, \nonumber \\
                                && p \in {\cal C}_0 \subseteq \Re^{\tau}_+, \nonumber                            
\end{eqnarray}
$\circ$ is the Hadamard product; the coefficient $\Lambda_j$ is calculated by $(\Lambda_j)_{m,n} = \bar{p}_n/\bar{p}^j_m$, $n=1,...,\tau$, $m=1,...,\tau_j$; 
and the parameter $\delta$ is calculated by solving the following optimization problem
\begin{eqnarray}
\min_{\delta,y_{j},v_{i,j},w_{i,j}} && \max_{j \in {\cal J}} |\delta_j - \tilde{\rho}(\vec{X}_j)| \nonumber \\
{\rm subject\;to} && \vec{1}^{\top}v_{i,j}+\vec{1}^{\top}w_{i,j}\leq\delta_{i}-\delta_{j}+y_{j}^{\top}\vec{S}_{j},\;\;\; \forall  j \in {\cal J},\; \forall i\neq j,\label{refff} \\
 && \vec{S}_{i}y_{j}^{\top}-\Lambda_{i,j} \circ (v_{i,j}\vec{1}^{\top})-\vec{1}w_{i,j}^{\top}\leq0 , \;\;\; \forall j \in {\cal J},\;   \forall i\neq j, \label{refff3} \\
 && y_j \in {\cal C}_j \subseteq \Re^{\tau_j}_+ ,\;\;\; \forall j \in {\cal J}, \nonumber \\
 && y_{t}^{\top}\vec{S}_t \leq h_t(y_{t}),\;\;\; \forall t \in {\cal T}, \label{refff2}\\
 && \delta_{i}\leq\delta_{j} ,\;\;\; \forall(i,j)\in{\cal B}, \nonumber
\end{eqnarray}
where $\delta \in \Re^{|{\cal J}|}$, $y_j \in \Re^{\tau_j}$, $v_{i,j} \in \Re^{\tau_i}$, $w_{i,j} \in \Re^{\tau_j}$; $\vec{S}_t:=(Z(x^{t},\xi^t_1),...,Z(x^{t},\xi^t_{\tau_0^t}))^\top$; the set ${\cal B}:= \left\{(i,j)\in\{1,2,...,{\cal J}\}^{2}\; \middle |\;(\vec{X}_{i},\vec{X}_{j})\in\{(\vec{L}_{k},\vec{U}_{k})\}_{k\in {\cal K}}\right\}$; 
and $h_t$ denotes the function $h_t(y):=\min_{x} \left\{\sum_{o=1}^{\tau_0^t} y_o Z(x,\xi^t_o) \; \middle |\;x\in{\cal X}^{t} \right\}$. The coefficient $\Lambda_{i,j}$ is calculated by $(\Lambda_{i,j})_{m,n} = \bar{p}_n^{j}/\bar{p}^i_m$, $n=1,...,\tau_j$, $m=1,...,\tau_i$.

Moreover, the above set ${\cal C}_j$, $j \in \{0\} \cup {\cal J}$ can be derived from the set $\bars{{\cal C}}$ (i.e., the support set of the reference risk function $\tilde{\rho}$ in the case where $\mathbb{P}$ is uniform) using 
\begin{equation}
 {\cal C}_j= \left\{ y \; \middle |\; {\cal H}_{F_j}((\lambda_{F_j})^{-1} \circ y) \in \bars{{\cal C}}\right\}, \label{cset}
\end{equation}
where $F_0:= F_Z$, $y \in \Re^{\tau_j}$, $\lambda_{F_j} := (\bar{p}^j_{1}|\Omega|,..., \bar{p}^j_{\tau_j}|\Omega|)^\top$ and $(\lambda_{F_j})^{-1} \circ \lambda_{F_j} = \vec{1}$, and 
${\cal H}_{F_j} : \Re^{\tau_j} \rightarrow \Re^{|\Omega|}$  stands for an operator associated with $F_j$ 
that generates a vector in $\Re^{|\Omega|}$ from a vector in the dimension of $|\text{supp}(F_j)|$. Specifically, it 
replicates each entry $\tilde{y}_o$ of a given vector $(\tilde{y}_1,...,\tilde{y}_{\tau_j})^\top \in \Re^{\tau_j}$ by $\bar{p}^j_o|\Omega|$ many times, where we denote the replications by $\vec{y}_o$, and generates a vector $(Y(\omega_1),...,Y(\omega_{|\Omega|}))$ in $\Re^{|\Omega|}$ by concatenating the replication vectors (i.e., $(\vec{y}_1^\top,...,\vec{y}_{\tau_j}^\top):= (\tilde{y}_1,...,\tilde{y}_1,\tilde{y}_2,...,\tilde{y}_2,...,\tilde{y}_{\tau_j},...,\tilde{y}_{\tau_j})$).
\end{proposition}

It is not hard to confirm that the convex programs presented above are similar to those presented in the previous section (see Proposition \ref{prop_final}). In particular, the programs here would recover the ones presented in Proposition \ref{prop_final}  if we set the size of the support to be the same across all the distributions and assign equal weight to each support. In other words, there is simply more freedom in these programs to describe distributions in terms of both support and probabilities. More importantly, this extra freedom adds no additional computational burden (in the sense that incorporating this information only affects the parameter values, but not the size, of the above convex programs). Thus, this may provide a further reason to encourage the decision maker to disclose the probability information (see Remark \ref{remm1}).
Similar results for the inverse problems (\ref{eq:inv2})--(\ref{eq:inv3}) can be found in Appendix \ref{ssec02}, Corollary \ref{inv4a} and Corollary \ref{inv4b}. 

It is the complexity of the set ${\cal C}_j$ that appears to remain dependent on the size of the outcome space $|\Omega|$. As shown in the examples below, such a dependency can often be removed by considering more specifically the exact form of the reference risk function $\tilde{\rho}$. This would thus allow the whole problem in Proposition \ref{lastpros} to be recast independently from the exact construction of the sample space. In particular, we consider below the implementation of several popular risk measures as the reference risk function $\tilde{\rho}$ and derive their corresponding set ${\cal C}_j$ using the relation \eqref{cset}. The distribution 
$F_j$ here is generally expressed by $F_j = \sum_{o=1}^{\tau_j} \bar{p}_o {\bf Dirac}((\vec{S})_o)$. The support set $\bars{{\cal C}}$ required in \eqref{cset} can be identified through the dual representation of these risk measures (see Appendix \ref{ssec01}) with the use of a uniform probability measure (i.e., $\mathbb{P}(\{\omega_i\}) =\frac{1}{M},\; \forall \omega_i \in \Omega$).
\vspace{0.1in}

\begin{example} ({\bf Risk measures and their corresponding set of support ${\cal C}_j$}) \label{spec2}
\begin{enumerate}
\item ({\bf Maximum loss}) The risk function is defined by $\tilde{\rho}(Z)=\max_{i}\{Z(\omega_{i})\}$,
and its corresponding set of support ${\cal C}_{j}$ is simply
\[
{\cal C}_j= \left\{q \in \mathbb{R}^{\tau_j}_+\;\middle |\; q^\top \vec{1}=1\right\}.
\]
\item ({\bf Expectation}) The risk function $\tilde{\rho}(Z)=\mathbb{E}[Z]$
has the following set of support: 
\[
{\cal C}_j= \left\{q \in \mathbb{R}^{\tau_j}_+ \;\middle |\; q = \bar{p}\right\}.
\]
\item ({\bf Mean-upper-semideviation}) The risk function is defined by 
\[
\tilde{\rho}(Z)=\mathbb{E}[Z]+\gamma\mathbb{E}[|Z-\mathbb{E}[Z]|],\;\gamma\in[0,\frac{1}{2}],
\]
and its set of support takes the form
\[{\cal C}_j= \left\{q \in \mathbb{R}^{\tau_j}_+ \;\middle |\; q_o= \bar{p}_o(1 + \gamma(h_o - \sum_{o=1}^{\tau_j} \bar{p}_o h_o)),\; o=1,...,\tau_j,\; \sum_{o=1}^{\tau_j} \bar{p}_o |h_o|^{t} \leq 1, \;h \geq 0\right\}.\]
\item ({\bf Conditional value-at-risk (CVaR)}) The risk function can be defined
by 
\[
\tilde{\rho}(Z)= \frac{1}{1-\alpha} \int_{\alpha}^1 F_{Z}^{-1}(t)dt, 
\]
where $F_Z^{-1}$ stands for the generalized inverse distribution function, and its corresponding set of support admits
\[
{\cal C}_j = \left\{q \in \mathbb{R}^{\tau_j}_+ \;\middle |\; q_o \leq \frac{1}{1-\alpha}\bar{p}_o, \;o=1,...,\tau_j,\; q^\top \vec{1}=1\right\}.
\]
\item ({\bf Spectral risk measures}) As a generalization of CVaR, the risk function
is defined by 
\[
\tilde{\rho}(Z)=\int_{0}^{1}F_{Z}^{-1}(t)\phi(t)dt,
\]
where the function $\phi : [0, 1] \rightarrow [0, 1]$ is non-decreasing and satisfies $\int_{0}^{1}\phi(t)dt=1$. The function $\phi$ is also known as the risk spectrum.
We should note that it is not possible to derive the reduced set ${\cal C}_j$ that is independent of the sample space $\Omega$ for general spectral risk measures because one can always seek a more ``detailed" spectrum $\phi$ by increasing the size of the sample space. Even so, for practical purposes, a ``step-wise" spectrum, $\phi^-(p):=\sum_{k=1}^{K}\bar{\phi}_{k}{\bf 1}_{(p_{k-1},p_{k}]}(p)$ for some $0<\bar{\phi}_{1}<\cdots<\bar{\phi}_{K}$ and $0=p_{0}<p_{1}<\cdots<p_{K}=1$, is usually sufficient, which can approximate any general spectrum to a pre-determined precision. By assuming that $p_k$ takes rational numbers as values, we prove at the end of Appendix \ref{ssec03} that, based on the representation $\bars{{\cal C}} = \text{Conv}(\left\{ \sigma(\bar{\phi}),\; \sigma \in \Sigma\right\})$ from Example \ref{specex}, the set of support ${\cal C}_j $ can be reduced to 
\[
 {\cal C}_j =\left\{ q \;\middle |\; q = \bar{Q}\bar{\phi},\; \bar{Q}\vec{1} = \bar{p}, \; \bar{Q}^\top \vec{1} = p^{\phi},\; \bar{Q}\geq 0\right \},
\]
where $\bar{Q} \in \Re^{\tau_j \times K}$ and $(p^{\phi})_k :=  (p_{k}-p_{k-1})$, $k=1,...,K$.
\end{enumerate}
\end{example}

\section{Solving the Inverse Problem in Decision Space} \label{sec4}
There is a distinct difference, from the complexity point of view, between the inverse problem (\ref{eq:inv4}) and the inverse problems addressed in the previous section. In particular, in this setting both the risk function $\rho$ and the decision variable $x$ are variables and the problem becomes non-convex even in the simplest case where $\rho$ is linear (i.e., $\rho(\vec{Z}) = y^\top \vec{Z}$ for some unknown $y$). The goal of this section is to discover cases where the problem can be solved in a relatively efficient way, by which we mean from a practical standpoint. In particular, we assume the following structure for the problem. Note that for simplicity we suppress the index $T$ in $\vec{Z}^T(x)$ and ${\cal X}^T$ in the rest of this section. 

\begin{assumption} \label{assi4}
The loss function $\vec{Z}(x)$ is linear; namely, it generally admits $\vec{Z}(x) = Zx$ for some $Z \in \Re^{|\Omega| \times n}$, and in the case of law invariance it has finite support $\{\xi_o^\top x\}_{o=1}^{\tau_0}$ and a probability distribution $F_{\xi}$ that satisfies $\mathbb{P}(\xi = \xi_o) = \bar{p}_o^{\xi}$ for $o=1,...,\tau_0$. In addition, the feasible set ${\cal X}$ is a non-empty, bounded full-dimensional polytope that takes the form ${\cal X}:= \left\{ x\; \middle |\; Ax \geq b \right\}$ for some $A \in \Re^{N \times n}$ and $b\in \Re^N$.
\end{assumption}

We show in the following proposition that the inverse problem (\ref{eq:inv4}) reduces to solving an inverse linear optimization problem.

\begin{proposition} \label{prodec}
In the case where Assumption \ref{asss1} and \ref{assi4} hold, the inverse problem (\ref{eq:inv4}) can be solved by a risk function $\rho_{\delta} \in {\cal L}(\{\vec{X}_{j}\}_{j\in \{1,2\}}, {\cal C})$, where ${\cal C}$ is the support set of the reference risk function $\tilde{\rho}$,  and $\vec{X}_1 = Zx^*$, $\delta_1 = {y^*}^\top Zx^*$, $\vec{X}_2 = \vec{0}$, $\delta_2 = 0$. The values of $x^*$ and $y^*$ are calculated by solving 
\begin{eqnarray}
\min_{x' \in \mathbb{R}^n ,y \in \mathbb{R}^{|\Omega|}} && || x' - x^T || \nonumber\\
{\rm subject\;to} && x' \in \arg\min_x  \left\{ y^\top Zx \;\middle |\; Ax \geq b \right\}, \label{invlin}\\
 && y \in {\cal C}. \nonumber
\end{eqnarray}
Moreover, the above problem can be solved by the following mixed-integer program (MIP):
\begin{subequations}
\begin{eqnarray}
\min_{x' \in \mathbb{R}^n ,y \in \mathbb{R}^{|\Omega|}, u \in \mathbb{R}^{N}, \eta \in \mathbb{R}^{N}} && || x' - x^T || \nonumber\\
{\rm subject\; to} && Ax' - b \geq 0, \label{dec0} \\
 && A^\top u = Z^\top y,  \label{dec1}\\
 && Ax' - b \leq M \eta, \label{dec2}\\
 && u \leq M(1-\eta), \label{dec3} \\
 && u \geq 0,\; y \in {\cal C},\;  \eta \in \{0,1\}^N, \label{dec4}
\end{eqnarray}
\end{subequations}
for a sufficiently large $M$. 
\end{proposition}

The problem (\ref{invlin}) has also been studied in the context of inverse linear optimization, and it is known that it may be solved in closed-form in the special case where $y$ is unconstrained and the rank of $Z$ equals the dimension of $x$ (\cite{ChanLeeTe2018}). Unfortunately, the problem in the general form (\ref{invlin}) is known to be difficult to solve, as the set of feasible solutions to (\ref{invlin}) is non-convex. Nevertheless, the MIP program that we present above can be very relevant from a practical standpoint. In particular, in the case where the 1-norm or $\infty$-norm is applied in the objective function, this class of programs can often be solved efficiently as mixed-integer linear programs (MILPs) on a large scale using commercial solvers such as Gurobi or Cplex. 

Next, we consider the case where $\rho$ satisfies permutation-invariance. Because we must take into account all possible permutations, the inverse problem (\ref{eq:inv4}) in this case is more involved and cannot be reduced as cleanly as (\ref{invlin}). Nevertheless, we show in the following proposition that it remains possible to solve the problem through a MIP program where the number of binary variables grows in the order of $O(|\Omega|^2)$. 
 
\begin{proposition} \label{perminv}
In the case where Assumption \ref{asperm} and \ref{assi4} hold, the inverse problem (\ref{eq:inv4}) can be solved by a risk function $\bars{\rho}_{\delta} \in \bars{{\cal L}}(\{\vec{X}_{j}\}_{j\in \{1,2\}}, \bars{{\cal C}})$, where $\bars{{\cal C}}$ is the support set of the reference risk function $\tilde{\rho}$  and $\vec{X}_1 = Zx^*$, $\delta_1 = {y^*}^\top Zx^*$, $\vec{X}_2 = \vec{0}$, and $\delta_2 = 0$. The values of $x^*$ and $y^*$ are calculated by solving the following mixed-integer program (MIP):  
\begin{subequations}
\begin{eqnarray}
\min_{x' \in \mathbb{R}^n ,y \in \mathbb{R}^{|\Omega|}, u \in \mathbb{R}^{N}, \eta \in \mathbb{R}^N,\nu_{ij}} && ||x'-x^{T}|| \nonumber\\
{\rm subject\;to\;} && (Zx')_{i}\leq(Zx')_{j}+M\nu_{ij}, \;\;\;\forall i,j \in \{1,...,|\Omega|\}, \;
i\neq j,  \label{invperm1} \\
 && (Zx')_{j}\leq(Zx')_{i}+M(1-\nu_{ij}),\;\;\; \forall i,j \in \{1,...,|\Omega|\},\;  i\neq j  , \label{invperm2} \\
 &&y_{i}\leq y_{j}+M \nu_{ij},\;\;\; \forall i,j \in \{1,...,|\Omega|\}, \;i\neq j    ,            \label{invperm}\\
 && y_{j}\leq y_{i}+M(1-\nu_{ij}),\;\;\;  \forall i,j \in \{1,...,|\Omega|\}, \; i\neq j     ,     \label{invperm4}\\
 && Ax' - b \geq 0, \nonumber \\
 && A^{\top}u=Z^{\top}y , \nonumber\\
 && Ax'-b\leq M\eta , \nonumber\\
 && u\leq M(1-\eta) , \nonumber\\
 && u\geq0,\; y\in \bars{{\cal C}}, \; \nu_{ij} \in\{0,1\},\; \eta\in\{0,1\}^N,  \nonumber 
\end{eqnarray}
\end{subequations}
for a sufficiently large $M$. 
\end{proposition}

As detailed in Appendix \ref{ssec03}, by carefully examining constraints that take into account all the permutations, the constraints boil down to requiring an ordering matching condition between $Zx'$ and subgradient $y$ (i.e., $(Zx')_{i'} \leq \cdots \leq (Zx')_{i''} \Leftrightarrow y_{i'} \leq \cdots \leq y_{i''}$). We can equivalently state this condition by the first four constraints (\ref{invperm1})--(\ref{invperm4}) using binary variables. To address the case of law invariance, we employ the same technique used in Section \ref{sec33} to first map the distribution $F_{\xi^\top x}$ to a random variable $Zx \in \mathbb{R}^{|\Omega|}$ over an outcome space $\Omega$ with a uniform probability measure. Based on this constructed random variable, we can then apply the previous proposition to formulate a MIP program. We show in the following proposition how the program can be further reduced to a program depending only on the size of the support $\{\xi_o^\top x\}_{o=1}^{\tau_0}$. 

\begin{proposition} \label{perminv2}
In the case where Assumption \ref{ref_law}, \ref{rational}, and \ref{assi4} hold, the inverse problem (\ref{eq:inv4}) can be solved by a risk function $\rho_{\delta}^{F} \in {\cal L}_{F}(\{F_{j}\}_{j\in \{1,2\}}, {\cal C}_{0})$, where 
$F_1= \sum_{o=1}^{\tau_0} \bar{p}_o^{\xi} {\bf Dirac}((\xi_o^\top x^*)_o)$, $\delta_1 = {y^*}^\top \Xi x^*$, $F_2= {\bf Dirac}(0)$, $\delta_2 = 0$, and $\Xi : = [\xi_1, \cdots, \xi_{\tau_0}]^\top$. The values of $x^*$ and $y^*$ are calculated by solving the following mixed-integer program (MIP):  
\begin{eqnarray}
\min_{x' \in \mathbb{R}^n ,y \in \mathbb{R}^{\tau_0}, u \in \mathbb{R}^{N}, \eta \in \mathbb{R}^N,\nu_{ij}} && ||x'-x^{T}|| \nonumber\\
{\rm subject\;to\;} && \xi_{i}^\top x' \leq \xi_{j}^\top x'+M\nu_{ij}, \;\;\;\forall i,j \in \{1,...,\tau_0\},\; i\neq j  , \nonumber \\
 && \xi_{j}^\top x' \leq \xi_{i}^\top x' +M(1-\nu_{ij}),\;\;\;\forall i,j \in \{1,...,\tau_0\},\; i\neq j  ,\nonumber\\
 && y_{i}\leq y_{j}+M\nu_{ij},\;\;\; \forall i,j \in \{1,...,\tau_0\},\; i\neq j ,\label{lawperm}\\
 && y_{j}\leq y_{i}+M(1-\nu_{ij}) ,\;\;\; \forall i,j \in \{1,...,\tau_0\},\; i\neq j , \label{lawperm2}\\
 && Ax'-b \geq 0, \nonumber \\
 && A^{\top}u=  \Xi^\top ( \bar{p}^{\xi} \circ y), \nonumber \\
 && Ax'-b\leq M\eta  ,\nonumber\\
 && u\leq M(1-\eta) , \nonumber\\
 && u\geq0,\; y\in\mathbb {\cal C}_1,\;  \nu_{ij} \in \{0,1\},\; \eta\in\{0,1\}^N, \nonumber 
\end{eqnarray}
for a sufficiently large $M$. 

Moreover, the set ${\cal C}_0$ can be derived from \eqref{cset} and the set ${\cal C}_1$ can be derived from ${\cal C}_1:= \left\{ y \; \middle |\; {\cal H}_{F_1}((\frac{1}{|\Omega|})y) \in \bars{{\cal C}}\right\}$,  where $\bars{{\cal C}}$ is the support set of the reference risk function $\tilde{\rho}$ in the case of uniform probability measure.
\end{proposition}

All the MIP programs presented above could be solved in seconds in our experiments conducted in the next section. It is natural to consider extending the inverse problem (\ref{eq:inv4}) more generally to incorporate multiple observations (or preference elicitation relations). Unfortunately, not only do these more general cases quickly become intractable to analyze, but it also appears unreasonable to assume that their solutions admit any particular structure as observed in the above propositions. To solve these problems directly, one faces the complexity of bilinear constraints that are highly interdependent, which is computationally intractable even on a small scale.

Although it appears not possible to solve the problems in general, we believe the true value of the above MIP programs lies in providing a reasonable means to ``extrapolate" an optimal decision from the observed sub-optimal decision. What we meant by extrapolation here is that the decision is extrapolated from the optimality condition of some risk-averse function over the set of feasible decisions. This guarantee should make the extrapolated decision a more ideal candidate to incorporate in the inverse models than the original observed decision. Hence, we recommend applying the MIP programs in either one of the following two ways when there is a need to incorporate more observations:

\begin{enumerate}
\item First, solve the MIP programs based on the latest observation of decision $x^T$. Obviously,  if the observed decision $x^T$ is already optimal, then the MIP programs will return $x^T$. Otherwise, replace $x^T$ by the solution $x^{*}$ generated from the programs and then run the inverse model (\ref{eq:inv2}) discussed in the previous section by setting the optimality condition associated with $x^{*}$ to be tight (i.e, $\gamma_T=0$). 

\item Same as above, but run the inverse model (\ref{eq:inv2}) without imposing $\gamma_T=0$. 
\end{enumerate}

Based on the results in the above propositions, in either case there must exist at least one feasible risk function (even in case 1) that renders the decision $x^*$ optimal. The difference is that the first approach essentially puts absolute priority on the fitting to the decision $x^{*}$, whereas the second approach treats the sub-optimality of each observed decision equally important to improve. Thus, if one's main interest is to fit perfectly the decision (based on most current information), one can apply the first approach. Otherwise, if the interest is to seek a potentially better decision as the input to the other inverse models, one may apply the second approach.  

\section{Numerical Study} \label{numerical}
In this section, we illustrate the use of inverse optimization on
a portfolio selection problem. We simulate the situation where a fund
manager is required to construct a portfolio that aligns with a client's
personal preference but has fairly limited opportunity to assess
the client's risk preference. We assume that the client's true risk function satisfies the conditions of monotonicity, convexity, and translation invariance and that his/her past investments were made according to the following forward risk minimization problem: 
\[
\min_{x} \left\{\rho(-\sum_{i}x_{i}\vec{R}_{i})\;\middle |\;\vec{1}^{\top}x=1,\;x\geq0\right\},
\]
where $\vec{R}_{i}\in\Re^{|\Omega|}$ denotes the random returns of
an asset $i$ and $x_{i}$ stands for the proportion of the total
wealth invested in the asset $i$. The non-negativity constraint $x\geq0$
assumes that the client considers only long positions.

\begin{remark} \label{rm4}
Although throughout this paper we have assumed the ability to specify the constraints that define the forward problem, we should mention here a few words about the possibility that the constraints may be misspecified. For instance, although the long-only constraint is common in the practice of portfolio management, at times the client may actually be open to taking some short positions without the manager's awareness. The question then arises as to how well in this case our inverse models capture the true risk function (by assuming the long-only constraint in the forward problem) and if it is possible to detect potential misspecification. As discussed with more details in Appendix \ref{apee}, we find it generally not possible to conclude, solely based on observed decisions, if there is any constraint misspecification. However, even without this confirmation, we discuss in the appendix why and how our inverse models may still be effective in capturing the true risk function.
\end{remark}

In Section \ref{51}, we demonstrate the case with limited observations of made decisions. In Section \ref{52}, we consider more generally the cases where multiple observations are available. All computations are carried out in Matlab 2014a using GUROBI 5.0 as an optimization solver. YALMIP (\cite{Lofberg:2004aa}) is used to implement our models in Matlab.

\subsection{The case of single observation} \label{51}
In this section we first consider the case of small $|{\cal T}|$, particularly $|{\cal T}|=1$. We assume in all experiments that the client's true risk preference is captured by the risk measure, optimized-certainty-equivalent (OCE), which was first introduced by \cite{Ben-Tal:2007aa}. This class of risk measures was widely referred to in both the literature of optimization (e.g., \cite{Natarajan:2010aa}) and risk theory (e.g., \cite{Drapeau:2013aa}), given its generality and its one-to-one correspondence to a (dis-)utility function.  Namely, it is defined as 
\[
\rho_{OCE}(\vec{Z}):=\inf_{t \in \Re} \left\{t+\mathbb{E}[u(\vec{Z}-t)] \right\},
\]
where the function $u: \Re \rightarrow [-\infty, \infty)$ stands for a proper, closed, convex, and nondecreasing disutility function that satisfies $u(0)=0$ and $\partial u(0) \ni 1$ and $\partial u$ denotes the subdifferential of $u$. We assume that the client's risk preference is captured by the OCE risk measure, denoted by $\rho_{OCE}^s(\vec{Z})$, with the exponential disutility function $u_s(x):=\frac{1}{s}(e^{sx}-1)$, where $s$ is a parameter that controls the level of risk aversion. This class of disutility function is fairly standard in the literature (see, e.g., \cite{Natarajan:2010aa}). 
 
We assume further that one chooses a reference risk function $\tilde{\rho}$ by following the safety-first principle but is not fully ignorant of the upside of uncertain outcomes. Specifically, the reference risk function takes the following form of spectral risk measure, where a CVaR-$90\%$ is chosen to capture downside risk and a small weight $\lambda = 0.2$ is put on the average:
\[
\rho_{Spec}(\vec{Z}):=\lambda\mathbb{E}[\vec{Z}]+(1-\lambda)\rho_{CVaR-90\%}(\vec{Z}).
\]

We ran our experiments against the dataset of daily historical
returns from 335 companies that are part of the S\&P500 index during
the period from January 1997 to November 2013. We conducted 5000 experiments, each consisting of randomly choosing a time window of
60 days and 5 stocks from the 335 companies. The first thirty days of data were used for in-sample calculation, whereas the second thirty days were for out-of-sample evaluation. 

The following steps were taken to simulate how
the manager may employ imputed risk functions. First, to simulate
the past investment, we solved the forward problem based on the OCE
risk measure $\rho_{OCE}^{s}$ with different choices of the parameter
$s$. Then, we fed the obtained portfolio $x_{OCE}^{s}$ together
with the pre-specified spectral risk measure $\rho_{Spec}$ into the
model in Proposition \ref{lastpros} to generate an imputed convex risk function
$\rho_{IC}^{s}$. Finally, we solved the forward risk minimization
problem again based on the imputed risk function $\rho_{IC}^{s}$
to obtain a portfolio $x_{IC}^{s}$. Note that in solving the forward problem, it is possible that the optimal portfolio may not be unique. Throughout our experiments, we added a regularization term $\hat{\lambda} ||x||_2$ to the objective function with a small weight $\hat{\lambda} = 10^{-6}$ so as to ensure the uniqueness of the optimal solution. Intuitively, we looked for the most diversified portfolio among the optimal portfolios, as the L2 norm is known to encourage the diversification (\cite{DeMiguel:2009aa}).

We compared both in-sample and out-of-sample performances of the portfolios $x_{OCE}^{s}$,
$x_{Spec}$, and $x_{IC}^{s}$ optimized respectively based on the
OCE risk measures $\rho_{OCE}^{s}$, the spectral risk measure $\rho_{Spec}$,
and the imputed risk functions $\rho_{IC}^{s}$. In establishing the
outcome space $\Omega$ and the associated distribution used in any
of these risk functions, we used a uniform distribution constructed based
on the first thirty days of joint returns in each sample. In comparing 
the out-of-sample performances, we additionally compute the optimal out-of-sample portfolios $\hat{x}_{OCE}^s$ and $\hat{x}_{Spec}$ optimized respectively based on the risk measure $\rho_{OCE}^s$ and $\rho_{Spec}$ using the out-of-sample data (i.e., the 
second thirty days of joint return in each sample). We benchmark the performances of the in-sample portfolios (i.e.,  $x_{OCE}^{s}$, $x_{Spec}$, and $x_{IC}^{s}$) against the optimal out-of-sample portfolios $\hat{x}_{OCE}^s$ and $\hat{x}_{Spec}$.

Table \ref{tab1} and \ref{tab2} present respectively the in-sample and out-of-sample results in terms of the averages.  In reading the tables, when an entry corresponds to the portfolio $x_{IC}^{s}$ or $x_{OCE}^{s}$ and/or the measure $\rho_{OCE}^{s}$ parameterized by $s$, the $s-$value on the top of each column is the value specifying the parameter. All values in the tables are calculated by averaging the performances over 5000 experiments.  We also provide other statistics in terms of boxplot in Figure \ref{fig:box_in} and \ref{fig:box_out}. Note that the plus sign ``+" in the boxplots refers to the average.

\begin{table}[h]
\centering
\begin{tabularx}{\textwidth}{c *{8}{Y}}
\toprule
portfolio
 & \multicolumn{4}{c}{$\boldsymbol{\rho_{OCE}^s}$ (in lost p.p. relative to $x_{OCE}^s$)}  
 & \multicolumn{4}{c}{$\boldsymbol{\rho_{Spec}}$  (in lost p.p. relative to $x_{Spec}$)}\\
\cmidrule(lr){2-5} \cmidrule(l){6-9}
  & $\boldsymbol{s=0.1}$ & $\boldsymbol{s=1}$ & $\boldsymbol{s=10}$ & $\boldsymbol{s=100}$ & $\boldsymbol{s=0.1}$ & $\boldsymbol{s=1}$ & $\boldsymbol{s=10}$ & $\boldsymbol{s=100}$\\
\midrule
 $\boldsymbol{x_{Spec}}$ & 0.33 & 0.25 & 0.16 & 0.04 & {\bf 0.00} & {\bf 0.00} & {\bf 0.00} & {\bf 0.00} \\
$\boldsymbol{x_{IC}^s}$ & 0.05 & 0.03 & 0.02 & 0.01 & 1.27 & 0.95 & 0.39 & 0.04\\
$\boldsymbol{x_{OCE}^s}$ & {\bf 0.00} & {\bf 0.00} & {\bf 0.00} & {\bf 0.00} & 1.86 & 1.26 & 0.49 & 0.08\\
\bottomrule
\end{tabularx}
\caption{Comparison of the average in-sample performances in lost percentage points (lower is less risky) of the portfolios $x_{Spec}$, $x_{IC}^{s}$, and $x_{OCE}^{s}$ with respect to different choices of the parameter $s$. Each portfolio in the left is evaluated based on the true risk measure $\rho_{OCE}^{s}$ (relative to the performances of $x_{OCE}^{s}$), whereas in the right is evaluated based on the reference risk function $\rho_{Spec}$ (relative to the performances of $x_{Spec}$).}
\label{tab1}
\end{table}

\begin{figure}[h]
\centering
\includegraphics[width=1.1\textwidth]{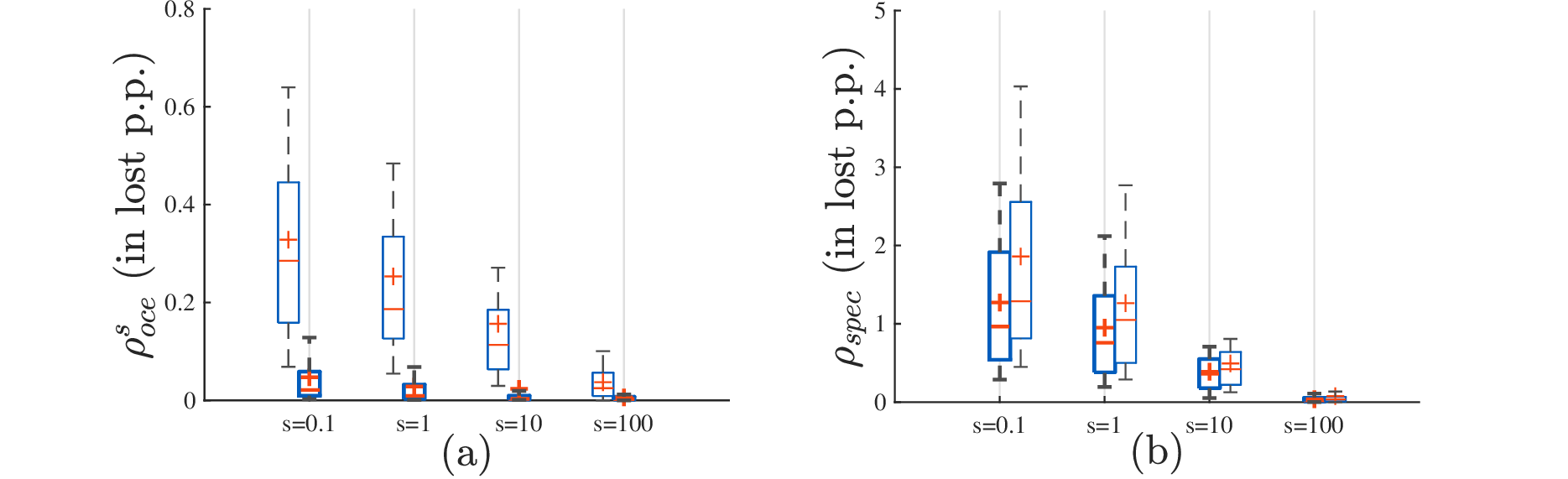}
\caption{Comparison of the in-sample performances in terms of boxplot. (a) Evaluations based 
on the true risk measure $\rho_{OCE}^{s}$ (relative to the performances of $x_{OCE}^{s}$), where each pair of boxplots consists of $x_{IC}^{s}$ (right) and $x_{Spec}$ (left). (b) Evaluations based on the reference risk function $\rho_{Spec}$ (relative to the performances of $x_{Spec}$), where $x_{IC}^{s}$ (left) and $x_{OCE}^s$ (right).}
\label{fig:box_in}
\end{figure}

It is not surprising to see in Table \ref{tab1} that, in terms of
in-sample performance, the best-performing portfolios are those optimized according to the measures used for performance
evaluation. Note that the portfolios optimized
based on the reference risk function  (i.e., $x_{Spec}$) can be deemed unsatisfactory
when evaluated according to the true risk measure $\rho_{OCE}^{s}$.
They underperform the optimal portfolios $x_{OCE}^{s}$ by an amount
up to 33 basis points (i.e., 0.33 p.p.), which can be difficult to justify in terms of their alignment with the performances desired by the client. On the other hand, the portfolios optimized
based on the imputed risk functions  (i.e., $x_{IC}^{s}$) perform more closely to the optimal portfolios $x_{OCE}^{s}$ with less than 5 basis
points' difference. Note that although by construction the imputed
risk functions $\rho_{IC}^{s}$ guarantee the optimality of the portfolios
$x_{OCE}^{s}$, minimizing $\rho_{IC}^{s}$ in the forward problem
does not necessarily lead to the same optimal solution  (i.e., $x_{IC}^{s}\neq x_{OCE}^{s}$). Even so, the benefit of incorporating the solution $x_{OCE}^{s}$
into the imputed risk function $\rho_{IC}^{s}$ is still clear when one considers the improvement of $x_{IC}^{s}$ over $x_{Spec}$ in terms of the true risk  (i.e., $\rho_{OCE}^{s}$). It is expected also from our formulation of the inverse
problem that the imputed risk function $\rho_{IC}^{s}$ should not
differ too significantly from the spectral risk measure $\rho_{Spec}$.
We can see that the results evaluated based on $\rho_{Spec}$ provide the
evidence for that  (i.e., that the portfolio $x_{IC}^{s}$ also performs more similarly to the optimal portfolio $x_{Spec}$ in this case than the portfolio $x_{OCE}^{s}$). This also confirms the effectiveness
of the imputed risk functions $\rho_{IC}^{s}$ to take into account
the information contained in the reference risk function $\rho_{Spec}$. 
Moreover, from Figure \ref{fig:box_in} we can further see that the improvements of
$x_{IC}^{s}$ over $x_{Spec}$ (in terms of $\rho_{OCE}^s$) and $x_{OCE}^s$ (in terms of $\rho_{Spec}$) are also evident across all the statistics presented in the boxplots. The observation that the performances of $x_{IC}^s$ in fact dominate the performances of the others indicates a clear gain from employing an imputed risk function.

\begin{table}[h]
\begin{tabularx}{\textwidth}{c *{8}{Y}}
\toprule
portfolio
 & \multicolumn{4}{c}{ $\boldsymbol{\rho_{OCE}^s}$ (in lost p.p. relative to $\hat{x}_{OCE}^s$)}  
 & \multicolumn{4}{c}{ $\boldsymbol{\rho_{Spec}}$   (in lost p.p. relative to $\hat{x}_{Spec}$)}\\
\cmidrule(lr){2-5} \cmidrule(l){6-9}
  & $\boldsymbol{s=0.1}$ & $\boldsymbol{s=1}$ & $\boldsymbol{s=10}$ & $\boldsymbol{s=100}$ & $\boldsymbol{s=0.1}$ & $\boldsymbol{s=1}$ & $\boldsymbol{s=10}$ & $\boldsymbol{s=100}$\\
\midrule
 $\boldsymbol{x_{Spec}}$ & {\bf 0.55} & {\bf 0.42} & {\bf 0.37} & 1.10 & {\bf 1.12} & {\bf 1.26} & {\bf 1.12} & 1.11 \\
$\boldsymbol{x_{IC}^s}$ & 0.57 & 0.47 & 0.37 & 0.99 & 2.25 & 1.79 & 1.13 & 1.02\\
$\boldsymbol{x_{OCE}^s}$ & 0.57 & 0.49 & 0.39 & {\bf 0.95} & 2.91 & 2.08 & 1.18 & {\bf 0.99}\\
\bottomrule
\end{tabularx}
\caption{Comparison of the average out-of-sample performances in lost percentage points  (lower is less risky) of the portfolios $x_{Spec}$, $x_{IC}^{s}$, and $x_{OCE}^{s}$ with respect to different choices of the parameter $s$. Each portfolio in the left is evaluated based on the true risk measure $\rho_{OCE}^{s}$ (relative to the optimal out-of-sample portfolio $\hat{x}_{OCE}^{s}$), whereas in the right is evaluated based on the reference risk function $\rho_{Spec}$ (relative to the optimal out-of-sample portfolio $\hat{x}_{Spec}$).}
\label{tab2}
\end{table}

\begin{figure}[h]
\centering
\includegraphics[width= 1.1\textwidth]{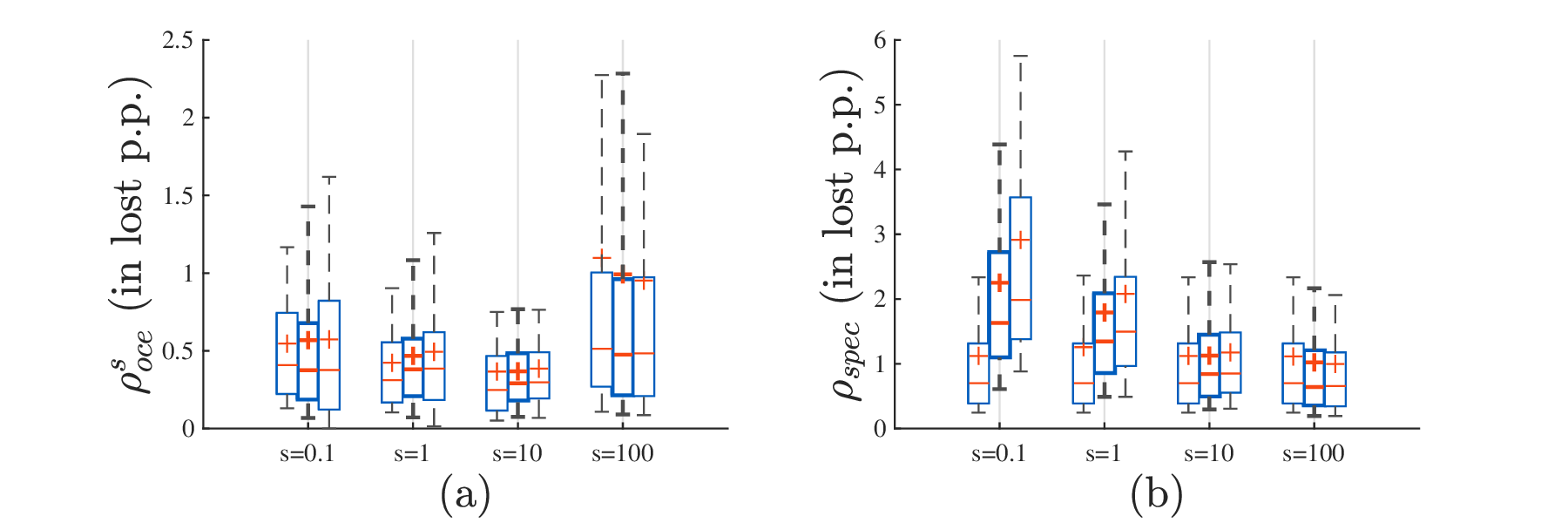}
\caption{Comparison of the out-of-sample performances in terms of boxplots, where each group of boxplots consists of $x_{IC}^{s}$ (center), $x_{Spec}$ (left), and $x_{OCE}^s$ (right). (a) Evaluations based on the true risk measure $\rho_{OCE}^{s}$ (relative to the optimal out-of-sample portfolio $\hat{x}_{OCE}^{s}$).
 (b)  Evaluations based on the reference risk function $\rho_{Spec}$ (relative to the optimal out-of-sample portfolio $\hat{x}_{Spec}$).}
\label{fig:box_out}
\end{figure}

The out-of-sample results presented in Table \ref{tab2} are calculated by subtracting the performances of the optimal out-of-sample portfolios $\hat{x}_{OCE}^s$ and $\hat{x}_{Spec}$ from the performances of the in-sample portfolios $x_{OCE}^s$, $x_{Spec}$, and $x_{IC}^s$. Like Table \ref{tab1}, the best performances  (i.e., the lowest values) are bold in Table \ref{tab2}. Perhaps quite surprisingly, in the cases of $s=0.1,1,10$, the in-sample portfolio $x_{OCE}^s$ actually underperforms the portfolio $x_{Spec}$ when evaluated according to the true risk measure $\rho_{OCE}^s$. This means that in these cases the portfolios optimized based on $\rho_{OCE}^s$ do not generalize well to out-of-sample data. It is possible to explain this by drawing the connection between the risk measure $\rho_{OCE}^s$ and expectation, namely that $\rho_{OCE}^s$ is close to expectation when $s$ is small. It is well known in portfolio optimization that portfolios optimized based on sample averages are highly unstable and suffer from poor out-of-sample performances. Indeed, we can see from Figure \ref{fig:box_out} (a) that, in the cases $s=0.1$ and $s=1$, the performances of $x_{OCE}^s$ are not only higher in terms of the average but also in terms of the spread, in comparison with that of $x_{Spec}$. We believe the reason why portfolios optimized based on $\rho_{Spec}$ appear more stable is that $\rho_{Spec}$ provides more conservative estimates of risk than $\rho_{OCE}^s$ for small $s$. This can be seen by comparing Figure \ref{fig:box_out} (a) and (b), where the values of the former are clearly smaller than the latter, except the case of $s=100$ where the two are more similar. Actually, we see in the case of 
$s = 100$, which is the most risk-averse case, the portfolio $x_{OCE}^s$ turns to outperform $x_{Spec}$ also in terms of both risk measures $\rho_{OCE}^s$ and $\rho_{Spec}$. These observations appear to align with the common belief in robust optimization that solutions optimized based on more conservative estimates are likely to enjoy more stable performances. 

Interestingly, in all cases, the performance of the in-sample portfolio $x_{IC}^s$ consistently falls between the performances of the two portfolios $x_{Spec}$ and $x_{OCE}^s$ in most of the statistics. Hence, in the cases where the portfolios optimized based on the true risk measure $\rho_{OCE}^s$ may suffer from the instability issue, the imputed risk function $\rho_{IC}^s$ can also be considered as a useful surrogate for generating more stable portfolios. That the imputed risk function $\rho_{IC}^s$ resembles only partially the true risk function $\rho_{OCE}^s$ (through an observed decision) and is more similar to the reference risk function $\rho_{Spec}$ otherwise appears to provide a mechanism to dampen the potential over-fitting issue. This also provides a further reason why one should consider choosing a more conservative risk measure as the reference risk function: not only is it practically more sensible (as mentioned in the introduction), but it also provides the basis for the bias-variance tradeoff  (i.e., biased towards more conservative estimates with the hope to reduce the variance). Note that from Figure \ref{fig:box_out} (a), we see that in the cases $s=0.1$ and $s=1$, the performance of $x_{Spec}$ is worse than that of $x_{OCE}^s$ in the lower quantiles and this can be the price paid for using more conservative risk measures  (i.e., $\rho_{Spec}$). The imputed risk function $\rho_{IC}^s$ partially corrects the bias (from the observed decision) and reaches a finer level of bias-variance tradeoff. One may consider also tuning the parameter of the reference risk function (e.g., $\lambda$ in $\rho_{Spec}$) to explore different tradeoff levels.
 
\subsection{The case of multiple observations} \label{52}
In this section, we consider incorporating multiple observations in the inverse models and focus on studying the resulting performances as $|{\cal T}| \rightarrow \infty$. We use the historical returns of the same 335 companies considered in the previous section. Following closely the experiment setup in \cite{Delage:2015aa}, here we consider weekly returns and assume that at every given time point the client applies the latest 13 weeks of joint returns to construct plausible scenarios  (i.e., $|\Omega| = 13$). In each of our experiments, we randomly draw five assets from the pool of 335 and $|{\cal T}|$ samples of their 13-week returns $R^t \in \Re^{13\times 5}$, $t=1,...,|{\cal T}|$. Based on each sample $R^t$, we solve the forward problem using the true risk function, which gives us the observations $(x^t, R^tx, {\cal X})$, $t=1,...,|{\cal T}|$, where ${\cal X}:=\left\{x\;\middle |\;\vec{1}^{\top}x=1,\;x\geq 0\right\}$. We then apply suitable inverse models, as detailed below, to generate the risk functions. To evaluate the performances resulting from the use of the imputed risk functions, we randomly select an alternative sample of 13-week returns $\hat{R}$ of the 5 assets, solve the forward problem based on the sample $\hat{R}$ and the imputed risk function, and evaluate the performance of the resulting portfolio based on the true risk function. We repeat such an experiment 1000 times for evaluating average performance.

In particular, we simulate both the case where the client has incomplete knowledge about the underlying distribution and is ambiguity-averse (like Example \ref{exx2}) and the case where the client is able to identify and reveal the distribution (like Example \ref{exx3}). Finally, we simulate the case where the client did not optimally make the decisions according to the true risk function.

\subsubsection*{The case of distributional ambiguity}
We assume in this scenario that the client's true risk preference is captured by the following form of distributionally robust risk measure 
$$  \rho_{OCE}^{s,\uparrow}(\vec{Z}) :=  \sup_{q \in \Delta} \left\{\rho_{OCE}^s (\vec{Z};q)\;\middle | \sum_{i =1}^{|\Omega|}  |q_i  -\hat{q}_i | \leq d   \right\}, $$ where variation distance is applied to measure the difference between distributions and we set $d = 0.1$. Recall that the notation $\rho_{OCE}^s (\vec{Z};q)$ refers to the use of distribution $q$ in evaluating the OCE risk measure and that $\hat{q}$ stands for the empirical distribution.

Here, the measure of variation distance is chosen because of its simplicity, which is perhaps easier to interpret than other distance measures. Of course, in our simulation, 
the manager has no knowledge about this form (neither the function $\rho_{OCE}^s$ nor the uncertainty set of $q$) and has access only to decisions made according to the function $\rho_{OCE}^{s,\uparrow}(\vec{Z})$. Confirming that the client is insensitive to the ordering of outcomes, the manager can choose to implement the inverse model (\ref{eq:inv3}) and assume that the client's true risk function is permutation invariant  (i.e., setting $\rho \in {\cal R}:=\bars{{\cal R}}$ in (\ref{eq:inv3})). The convex program presented in Corollary \ref{inv3b} in Appendix \ref{ssec02} is implemented.

Figure \ref{fig:InSample1} shows the average performances based on the true risk function $\rho_{OCE}^{s,\uparrow}(\vec{Z})$ for various risk aversion parameter $s$. We also provide the performance evaluation for the famous $(1/N)$-investment rule  (i.e., setting $x_d = (1/5)$, $d=1,...,5$), which has often been considered as a popular rule-of-thumb for dealing with ambiguity. The boxplots presented in this figure (and also Figure \ref{fig:InSample2}, \ref{fig:InSample3}, and \ref{fig:InSample4}) provide the statistics of  $ \rho(\vec{Z}_S({x_{IC}^s}(S)))- \rho(\vec{Z}_S({x_{OCE}^s}(S))) + \frac{1}{|{\cal S}|} \sum_{S \in {\cal S}} \rho(\vec{Z}_S({x_{OCE}^s}(S)))$, where $\rho$ represents the true risk function (i.e., $\rho:=\rho_{OCE}^{s,\uparrow}$ in Figure \ref{fig:InSample1}), $\vec{Z}_S(\cdot)$ is the total return estimated based on each sample of historical returns (denoted by $S \in {\cal S}$), and $x_{IC}^s(S)$ (resp. $x_{OCE}^s(S)$) stands for the portfolio optimized based on the imputed risk function (resp. the true risk function) and each sample $S$. One can see in Figure \ref{fig:InSample1} that as the number of observations $|{\cal T}|$ increases, the performance of portfolios optimized based on the imputed risk functions $\bars{\rho}_{\delta}$ converge to the performance of portfolios optimized based on the true risk function $\rho_{OCE}^{s,\uparrow}(\vec{Z})$. The rate of convergence is fast for a small number of samples  (i.e., the performance is improved most rapidly when $|{\cal T}|$ is small), and it generally takes about 10--50 samples to reach a good accuracy (i.e., with less than 0.1 p.p. difference). This is encouraging because it indicates how the risk preference can be efficiently captured even without complete knowledge of the underlying distribution. One can also see that the $(1/N)$-investment rule unfortunately provides a poor proxy of one's optimal decision and that the cost of its naive form is high. In this sense, the inverse optimization approach can already serve as a useful alternative even when, for instance, only a single observation is available $|{\cal T}|=1$.

\begin{figure}[h]
\centering
\begin{minipage}{\textwidth}
\centering
\includegraphics[width=\textwidth]{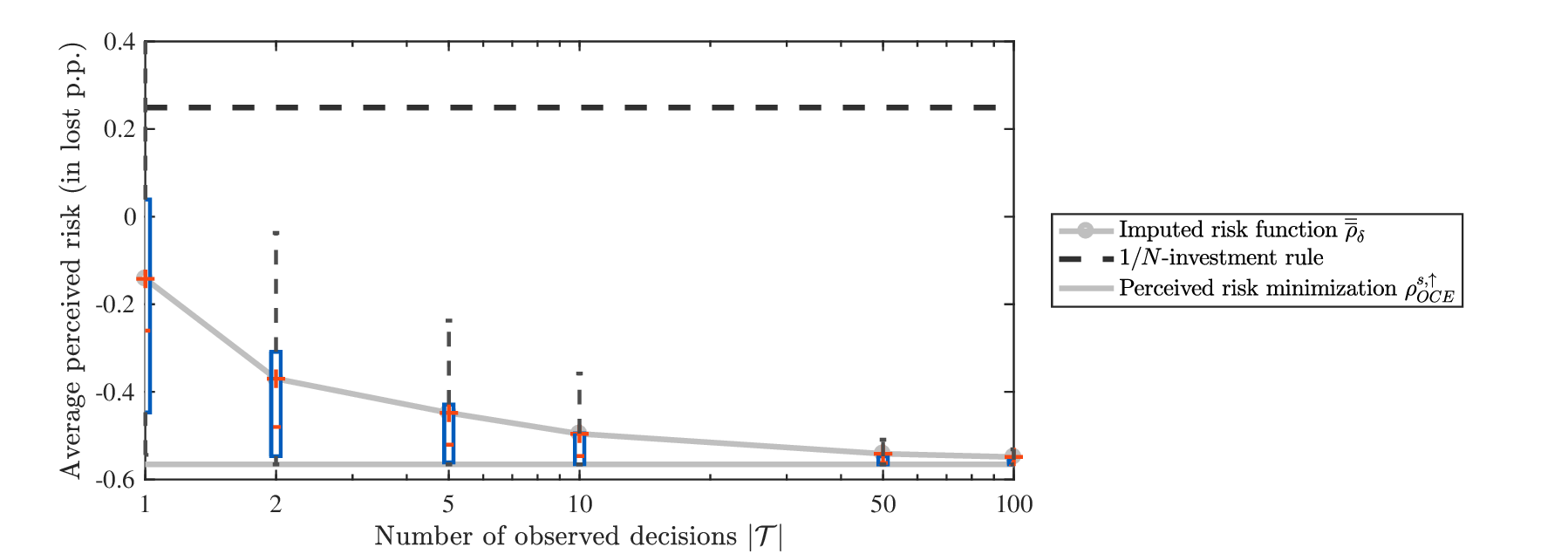}
(a)
\end{minipage}
\begin{minipage}{\textwidth}
\centering
\includegraphics[width=\textwidth]{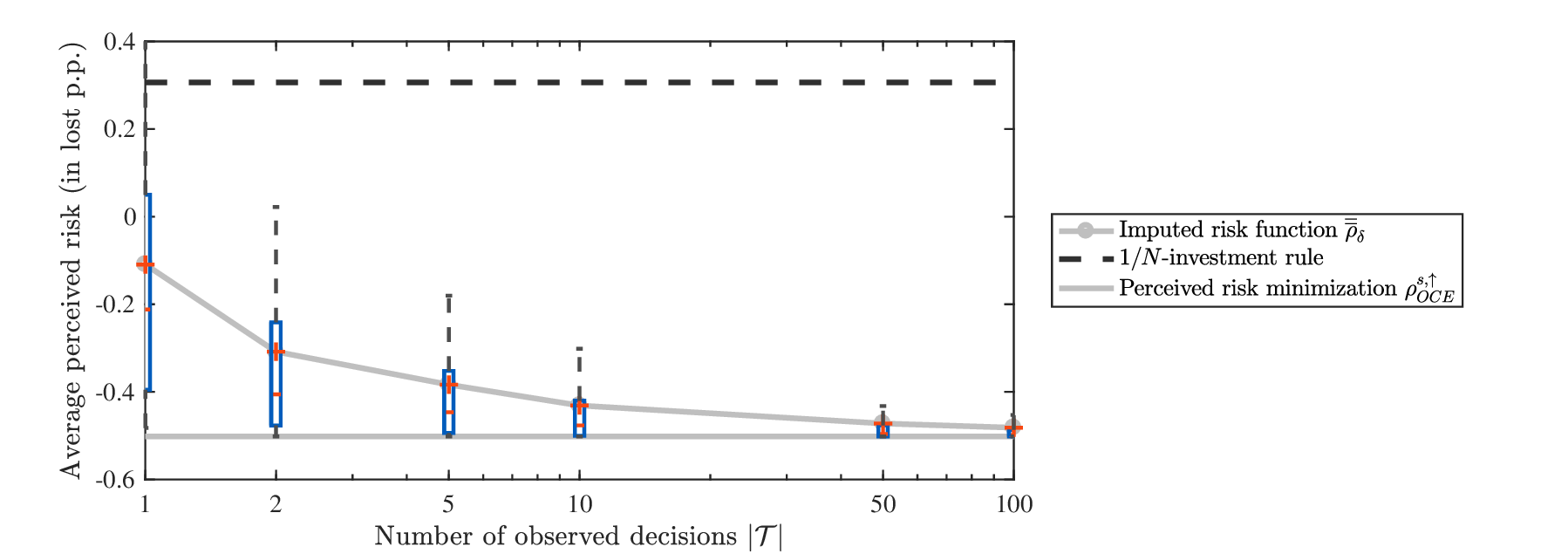}
(b)	
\end{minipage}
\begin{minipage}{\textwidth}
\centering
\includegraphics[width=\textwidth]{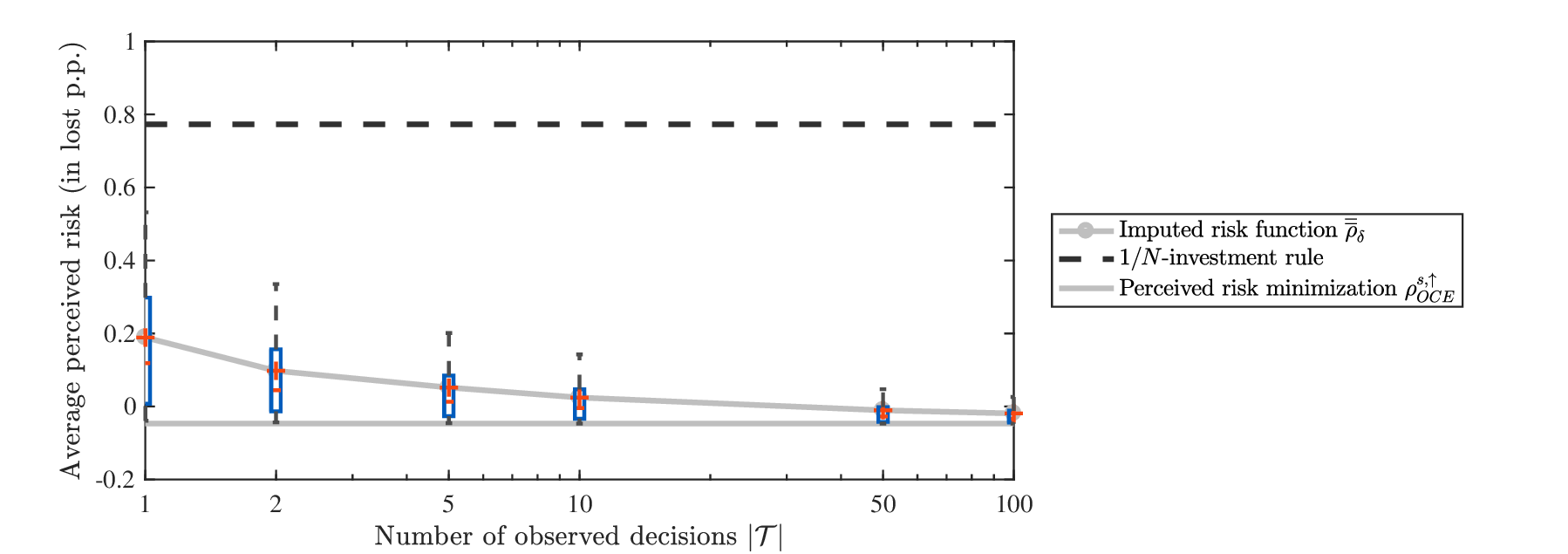}
(c)	
\end{minipage}
\caption{Comparison of the average perceived risk (in lost percentile points) for the portfolio obtained using either imputed risk function $\bars{\rho}_{\delta}$ or $1/N$-investment rule with up to 100 observed decisions. We also report the best average perceived risk that could be obtained if the representation of this perception was exactly known. In particular, the true risk function $\rho_{OCE}^{s,\uparrow}$ above is set by (a) $s=0.1$, (b) $s=1$, and (c) $s=10$.}
\label{fig:InSample1}
\end{figure}

\subsubsection*{The case of known distributions} In this scenario, we assume that the client's true risk function is the OCE risk measure and that she can identify and reveal the distribution $q$ used for evaluating the risk function  $\rho_{OCE}^s(\vec{Z};q)$. The manager in this case can choose to solve the inverse model (\ref{eq:inv3}) and assume that the client's true risk function is law invariant (i.e., setting $\rho \in {\cal R}_{F}$ in (\ref{eq:inv3})). The convex program presented in Corollary \ref{inv4b} in Appendix \ref{ssec02} is implemented. In each of our experiments, we sample uniformly from the probability simplex $\Delta$ a distribution $q$ to simulate the distribution provided by the client. 

Figure \ref{fig:InSample2} shows the average performances evaluated based on the true risk function $\rho_{OCE}^s(\vec{Z})$ for various risk aversion parameters $s$. We see that the rate of convergence here is similar to what we observed for the case of distributional ambiguity. The only minor difference we noticed is that in the case of large samples $|{\cal T}|=100$, the convergence seems to be slighter stronger in the case of known distributions  (i.e., it is more evident that the average performance gets closer to the lower bound), particularly in Figure \ref{fig:InSample2} (a) and (b). While we cannot comment with certainty, this might indicate the benefit of incorporating exact distribution information.

\begin{figure}[h]
\centering
\begin{minipage}{\textwidth}
\centering
\includegraphics[width=\textwidth]{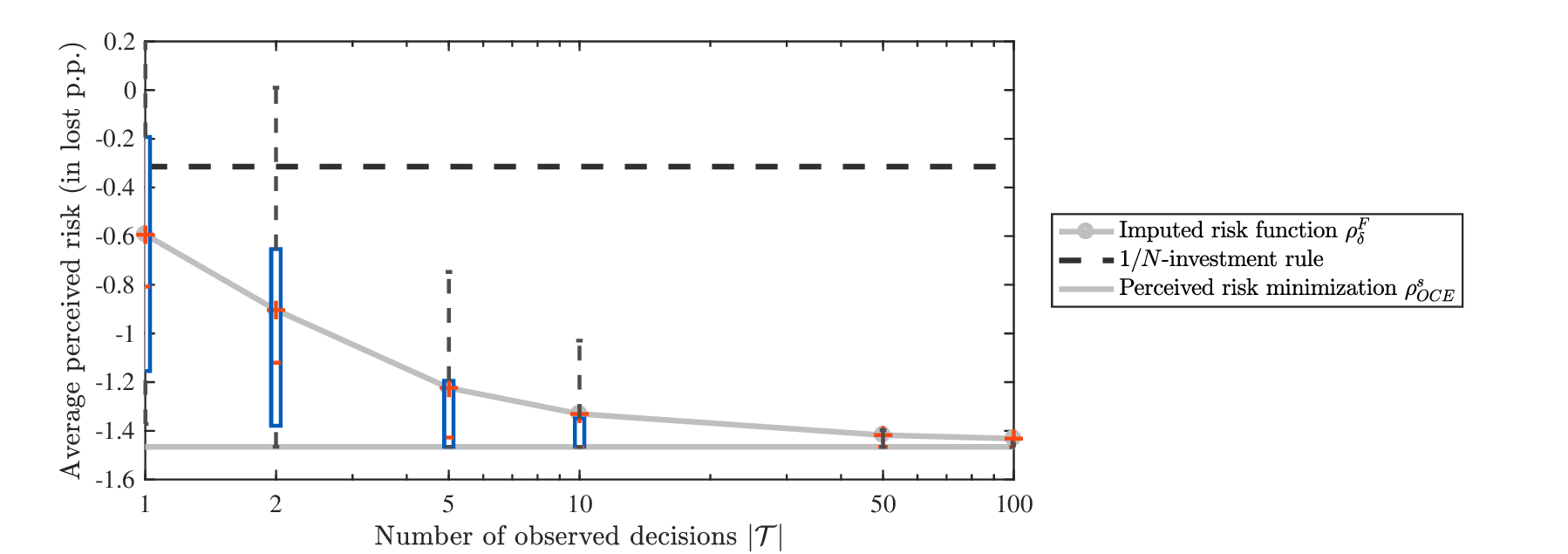}
(a)	
\end{minipage}
\begin{minipage}{\textwidth}
\centering
\includegraphics[width=\textwidth]{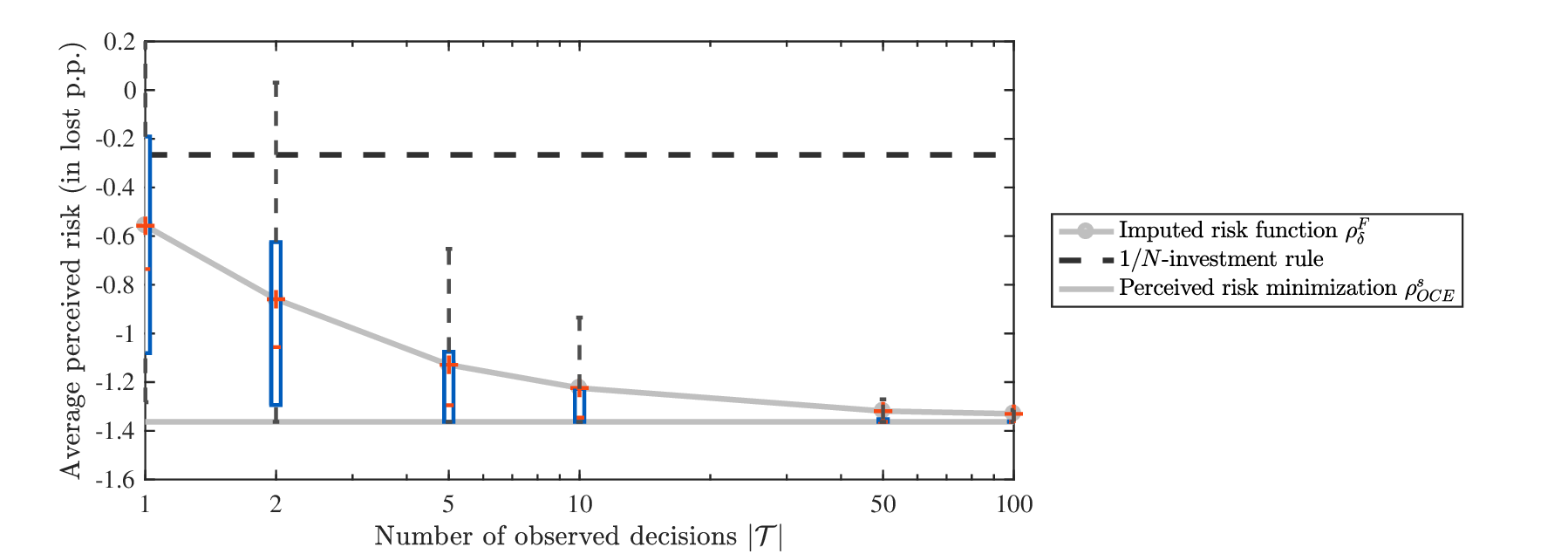}
(b)	
\end{minipage}
\begin{minipage}{\textwidth}
\centering
\includegraphics[width=\textwidth]{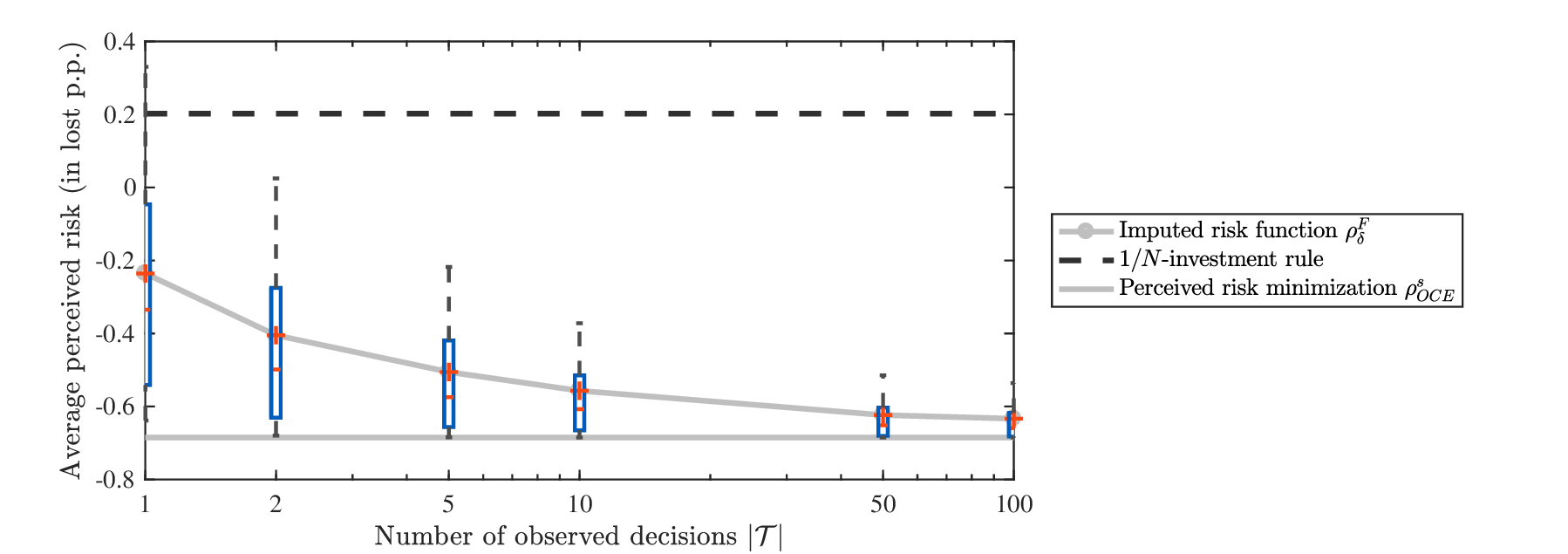}
(c)	
\end{minipage}
\caption{Comparison of the average perceived risk (in lost percentile points) for the portfolio obtained using either imputed risk function $\rho_{\delta}^{F}$ or $1/N$-investment rule with up to 100 observed decisions. We also report the best average perceived risk that could be obtained if the representation of this perception was exactly known. In particular, the above true risk function $\rho_{OCE}^{s}$ is set by (a) $s=0.1$, (b) $s=1$, and (c) $s=10$.}
\label{fig:InSample2}
\end{figure}

Although we have mentioned earlier that it is important to have the decision maker reveal the distributions in use, one may wonder what would result from some unintended inaccuracy. As a preliminary step to answer this question, we simulate here a situation where the distributions provided by the client for inputting into our inverse models might be different from the actual ones she used for solving the forward problem. As mentioned in Remark \ref{remm1}, it would not be reasonable to consider distributions that are too different from the ones used by the decision maker  (i.e., reflecting completely different views of which outcome is more or less likely to occur). Here, we assume that the distributions input into the inverse models are at least similar to those used by the decision maker in that they both assign the probability to each historical return in a non-decreasing fashion with respect to the time of the observed returns. This is an assumption commonly made in finance for example. In other words, the distribution $q$ should satisfy $q \in \Delta_e = \left\{ p \in \Delta \;\middle |\; p_1 \leq \cdots \leq p_{|\Omega|}\right\} $ where $p_i$ is now indexed according to the time of each observed return  (i.e., $p_{|\Omega|}$ is now the probability assigned to the most recent observed return). We stress that the risk function itself does not depend on the time  (i.e., it stays constant). In each of the experiments we randomly sample one distribution $q'$ from $\Delta_e$ for solving the forward problem to generate observed decisions and then sample another distribution $q''$ from $\Delta_e$ for solving the inverse model. The sampling is done by generating a vector uniformly from the probability simplex $\Delta$ and then sorting the vector. Note that in this case we apply first the inverse model (\ref{eq:inv2}) because the observed decisions now may not necessarily be optimal with respect to $q''$. Once we find the minimum gaps $\gamma$, we substitute them into the model (\ref{eq:inv3}) and solve (\ref{eq:inv3}) to obtain the parameter $\delta$. The convex program presented in Corollary \ref{inv4a} in Appendix \ref{ssec02} is additionally implemented.

Figure \ref{fig:InSample3} shows the average performance evaluated based on the true risk function $\rho_{OCE}^s(\vec{Z};q)$ with the true distributions  (i.e., $q:=q'$). Interestingly, even in this case one can see that the imputed risk functions still greatly benefit from the observed decisions. While the performance does not converge as exactly as the case of known distributions in, for instance, the case $|{\cal T}|=100$, it does exhibit strong convergence rate overall.

\begin{figure}[h]
\centering
\begin{minipage}{\textwidth}
\centering
\includegraphics[width=\textwidth]{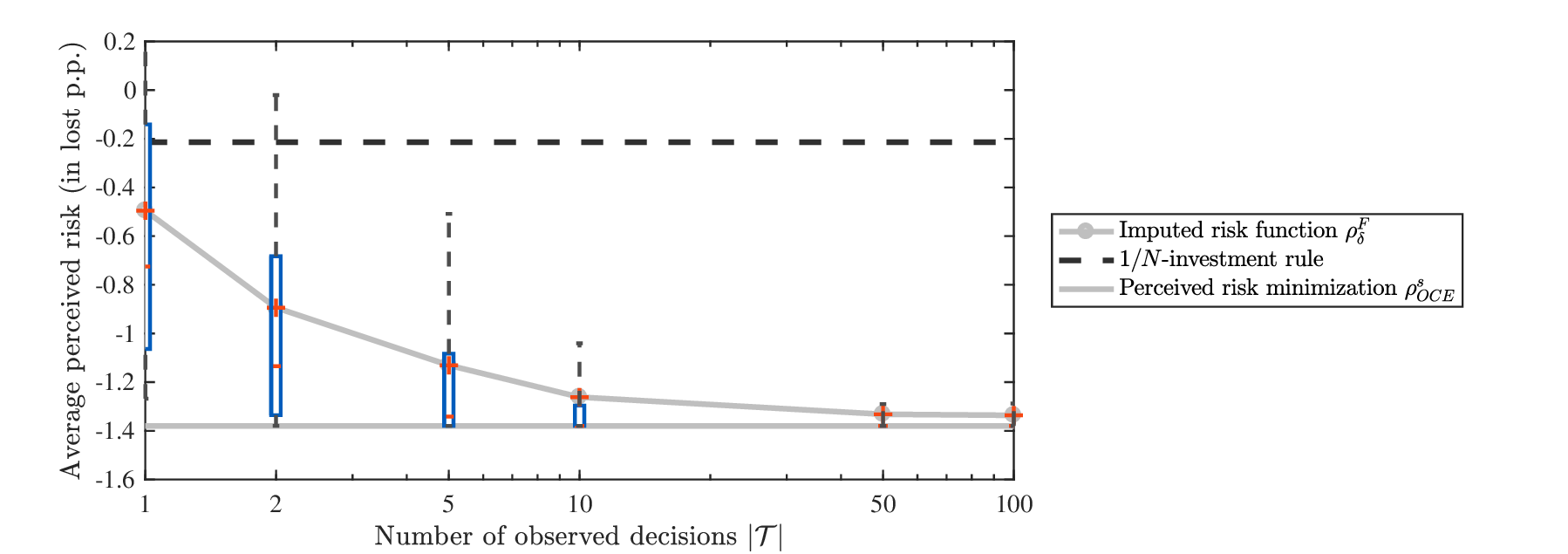}
(a)	
\end{minipage}
\begin{minipage}{\textwidth}
\centering
\includegraphics[width=\textwidth]{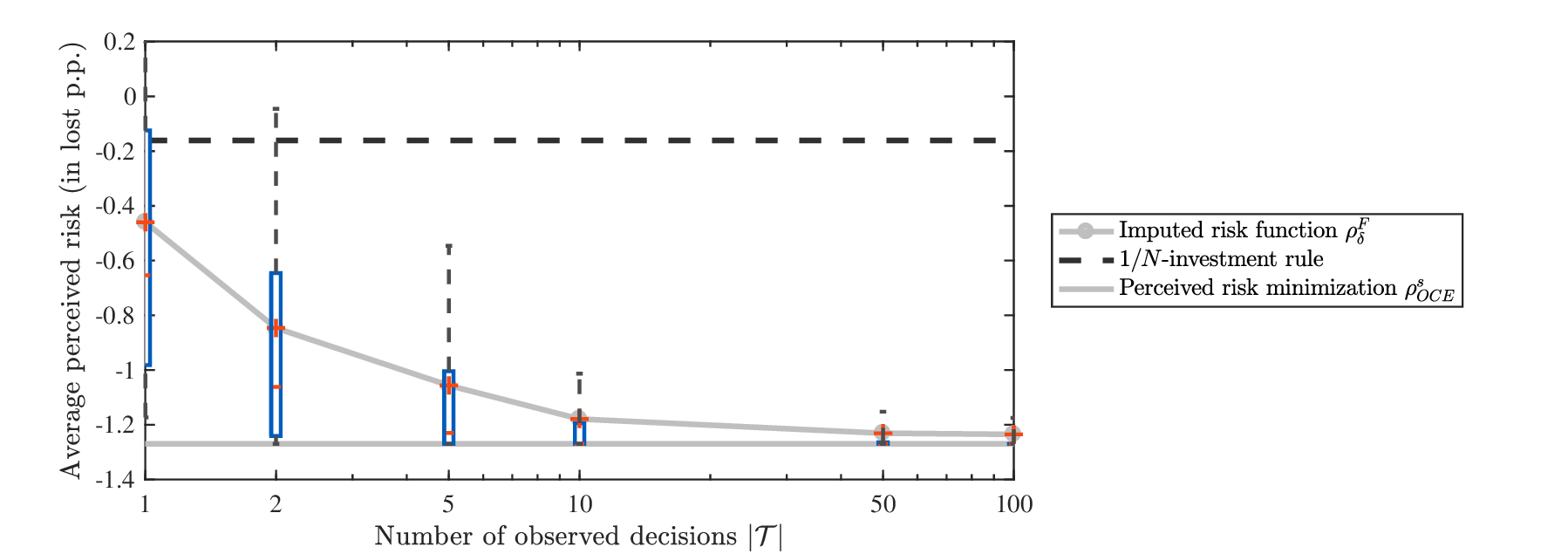}
(b)	
\end{minipage}
\begin{minipage}{\textwidth}
\centering
\includegraphics[width=\textwidth]{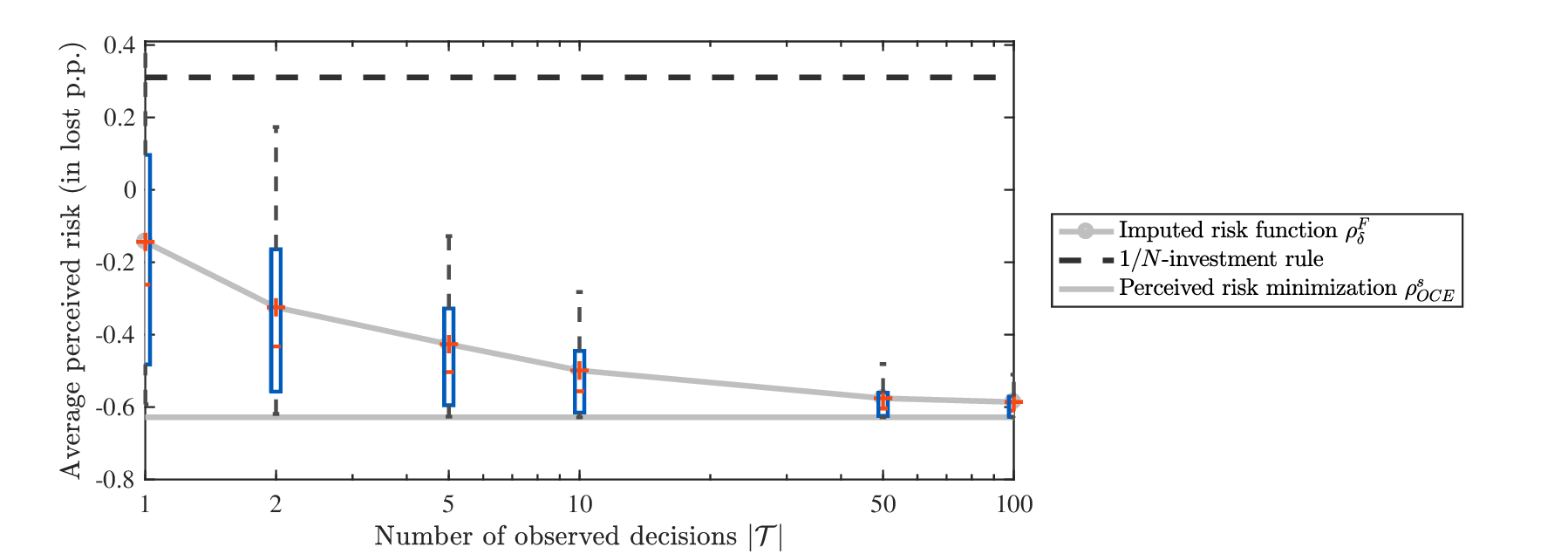}
(c)	
\end{minipage}
\caption{Comparison of the average perceived risk (in lost percentile points) for the portfolio obtained using either imputed risk function $\rho_{\delta}^{F}$ or $1/N$-investment rule with up to 100 observed decisions. We also report the best average perceived risk that could be obtained if the representation of this perception was exactly known. In particular, the above true risk function $\rho_{OCE}^{s}$ is set by (a) $s=0.1$, (b) $s=1$, and (c) $s=10$.}
\label{fig:InSample3}
\end{figure}

\subsubsection*{The case of sub-optimal observed decisions} We consider now the case where the observed decisions are sub-optimal. Here we use the same setting described in the case of distributional ambiguity as the testbed. In each of the experiments, after obtaining the observed decision $x^t$ from optimizing the true risk function  (i.e., the distributionally robust risk measure), we generate a sub-optimal decision $\hat{x}^t$ by perturbing $x^t$ as follows:
 \[\hat{x}^t = \tilde{\lambda} (\frac{1}{5})\vec{1} + (1-\tilde{\lambda}) x^t, \] 
 where $\tilde{\lambda}$ is uniformly generated from $[0, 0.1]$. That is, we take a convex combination between $x^t$ and a portfolio assigning equal weight to each asset.
 
To address the sub-optimality of decisions $\hat{x}^t$, we consider solving both the inverse model (\ref{eq:inv2}) and the inverse model (\ref{eq:inv4}). In particular, to implement the inverse model (\ref{eq:inv4}), for each sub-optimal decision we apply first the program (\ref{invperm1}) to seek an alternative optimal decision $\tilde{x}^t$ and then follow the first approach discussed in the end of Section \ref{sec4} to incorporate multiple observations. 
 
Figure \ref{fig:InSample4} presents the results of average performances for the portfolios generated from the two approaches  (i.e., (\ref{eq:inv2}) versus first solving (\ref{eq:inv4}) then solving (\ref{eq:inv2})). There appears no noticeable difference between the performances of the portfolios generated from the two approaches. However, in Figure \ref{fig:boxplot} one can see that the optimal portfolios $\tilde{x}^t$ generated based on the model (\ref{eq:inv4}) are far closer to the observed sub-optimal decisions $\hat{x}^t$ than the optimal portfolios generated based on the model (\ref{eq:inv2}) only. Hence, if the client has a concern about what portfolio $x$ specifically the model would suggest to invest and whether it is aligned with her past investment decision, the model (\ref{eq:inv4}) might be easier to justify for that purpose.

\begin{figure}[h]
\centering
\begin{minipage}{\textwidth}
\centering
\includegraphics[width=\textwidth]{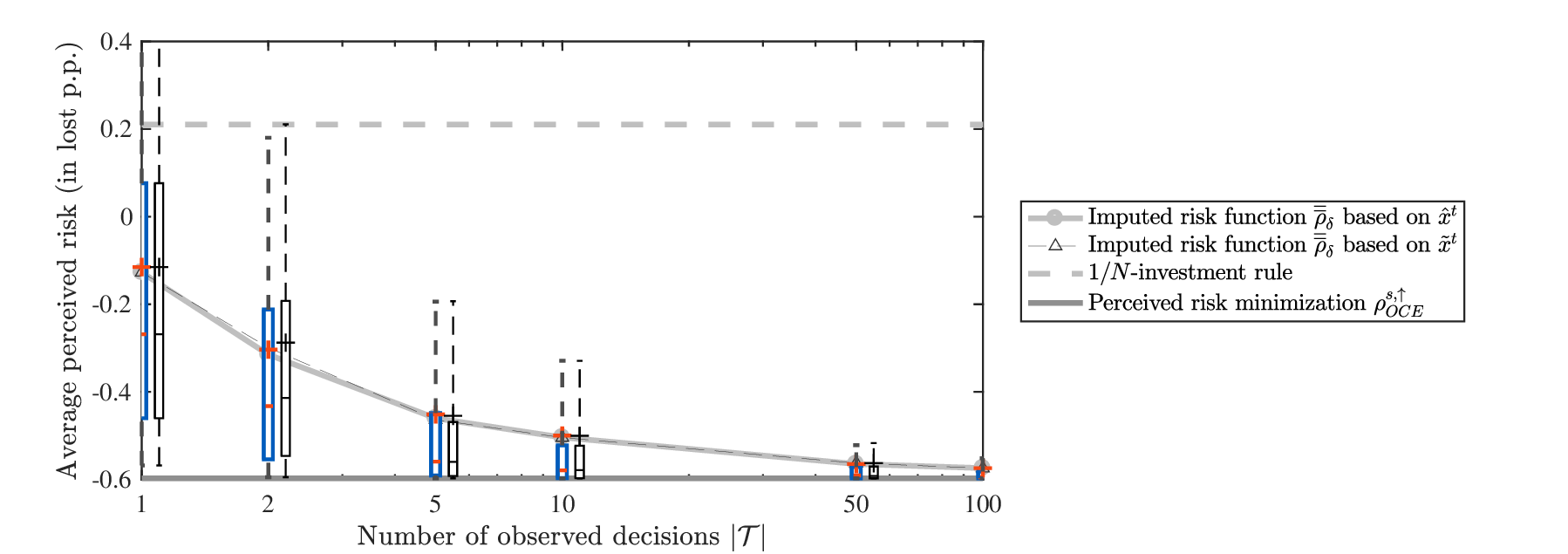}
(a)	
\end{minipage}
\begin{minipage}{\textwidth}
\centering
\includegraphics[width=\textwidth]{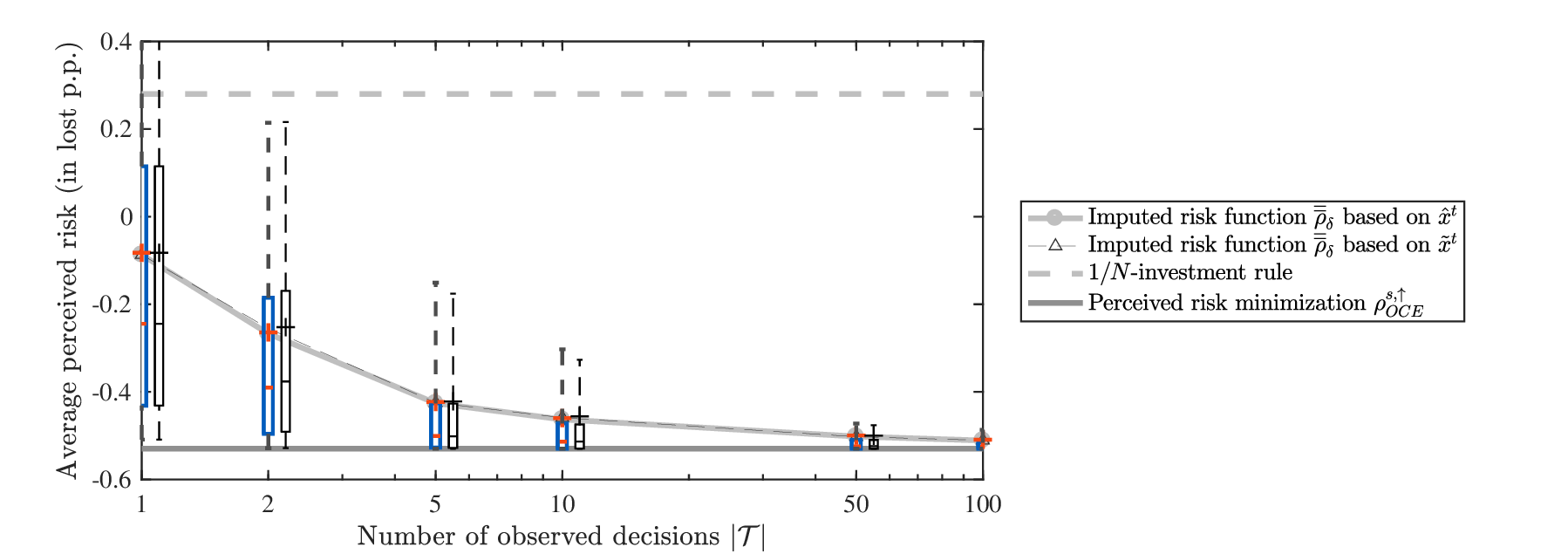}
(b)	
\end{minipage}
\begin{minipage}{\textwidth}
\centering
\includegraphics[width=\textwidth]{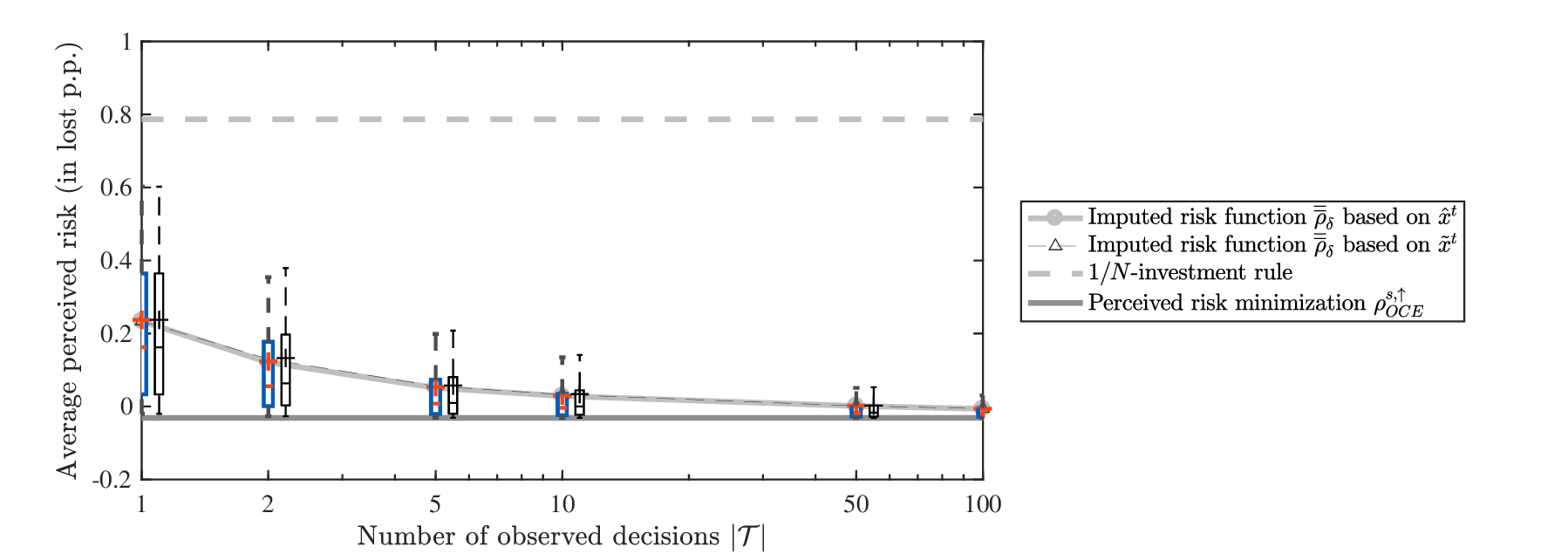}
(c)	
\end{minipage}
\caption{Comparison of the average perceived risk (in lost percentile points) for the portfolio obtained using either imputed risk function $\bars{\rho}_{\delta}$ or $1/N$-investment rule with up to 100 observed decisions. We also report the best average perceived risk that could be obtained if the representation of this perception was exactly known. In particular, the true risk function above $\rho_{OCE}^{s,\uparrow}$ is set by (a) $s=0.1$, (b) $s=1$, and (c) $s=10$. Each pair of boxplots consists of a boxplot based on $\hat{x}^t$ (left) and based on $\tilde{x}^t$ (right).}
\label{fig:InSample4}
\end{figure}

\subsubsection*{Computation time} 
We present in Table \ref{tab4} the computation time taken in an experiment for the 
case of distributional ambiguity. We ran the inverse model (\ref{eq:inv3}) for 
different numbers of assets and outcomes. As seen, the time is not sensitive to the number of assets but grows more noticeably in the number of outcomes. The rate of growth in general is fairly consistent with how the size of the convex program like (\ref{mod2}) grows in the number of observed decisions and outcomes. Namely, the program has $O(|\Omega| |{\cal T}|^{2})$ decision variables and $O(|\Omega|^{2}|{\cal T}|^{2})$ constraints, given that the set $\{\vec{X}_j\}_{j \in {\cal J}}$ consists of only random losses selected by observed decisions $x^t$, $t\in {\cal T}$ (and zero vector $\vec{0}$).

\begin{table}[h]
\centering
\begin{tabular}{clllll} 
\toprule
Number of observed decisions & 1    & 5    & 10   & 50     & 100      \\ 
\hline
Number of assets             &      &      &      &        &          \\
5                            & 0.01 & 0.07 & 0.61 & 22.96  & 159.00   \\
100                          & 0.01 & 0.11 & 0.62 & 22.05  & 160.00   \\
300                          & 0.01 & 0.08 & 0.52 & 28.12  & 165.00   \\ 
\hline
                             &      &      &      &        &          \\ 
\hline
Number of observed decisions & 1    & 5    & 10   & 50     & 100      \\ 
\hline
Number of outcomes           &      &      &      &        &          \\
13                           & 0.01 & 0.07 & 0.61 & 22.96  & 159.00   \\
26                           & 0.02 & 0.63 & 4.28 & 152.14 & 863.97   \\
39                           & 0.07 & 1.06 & 5.69 & 373.91 & 8050.00  \\
\bottomrule
\end{tabular}
\caption{Computation time in seconds: with fixed number of outcomes $|\Omega|$=13 (Top), and with fixed number of assets, 5 assets (Bottom).}
\label{tab4}
\end{table}

\begin{figure}[h]
\centering
\includegraphics[width=\textwidth]{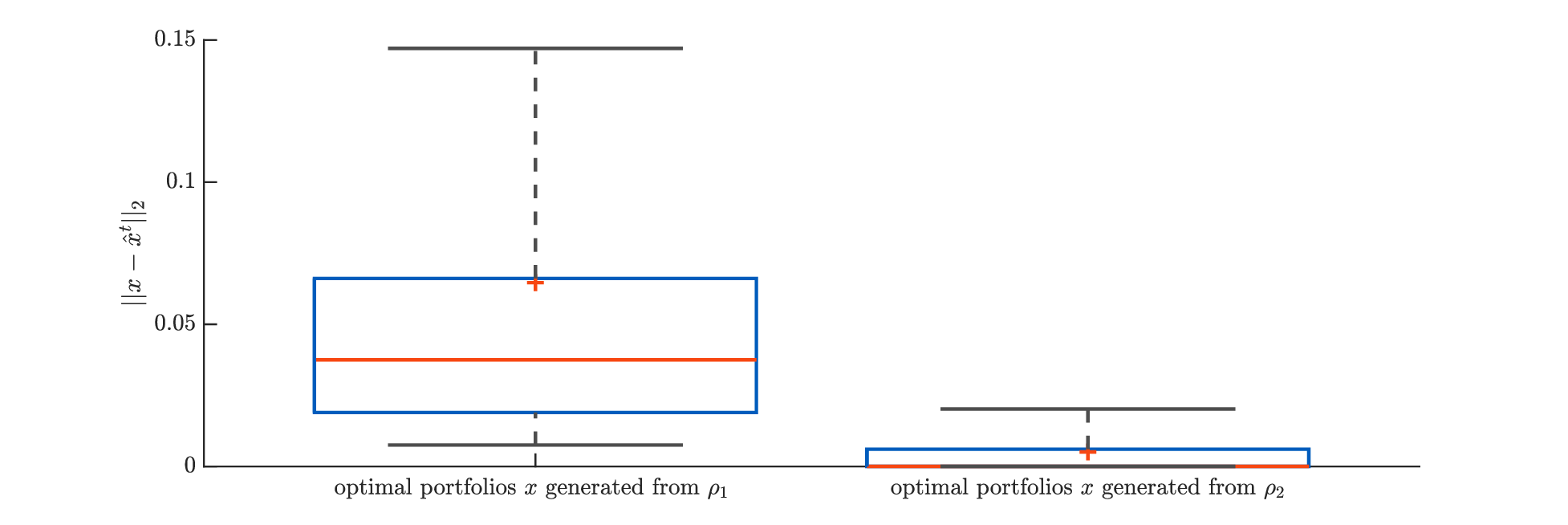}
\caption{The L2-norm distances between the optimal portfolios generated from imputed risk functions and the observed portfolios $\hat{x}^t$ out of all experiments with sub-optimal observed decisions: $\rho_1$ denotes risk functions generated from the inverse model (\ref{eq:inv2}) (left), and $\rho_2$ denotes risk functions generated from the inverse model (\ref{eq:inv4}) (right).} 
\label{fig:boxplot}
\end{figure}

\subsubsection*{Out-of-Sample Performance}
Finally, we conducted out-of-sample testing based also on the setting described in the case of distributional ambiguity. In particular, in each of the experiments we further collected the realized return of each portfolio in the week that follows right after the weeks used for portfolio optimization. Then, we calculated for each portfolio the realized risk  (i.e., calculating the distributionally robust risk measure $\rho_{OCE}^{s,\uparrow}$ based on the collected samples). The results can be found in Figure \ref{fig:InSample5}. 
We observe that the sign of convergence remains strong in these results.  
This provides further evidence of how the imputed risk function closely resembles the true risk function in general.  Another observation is that the convergence of the out-of-sample performance in general is not as ``regular" as the convergence of the in-sample performance. In particular, in the case of high risk-aversion  (i.e., $s=10$), the realized risk of the portfolios generated based on the imputed risk function already reach a similar level as that of the portfolios generated based on the true risk function, with only a small number of observations. After reaching that level, the performance then slightly goes up and down as more observations are acquired. We believe that this is related to the fact that the inverse model used in the experiments for the case of distributional ambiguity is the one that seeks the worst-case risk function. As the worst-case function may encourage a more conservative choice of a portfolio, it thus might not be that surprising why such a portfolio already performs well, even with a small number of observed decisions, when evaluated based on a highly risk-averse OCE risk measure.

\begin{figure}[h]
\centering
\begin{minipage}{\textwidth}
\centering
\includegraphics[width=\textwidth]{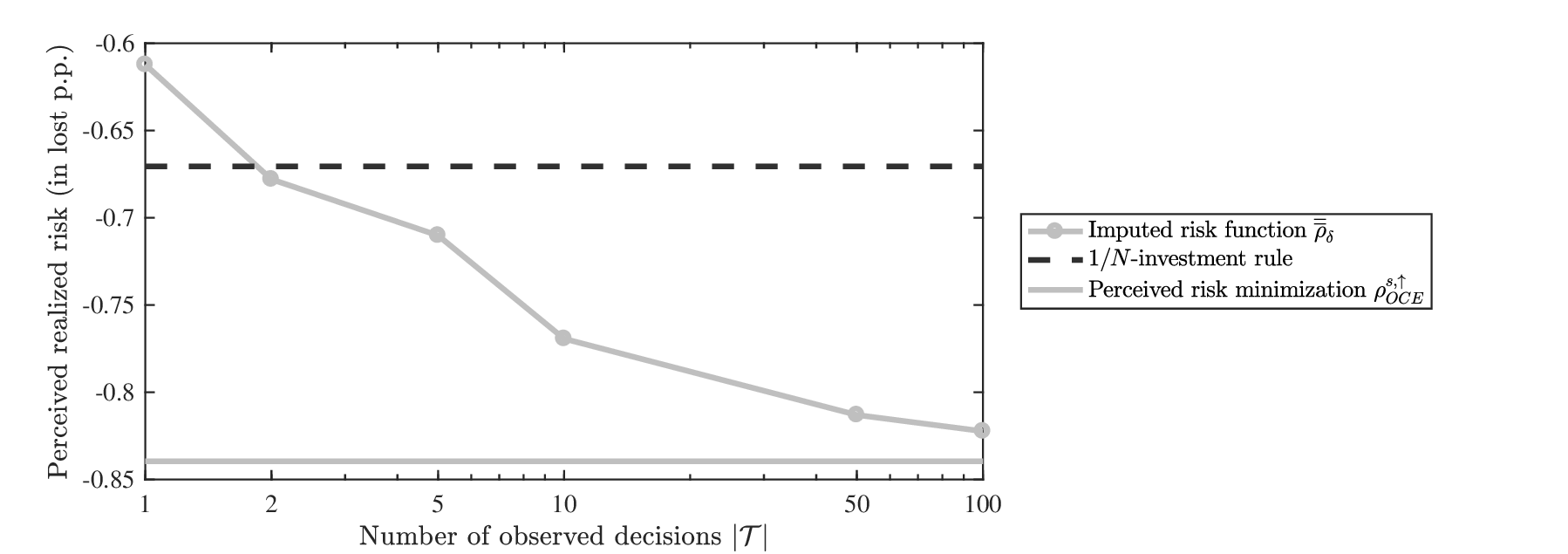}
(a)	
\end{minipage}
\begin{minipage}{\textwidth}
\centering
\includegraphics[width=\textwidth]{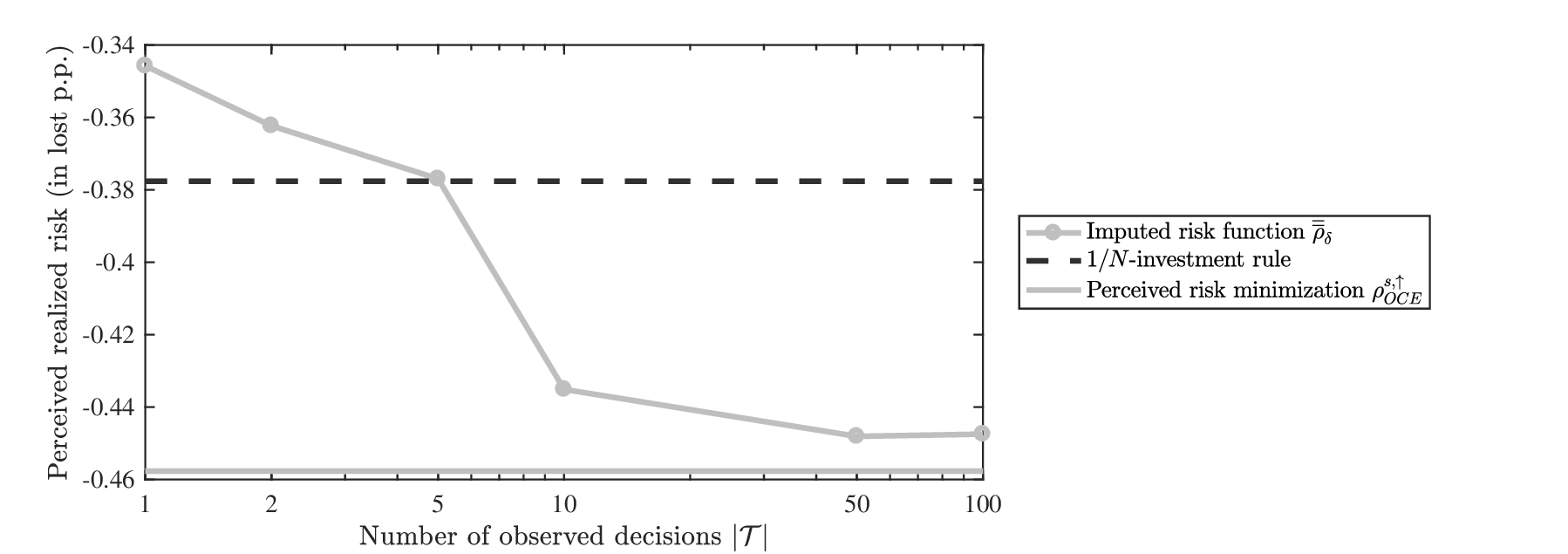}
(b)	
\end{minipage}
\begin{minipage}{\textwidth}
\centering
\includegraphics[width=\textwidth]{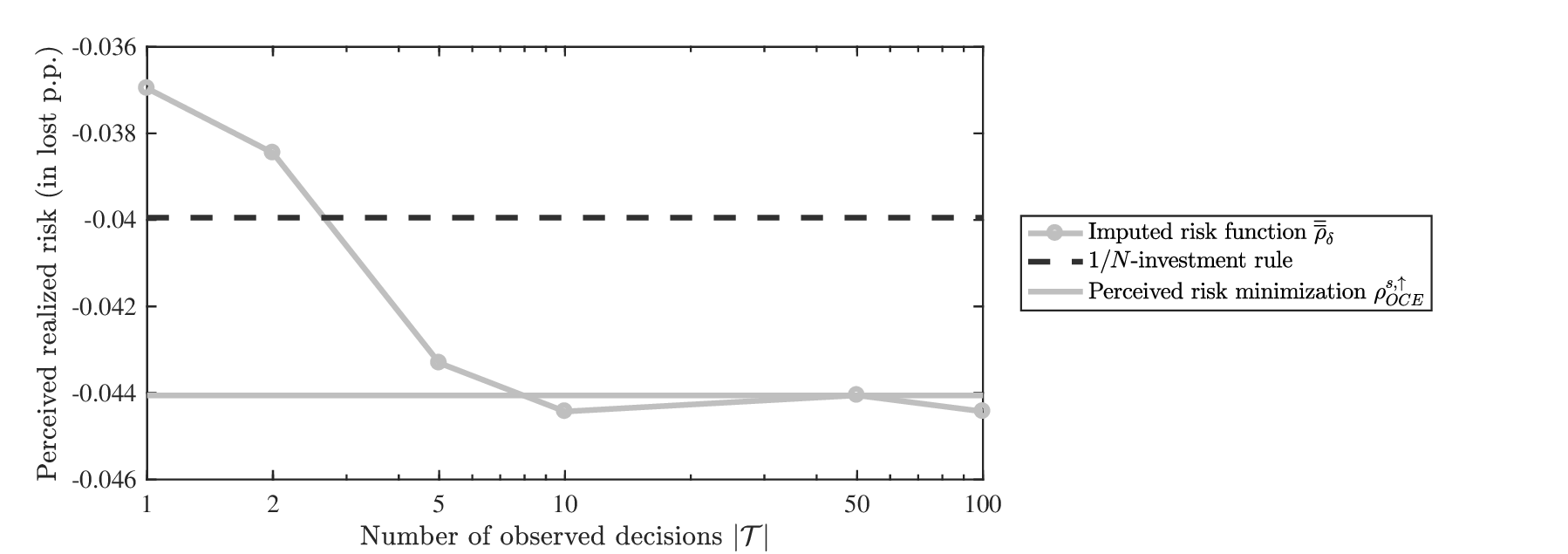}
(c)	
\end{minipage}
\caption{Comparison of the perceived realized risk (in lost percentile points) for the portfolio obtained using either imputed risk function $\bars{\rho}_{\delta}$ or $1/N$-investment rule with up to 100 observed decisions. We also report the best average realized risk that could be obtained if the representation of this perception was exactly known. In particular, the above true risk function $\rho_{OCE}^{s,\uparrow}$ is set by (a) $s=0.1$, (b) $s=1$, and (c) $s=10$.}
\label{fig:InSample5}
\end{figure}

\section{Conclusions}
We proposed in this paper a non-parametric inverse optimization framework for risk averse optimization problems involving convex risk functions. The framework allows one to impute a risk function based on preference information acquired from multiple sources, including preference elicitation, observed decisions made for a forward problem, and a reference risk function. The framework can be flexibly applied to generate a risk function based on different criteria. We showed that in general the solution can be efficiently identified by solving convex programs and demonstrated in numerical experiments that the imputed risk function that incorporates the information of past decisions could indeed generate risk estimates that are significantly closer to the true risk level. In addition to the portfolio selection application considered in this paper, other applications such as homeland security (\cite{Haskell:2018aa}) could also benefit from the development of our framework. We leave the study of these other applications as our future work.

\ACKNOWLEDGMENT{
The author gratefully acknowledges that this research has been supported by the Canadian Natural Sciences and Engineering Research Council under [Grant RGPIN-2014-05602]. The author is indebted to the support of Li Lily Liu in preparing this manuscript. The author thanks the associate editor and three reviewers for their detailed comments and suggestions that greatly improved the quality of the paper. 
}


\bibliographystyle{ormsv080}
\bibliography{refs2}

\ECSwitch
\ECHead{Additional Material}

\counterwithin{theorem}{section}
\counterwithin{lemma}{section}
\counterwithin{example}{section}
\counterwithin{figure}{section}
\section{Results about conjugate duality theory and the function $\rho_{\delta} \in {\cal L}(\{\vec{Z}_j\}_{j\in {\bar {\cal J}}}, \bar{{\cal C}})$ }\label{ssec00}

Recall that the conjugate $\rho^{*}$ of a function $\rho:\Re^{|\Omega|}\rightarrow\Re$
is defined as
$
\rho^{*}(p)=\sup_{\vec{Z}} \left\{p^{\top}\vec{Z}-\rho(\vec{Z}) \right\},
$
and the biconjugate $\rho^{**}$ of $\rho$ is defined as
$
\rho^{**}(\vec{Z})=\sup_{p}\left\{p^{\top}\vec{Z}-\rho^{*}(p) \right\}.
$

\begin{theorem} \label{rock}
(Conjugate Duality Theory (see, e.g.,  \cite{Rockafellar:1974aa}
for detailed proofs)) 
Given any function $\rho:\Re^{|\Omega|}\rightarrow \Re$,
the biconjugate $\rho^{**}$ satisfies $\rho^{**}\leq\rho$, and if $\rho$ is 
convex, then the following must hold.
\begin{enumerate}
\item $\rho=\rho^{**}$, and
\item given any given $\vec{Z}^{*}\in{\rm dom}(\mbox{\ensuremath{\rho}})$
such that $\rho(\vec{Z}^{*})$ is subdifferentiable at $\vec{Z}^{*}$,
the subdifferential satisfies $\partial\rho(\vec{Z}^{*})=\partial\rho^{**}(\vec{Z}^{*})=\arg\max_{p} \left\{p^{\top}\vec{Z}^{*}-\rho^{*}(p) \right\}$.
\end{enumerate}
\end{theorem}

In the following lemma, we provide the conjugate of the risk function $\rho_{\delta} \in {\cal L}(\{\vec{Z}_j\}_{j\in {\bar {\cal J}}}, \bar{{\cal C}})$.
\begin{lemma} \label{lemma1}
The conjugate function of the risk function $\rho_{\delta} \in {\cal L}(\{\vec{Z}_j\}_{j\in {\bar {\cal J}}}, \bar{{\cal C}})$ admits the form 
\[
\rho_{\delta}^{*}(y)=\begin{cases}
\max_{j\in {\bar {\cal J}}}\{y^{\top}\vec{Z}_{j}-\delta_{j}^{*}\} & {\rm if}\; y\in \bar{{\cal C}}\\
\infty &  {\rm if}\;  y \notin \bar{{\cal C}}
\end{cases},
\]
where $\delta_{j}^{*}:=\rho_{\delta}(\vec{Z}_{j})$, $\forall j\in {\bar {\cal J}}$.
\end{lemma}
\begin{proof}{Proof of Lemma \ref{lemma1}}
By definition, we have 
$$\rho_{\delta}^{*}(y)=\sup_{\vec{Z}}\{y^{\top}\vec{Z}-\rho_{\delta}(\vec{Z})\}\geq\max_{j\in {\bar {\cal J}}}\{y^{\top}\vec{Z}_{j}-\rho_{\delta}(\vec{Z}_{j})\}=\max_{j\in {\bar {\cal J}}}\{y^{\top}\vec{Z}_{j}-\delta_{j}^{*}\}.$$ 
To show the other direction, note first that 
\begin{eqnarray}
  \rho_{\delta}^*(y) && = \sup_{\vec{Z}} \left\{ y^\top \vec{Z}-\rho_{\delta}(\vec{Z}) \right\} \nonumber \\
                && = \sup_{\vec{Z}} \left\{ y^\top \vec{Z} - 
                               \left( \sup_{p \in \bar{{\cal C}}} p^\top \vec{Z} - \max_{j \in {\bar {\cal J}}}
                               \left\{p^\top \vec{Z}_j - \delta_j\right\} \right) \right\} \nonumber \\
                &&= \sup_{\vec{Z}}\inf_{p \in \bar{{\cal C}}}  \left\{ (y-p)^\top  \vec{Z}           
                               + \max_{j \in {\bar {\cal J}}}    \left\{p^\top \vec{Z}_j - \delta_j\right\} \right\} \nonumber\\
                &&=  \inf_{p \in \bar{{\cal C}}}\sup_{\vec{Z}}  \left\{ (y-p)^\top  \vec{Z}           
                               + \max_{j \in {\bar {\cal J}}}    \left\{p^\top \vec{Z}_j - \delta_j\right\} \right\} \nonumber  \\  
                &&=    \max_{j \in {\bar {\cal J}}}    \left\{y^\top \vec{Z}_j - \delta_j\right\},\; \text{if } y \in \bar{{\cal C}},\;\text{or } \infty \text{ otherwise},         \nonumber
 \end{eqnarray}
where the fourth equality follows the Sion's minimax theorem (\cite{Sion:1958aa}), as the function is linear in $\vec{Z}$, convex in $p$, and the set $\bar{{\cal C}}$ is compact.

We can then show that 
$$\rho_{\delta}^*(y) = \max_{j\in {\bar {\cal J}}} \left\{y^{\top}\vec{Z}_{j}-\delta_{j} \right\}  \leq  \max_{j\in {\bar {\cal J}}}\left\{y^\top \vec{Z}_j - \delta_j^*\right\},\; {\rm for}\; y \in \bar{{\cal C}},$$ since 
$$\delta_j^* = \rho_{\delta}(\vec{Z}_{j})  =  \sup_{p\in\bar{{\cal C}}} \left\{p^{\top}\vec{Z}_{j}-\max_{i\in {\bar {\cal J}}} \left\{p^{\top}\vec{Z}_{i}-\delta_{i}\right\}\right\}
  \leq  \sup_{p\in\bar{{\cal C}}} \left\{p^{\top}\vec{Z}_{j}-(p^{\top}\vec{Z}_{j}-\delta_{j}) \right\}=\delta_{j}.$$\Halmos
\end{proof}

\newpage
\section{Dual representation of convex risk functions} \label{ssec01}
\begin{theorem} (see, e.g.,  \cite{Ruszczynski:2006aa}) \label{repconv}
A risk function $\rho$ satisfies the axiom of monotonicity, convexity, and translation invariance
if any only if it admits the representation 
\begin{equation} \label{sup}
\rho(\vec{Z})=\sup_{p \in \Delta}\{p^{\top}\vec{Z}-\rho^{*}(p)\}, 
\end{equation}
where $\Delta := \left\{p \in \Re^{|\Omega|}\;\middle |\; \vec{1}^\top p = 1,\;p\geq 0   \right\}$.
\end{theorem}

The following is a list of risk functions that have been considered in this paper as the candidates for the reference risk function $\tilde{\rho}$.
\begin{table}[h]
\center
\begin{tabular}{|c|c|}
\hline 
Risk function & Formulation\tabularnewline
\hline 
\hline 
Maximum loss & $\max_{i}\{Z(\omega_{i})\}$\tabularnewline
\hline 
Expectation & $\mathbb{E}[Z]$\tabularnewline
\hline 
Mean-absolute-deviation & $\mathbb{E}[Z]+\gamma \mathbb{E}[|Z-\mathbb{E}[Z]|]$, $\gamma \in [0,\frac{1}{2}]$\tabularnewline
\hline 
Mean-upper-semideviation & $\mathbb{E}[Z]+\gamma(\mathbb{E}[([Z-\mathbb{E}[Z]]^{+})^{s}])^{1/s}$, $\gamma \in [0,1]$, $s \geq 1$ \tabularnewline
\hline 
Conditional Value-at-Risk (CVaR) &  $\frac{1}{1-\alpha} \int_{\alpha}^1 F_{Z}^{-1}(t)dt$, $\alpha \in [0,1)$ \tabularnewline
\hline 
Spectral risk measures & $\int_{0}^{1}F_{Z}^{-1}(t)\phi(t)dt$, \tabularnewline
& $\phi$ is nonnegative, non-decreasing, and $\int_0^1 \phi(t)dt = 1$ \tabularnewline
\hline 
\hline 
\end{tabular}
\caption{Several well-known risk functions, where $F_Z^{-1}$ stands for the generalized inverse distribution function. }
\label{table:riskmeasures}
\end{table}

All the risk functions in Table \ref{table:riskmeasures} are law invariant convex risk functions, and we provide their dual representations based on (\ref{sup}). In Example \ref{Ex2} to \ref{Ex5}, we denote by $\bar{p} \in \Re^{|\Omega|}$ the probability mass function  (i.e., $\bar{p}_i = \mathbb{P}(\{\omega_i\})$, $\omega_i \in \Omega$). 

\begin{example} \label{Ex1} (Maximum loss) Its dual representation is simply (\ref{sup}) with 
$\rho^*(p) = 0$.
\end{example}
\begin{example} \label{Ex2} (Expectation) Its dual representation is also trivial: (\ref{sup}) with 
$\rho^*(p) = 0$ and $\Delta$ replaced by ${\cal C}:= \left\{q\; \middle |\; q = \bar{p}\right\}$. 
\end{example}
\begin{example}
(Mean-absoulte-deviation) Its dual  representation has been studied in \cite{Ruszczynski:2006aa}: namely, (\ref{sup}) with $\rho^*(p) = 0$ and $\Delta$ replaced by ${\cal C}:= \left\{q\; \middle |\; q_i = \bar{p}_i(1 + \gamma(h_i - \sum_{i} \bar{p}_i h_i)),\;  ||h||_{\infty} \leq 1\right\}$ and $h \in \Re^{|\Omega|}$.
\end{example}
\begin{example}
(Mean-upper-semideviation) The derivation of its dual representation is similar to that of the previous example (see, e.g., \cite{Ruszczynski:2006aa}); namely, the representation is given by
$$ \rho(\vec{Z}) = \max_{p \in {\cal C}} p^\top \vec{Z}, $$
where ${\cal C}:= \left\{q\; \middle |\; q_i = \bar{p}_i(1 + \gamma(h_i - \sum_{i} \bar{p}_i h_i)),\; \sum_{i} \bar{p}_i |h_i|^{t} \leq 1,\; h \geq 0\right\}$, $h \in \Re^{|\Omega|}$ and $t = \frac{s}{s-1}$. 
\end{example}

\begin{example} \label{Ex5} (Conditional Value-at-Risk (CVaR))
The following dual representation of CVaR is fairly standard: 
$$ \rho(\vec{Z}) = \max_{p \in {\cal C}} p^\top \vec{Z}, $$
where ${\cal C}:= \left\{ q \; \middle |\; q_i \leq \frac{1}{1-\alpha} \bar{p}_i,\; 1^\top q = 1,\; q \geq 0 \right\}.$
\end{example}

Note that in the following example we assume that the probability mass function $\bar{p}$ is uniform so that we can write the dual representation in terms of permutation operators. This provides the basis to derive a more general representation in Example \ref{spec2}.

\begin{example}\label{specex} (Spectral risk measures) Given that $\mathbb{P}(\{\omega_i\}) = 1/M$ for $\omega_i \in \Omega$, any spectral risk measure can be equivalently written as
\footnote{Indeed, letting $\bar{g} : \{1,...,M\} \rightarrow \{1,...,M\}$ be an one-to-one mapping such that 
$Z(\omega_{\bar{g}(1)}) \leq Z(\omega_{\bar{g}(2)}) \leq \cdots \leq Z(\omega_{\bar{g}(M)})$, we have for any $t \in (\frac{j-1}{M}, \frac{j}{M}]$, $j \in \{1,...,M\}$, $F_Z^{-1}(t) = Z(\omega_{\bar{g}(j)})$ must hold. This can be verified via the definition $ F^{-1}_Z(t) = \inf \{z \; :\; \sum_{Z(\omega_i) \leq z} \mathbb{P}(\{\omega_i\}) \geq t\}$. Namely, since $\mathbb{P}(\{\omega_i\}) = \frac{1}{M}$ we must have $\sum_{Z(\omega_i) \leq Z(\omega_{\bar{g}(j)})}\mathbb{P}(\{\omega_i\}) \geq t^*$, $t^* \in (\frac{j-1}{M}, \frac{j}{M}]$. And there must not exist $z < Z(\omega_{\bar{g}(j)})$ such that $\sum_{Z(\omega_i) \leq z} \mathbb{P}(\{\omega_i\}) \geq t^*$ because such $z$ must be $z \in \{ Z(\omega_{\bar{g}(k)})\}_{k=1,...,j-1}$ and it contradicts the fact that 
$\sum_{Z(\omega_i) \leq Z(\omega_{\bar{g}(k)})} \mathbb{P}(\{\omega_i\}) \leq \frac{k}{M}$ for any $k \in \{1,...,j-1\}$. Hence, we can write $\rho(\vec{Z}) = \sum_{j=1}^M(\int_{\frac{j-1}{M}}^{\frac{j}{M}}
F_Z^{-1}(t)\phi(t)dt) = \sum_{j=1}^M(\int_{\frac{j-1}{M}}^{\frac{j}{M}}\phi(t)dt)Z(\omega_{\bar{g}(j)})
= \sum_{j=1}^M \phi_j Z(\omega_{\bar{g}(j)})$, where $\phi_j := \int_{\frac{j-1}{M}}^{\frac{j}{M}}\phi(t)dt$.}:
$$\rho(\vec{Z})=\phi^{\top}\varphi(\vec{Z}),$$
where $\varphi:\Re^{M}\rightarrow\Re^{M}$ denotes an ordering operator such that  $\varphi(\vec{Z})_{1}\leq\cdots\leq\varphi(\vec{Z})_{M}$, and 
$\phi\in\Re^{M}$ satisfies $\phi \geq 0$, $\sum_{i=1}^M \phi_i = 1$, and $\phi_{1}\leq\cdots\leq\phi_{M}$. It is not difficult to verify that the following dual representation 
attains the same optimal value as the one above:
$$ \rho(\vec{Z}) = \max_{p,\sigma} \left\{ p^\top \vec{Z} \;\middle |\; p = \sigma(\phi),\; \sigma \in \Sigma \right\}, $$
where $\sigma$ is an operator that permutes a $M$-dimensional vector, and $\Sigma$ is the set of all such operators. It can be further reformulated as follows using the convex hull operator $\text{Conv}(\cdot)$:
$$ \rho(\vec{Z}) = \max_{p \in \Re^{M}_+ \cap {\cal C}} p^\top \vec{Z}, $$
where ${\cal C}:= \left\{ q \; \middle |\; q \in \text{Conv}(\{\sigma(\phi), \sigma \in \Sigma\})\right\}. $
\end{example}

\newpage
\section{Additional Results} \label{ssec02}
\begin{lemma}  \label{lemo}
Let $\rho := \sup_{\rho' \in {\cal R}_0} \rho'$, where ${\cal R}_0:= {\cal R}\; ({\rm or}\; {\cal R}_F) \cap {\cal R}_{inv}(\gamma^*) \cap {\cal R}_{el}(\{(L_k,U_k)\}_{k\in {\cal K}}) \cap {\cal R}_{ref}(\epsilon^*)$ for some fixed $\gamma^* \in [0,\infty)$ and $\epsilon^* \in (0, \infty]$. We have $ \rho \in {\cal R}_0.$
\end{lemma}
\begin{proof}{Proof of Lemma \ref{lemo}}
We skip the steps of confirming that $\rho \in {\cal R}\; ({\rm or}\; {\cal R}_F)$ and $\rho \in  {\cal R}_{el}(\{(L_k,U_k)\}_{k\in {\cal K}})$ because these steps can be found in Proposition 1 in \cite{Delage:2015aa}. We are left to show that $\rho \in {\cal R}_{inv}(\gamma^*) \cap {\cal R}_{ref}(\epsilon^*)$. We can confirm that $\rho \in {\cal R}_{inv}(\gamma^*)$, because
\[
\rho(\vec{W}^t) = \sup_{\rho' \in {\cal R}_0} \rho'(\vec{W}^t) 
\leq \sup_{\rho' \in {\cal R}_0} \rho'(\vec{W}) + \gamma^*
= \rho(\vec{W}) + \gamma^*,\; \forall \vec{W} \in {\cal W}^t, \]
where we use the fact that the condition $\rho'(\vec{W}^t) \leq \rho'(\vec{W})+\gamma^*,\;\; \forall \vec{W} \in {\cal W}^t$ holds for any $\rho' \in {\cal R}_0$.\\
We can also confirm that $\rho \in {\cal R}_{ref}(\epsilon^*)$, because
\[ 
\rho(\vec{Z}) =  \sup_{\rho' \in {\cal R}_0} \rho'(\vec{Z}) \leq 
\tilde{\rho}(\vec{Z}) +\epsilon,\; \forall Z,\]
and
\[   \rho(\vec{Z}) = \sup_{\rho' \in {\cal R}_0} \rho'(\vec{Z})  \geq \rho''(\vec{Z}) \geq 
\tilde{\rho}(\vec{Z}) - \epsilon,\; \forall Z, \]
where $\rho''$ is some function satisfying $\rho'' \in {\cal R}_0$. In the last inequality of the first line we use the fact that the condition $-\epsilon \leq \rho' - \tilde{\rho} \leq \epsilon$ holds for any $\rho' \in {\cal R}_0$. \Halmos
\end{proof}

\begin{lemma} \label{lem2}
A risk function $\rho$ satisfies $\rho \in \bars{{\cal R}}$ if and only if it admits a supremum representation $\rho(\vec{Z})=\sup_{p}\left\{p^{\top}\vec{Z}-\rho^{*}(p)\right\}$ that satisfies $\rho^*(p) = \rho^*(\sigma(p)), \; \forall \sigma \in \Sigma$. As a result, a coherent risk measure $(\ref{supcoh})$ is permutation invariant if and only if its support set ${\cal C}$ satisfies $p \in {\cal C} \Leftrightarrow \sigma(p) \in {\cal C}$.
\end{lemma}
\begin{proof}{Proof of Lemma \ref{lem2}}
We can first show that $\rho$ is indeed permutation invariant given $\rho^*(\sigma(p))=\rho^*(p)$, since
\begin{equation}
\rho(\sigma(\vec{Z})) = \sup_{p} p^\top \sigma(\vec{Z}) - \rho^*(p)
= \sup_{p} \sigma^{-1}(p)^\top \vec{Z} - \rho^*(p)
= \sup_{p} p^\top \vec{Z} - \rho^*(\sigma(p))
= \rho(\vec{Z}).
\end{equation}
To show the other direction, we apply the definition of the conjugate and have for any $\sigma\in\Sigma$,
$$\rho^*(\sigma(p)) = \sup_{\vec{Z}} \vec{Z}^\top \sigma(p) - \rho (\vec{Z}) = 
\sup_{\vec{Z}} \sigma^{-1}(\vec{Z})^\top p - \rho (\vec{Z}) 
=\sup_{\vec{Z}} \vec{Z}^\top p - \rho(\sigma(\vec{Z})),$$
and hence for any permutation-invariant $\rho$  (i.e., $\rho(\sigma(\vec{Z}))=\rho(\vec{Z})$), we must have $\rho^*(\sigma(p))=\rho^*(p)$. In the case of coherent risk measure, the fact that $\rho^*(p)$ satisfies $\rho^*(p) = 0$ for $p \in {\cal C}$ and $\infty$ otherwise implies immediately the symmetry of ${\cal C}$. \Halmos
\end{proof}

\begin{corollary} \label{inv3a}
Under the same assumption in Proposition \ref{prop_final}, the inverse optimization problem (\ref{eq:inv2}) can be solved by the risk function (\ref{mod111}), and the parameter $\delta$ is calculated by solving the convex program 
\begin{eqnarray}
\min_{\delta,y_{j},v_{i,j},w_{i,j}, \gamma} && \sum_{t \in {\cal T}} \gamma_t\nonumber \\
 {\rm subject\;to}&& \vec{1}^{\top}v_{i,j}+\vec{1}^{\top}w_{i,j}\leq\delta_{i}-\delta_{j}+y_{j}^{\top}\vec{X}_{j},\;\;\; \forall j \in {\cal J},\;\forall i\neq j, \nonumber \\
 && \vec{X}_{i}y_{j}^{\top}-v_{i,j}\vec{1}^{\top}-\vec{1}w_{i,j}^{\top}\leq 0,\;\;\; \forall j \in {\cal J},\; \forall i\neq j, \nonumber \\
 && y_j \in \bars{{\cal C}}, \;\;\; \forall j \in {\cal J}, \nonumber \\
 && y_{t}^{\top}\vec{X}_{t}\leq h_t(y_{t}) + \gamma_t,\;\;\; \forall t \in {\cal T}, \nonumber\\
 && \delta_{i}\leq\delta_{j},\;\;\; \forall(i,j)\in{\cal B}, \nonumber\\
 && |\delta_j - \tilde{\rho}(\vec{X}_j)| \leq \epsilon^*, \;\;\;\forall  j \in {\cal J},\nonumber
\end{eqnarray}
where $\delta \in \Re^{|{\cal J}|}$, $y_{j}\in\Re^{M}$, $v_{i,j}\in\Re^{M}$, $w_{i,j}\in\Re^{M}$, and $\gamma \in \Re^{|{\cal T}|}$; the set ${\cal B}:= \left\{(i,j)\in\{1,2,...,{\cal J}\}^{2}\; \middle |\;(\vec{X}_{i},\vec{X}_{j})\in\{(\vec{L}_{k},\vec{U}_{k})\}_{k\in {\cal K}} \right\}$; 
and $h_t$ denotes the function $h_t(y):=\min_{x}\left\{y^{\top}\vec{Z}^t(x) \;\middle |\;x\in{\cal X}^{t}\right\}$. 
\end{corollary}

\begin{corollary} \label{inv3b}
Under the same assumption in Proposition \ref{prop_final}, the inverse optimization problem (\ref{eq:inv3}) can be solved by the risk function (\ref{mod111}), and the parameter $\delta$ is calculated by solving the convex program 
\begin{eqnarray}
\max_{\delta,y_{j},v_{i,j},w_{i,j}} && \sum_{j \in {\cal J}} \delta_j \nonumber \\
 {\rm subject\; to}&& \vec{1}^{\top}v_{i,j}+\vec{1}^{\top}w_{i,j}\leq\delta_{i}-\delta_{j}+y_{j}^{\top}\vec{X}_{j},  \;\;\; \forall j \in {\cal J},\;\forall i\neq j, \nonumber \\
 && \vec{X}_{i}y_{j}^{\top}-v_{i,j}\vec{1}^{\top}-\vec{1}w_{i,j}^{\top}\leq 0, \;\;\; \forall j \in {\cal J}, \;\forall i\neq j, \nonumber \\
 && y_j \in \bars{{\cal C}}, \;\;\; \forall j \in {\cal J}, \nonumber \\
 && y_{t}^{\top}\vec{X}_{t}\leq h_t(y_{t}) + \gamma_t^*,  \;\;\; \forall t \in {\cal T}, \nonumber\\
 && \delta_{i}\leq\delta_{j}, \;\;\; \forall(i,j)\in{\cal B}, \nonumber \\
 && |\delta_j - \tilde{\rho}(\vec{X}_j)| \leq \epsilon^*,  \;\;\; \forall  j \in {\cal J},\nonumber
\end{eqnarray}
where $\delta \in \Re^{|{\cal J}|}$, $y_{j}\in\Re^{M}$,   $v_{i,j}\in\Re^{M}$, $w_{i,j}\in\Re^{M}$; the set ${\cal B}:= \left\{(i,j)\in\{1,2,...,{\cal J}\}^{2}\;\middle |\;(\vec{X}_{i},\vec{X}_{j})\in\{(\vec{L}_{k},\vec{U}_{k})\}_{k\in {\cal K}}\right\}$; 
and $h_t$ denotes the function $h_t(y):=\min_{x}\left\{y^{\top}\vec{Z}^t(x) \;\middle |\;x\in{\cal X}^{t}\right\}$. 
\end{corollary}

\begin{corollary} \label{inv4a}
Under the same assumption in Proposition \ref{lastpros}, the inverse optimization problem (\ref{eq:inv2}) can be solved by the risk function (\ref{pro6law}), and the parameter $\delta$ is calculated by solving the convex program 
\begin{eqnarray}
\min_{\delta,y_{j},v_{i,j},w_{i,j},\gamma} && \sum_{t \in {\cal T}} \gamma_t \nonumber \\
 {\rm subject\; to} && \vec{1}^{\top}v_{i,j}+\vec{1}^{\top}w_{i,j}\leq\delta_{i}-\delta_{j}+y_{j}^{\top}\vec{S}_{j},  \;\;\; \forall j \in {\cal J},\; \forall i\neq j, \nonumber\\
 && \vec{S}_{i}y_{j}^{\top}-\Lambda_{i,j} \circ (v_{i,j}\vec{1}^{\top})-\vec{1}w_{i,j}^{\top}\leq0 , \;\;\; \forall j \in {\cal J},\;   \forall i\neq j, \nonumber \\
 && y_j \in {\cal C}_j \subseteq \Re^{\tau_j}_+ ,  \;\;\; \forall j \in {\cal J},\nonumber \\
 && y_{t}^{\top}\vec{S}_t \leq h_t(y_{t}) + \gamma_t, \;\;\; \forall t \in {\cal T}, \nonumber\\
 && \delta_{i}\leq\delta_{j} , \;\;\; \forall(i,j)\in{\cal B}, \nonumber\\
 && |\delta_j - \tilde{\rho}(\vec{X}_j)| \leq \epsilon^*,  \;\;\; \forall  j \in {\cal J},\nonumber
\end{eqnarray}
where $\delta \in \Re^{|{\cal J}|}$, $y_j \in \Re^{\tau_j}$, $v_{i,j} \in \Re^{\tau_i}$, $w_{i,j} \in \Re^{\tau_j}$, and $\gamma \in \Re^{|{\cal T}|}$; $\vec{S}_t:=(Z(x^{t},\xi^t_1),...,Z(x^{t},\xi^t_{\tau_0^t}))^\top$; the set ${\cal B}:=\left\{(i,j)\in\{1,2,...,{\cal J}\}^{2}\; \middle |\;(\vec{X}_{i},\vec{X}_{j})\in\{(\vec{L}_{k},\vec{U}_{k})\}_{k\in {\cal K}}\right\}$; and $h_t$ denotes the function $h_t(y):=\min_{x} \left\{\sum_{o=1}^{\tau_0^t} y_o Z(x,\xi^t_o) \;\middle |\;x\in{\cal X}^{t}\right\}$. The coefficient $\Lambda_{i,j}$ is calculated by $(\Lambda_{i,j})_{m,n} = \bar{p}_n^{j}/\bar{p}^i_m$, $n=1,...,\tau_j$, $m=1,...,\tau_i$. Moreover, the above sets ${\cal C}_j$, $j \in {\cal J}$, can be derived as indicated in Proposition \ref{lastpros}.
\end{corollary}

\begin{corollary} \label{inv4b}
Under the same assumption in Proposition \ref{lastpros}, the inverse optimization problem (\ref{eq:inv3}) can be solved by the risk function (\ref{pro6law}), and the parameter $\delta$ is calculated by solving the convex program
\begin{eqnarray}
\max_{\delta,y_{j},v_{i,j},w_{i,j}} &&  \sum_{j \in {\cal J}} \delta_j  \nonumber \\
 {\rm subject\; to}&& \vec{1}^{\top}v_{i,j}+\vec{1}^{\top}w_{i,j}\leq\delta_{i}-\delta_{j}+y_{j}^{\top}\vec{S}_{j}, \;\;\; \forall j \in {\cal J},\; \forall i\neq j, \nonumber \\
 && \vec{S}_{i}y_{j}^{\top}-\Lambda_{i,j} \circ (v_{i,j}\vec{1}^{\top})-\vec{1}w_{i,j}^{\top}\leq0 ,  \;\;\; \forall j \in {\cal J},\;   \forall i\neq j, \nonumber \\
 && y_j \in {\cal C}_j \subseteq \Re^{\tau_j}_+ , \;\;\; \forall j \in {\cal J},\nonumber \\
 && y_{t}^{\top}\vec{S}_t \leq h_t(y_{t})+ \gamma_t^*, \;\;\; \forall t \in {\cal T}, \nonumber\\
 && \delta_{i}\leq\delta_{j} ,  \;\;\; \forall(i,j)\in{\cal B}, \nonumber \\
 && |\delta_j - \tilde{\rho}(\vec{X}_j)| \leq \epsilon^*,  \;\;\; \forall  j \in {\cal J},\nonumber
\end{eqnarray}
where $\delta \in \Re^{|{\cal J}|}$, $y_j \in \Re^{\tau_j}$, $v_{i,j} \in \Re^{\tau_i}$, and $w_{i,j} \in \Re^{\tau_j}$;  $\vec{S}_t:=(Z(x^{t},\xi^t_1),...,Z(x^{t},\xi^t_{\tau_0^t}))^\top$; the set ${\cal B}:= \left\{(i,j)\in\{1,2,...,{\cal J}\}^{2}\; \middle |\;(\vec{X}_{i},\vec{X}_{j})\in\{(\vec{L}_{k},\vec{U}_{k})\}_{k\in {\cal K}}\right\}$; and $h_t$ denotes the function $h_t(y):=\min_{x} \left\{\sum_{o=1}^{\tau_0^t} y_o Z(x,\xi^t_o) \;\middle |\;x\in{\cal X}^{t}\right\}$. The coefficient $\Lambda_{i,j}$ is calculated by $(\Lambda_{i,j})_{m,n} = \bar{p}_n^{j}/\bar{p}^i_m$, $n=1,...,\tau_j$, $m=1,...,\tau_i$. Moreover, the above sets  ${\cal C}_j$, $j \in {\cal J}$, can be derived as indicated in Proposition \ref{lastpros}.
\end{corollary}

\newpage
\section{Proofs} \label{ssec03}
\subsection*{Proof of Proposition \ref{pro4}} 
\begin{proof}{Proof of Proposition \ref{pro4}}
The proof consists of two parts.  In the first part of the proof, we show that for any optimal solution $\rho_0$, the domain of its conjugate function must admit the set ${\cal C}$. Then, in the second part we confirm the optimality of $\rho_{\delta}$ where $\delta_j = \rho_0(\vec{X}_j),\; \forall j \in {\cal J}$. 

{\bf (Part I)} Prove that any optimal solution ${\rho_0}$ admits the following representation
\begin{equation}
{\rho_0}(\vec{Z}) =\sup_{p\in{\cal C}}p^{\top}\vec{Z}-{\rho^{*}_0}(p),\label{eq:convconj}
\end{equation}
where ${\rho^{*}_0}(p):=\sup_{\vec{Z}}p^{\top}\vec{Z}-{\rho_0}(\vec{Z})$.

We start by invoking the theory of conjugate duality (Theorem \ref{rock} (1.)), 
which states that given that ${\rho_0}\in {\cal R}$ we can always represent ${\rho_0}$ by 
\begin{equation}
{\rho_0}(\vec{Z})=\sup_{p} p^{\top}\vec{Z}-{\rho^{*}_{0}}(p). \label{eq:convconj2}
\end{equation}

The rest is to show that the conjugate ${\rho^{*}_0}$ must satisfy $\rho^{*}_0(p)=\infty$ for any $p \in \Re^{|\Omega|} \setminus {\cal C}$. We prove this by contradiction. Suppose that there exists a solution $p_{0} \in  \Re^{|\Omega|} \setminus {\cal C}$ such that ${\rho^{*}_0}(p_{0})<\infty$. The fact that
${\cal C}$ is a closed convex set implies that (by hyperplane separation theorem) there must exist a nonzero vector $\vec{R} \in \Re^{|\Omega|}$ and $b \in \Re$ such that 
\begin{eqnarray}
 && p^\top \vec{R} \leq b,\;\forall p \in {\cal C}, \label{no1} \\
 && p_{0}^\top \vec{R} > b \;\;{\rm (or\; equivalently\; } p_{0}^\top \vec{R} = b + \epsilon\;\;{\rm for\; some\; }\epsilon > 0) \label{no2}
 \end{eqnarray}
hold. Let us consider a random loss $\vec{Z}:=\lambda \vec{R}$ for some $\lambda > 0$.  Based on the dual representation of the reference risk function $\tilde{\rho}$, we have 
\begin{equation}
\tilde{\rho}(\lambda\vec{R}) = \sup_{p \in {\cal C}} \lambda p^\top \vec{R} \leq
\lambda b \;\;\; ({\rm due\;to\;} (\ref{no1})). \label{eq:ineq2}
\end{equation}
Based on (\ref{eq:convconj2}), we also have for the fixed $p_0$
\begin{equation}
{\rho_0}(\lambda\vec{R})\geq p_{0}^{\top}(\lambda\vec{R})-{\rho^{*}_0}(p_{0})=\lambda b+\lambda\epsilon-{\rho^{*}_{0}}(p_{0}) \;\;\; ({\rm due\;to\;} (\ref{no2})).\label{eq:ineq1}
\end{equation}
By subtracting (\ref{eq:ineq2}) from
(\ref{eq:ineq1}), we arrive at
\[ {\rho_0}(\lambda\vec{R})-\tilde{\rho}(\lambda\vec{R})\geq\lambda\epsilon-{\rho^{*}_0}(p_{0})\rightarrow\infty,
\]
as $\lambda\rightarrow\infty.$ This contradicts the fact that ${\rho_0}$ gives the optimal value $u^{*}:=||{\rho_0}-\tilde{\rho}||_{\infty}<\infty$.

{\bf (Part II)} Verify the optimality of $\rho_{\delta}$, where $\delta_j = \rho_0(\vec{X}_j),\; \forall j \in {\cal J}$. 

To proceed, we need to prove first the following inequalities between $\rho_0$ and $\rho_{\delta}$:
\begin{eqnarray}
{\rho_0}(\vec{Z}) && = \sup_{p \in {\cal C}} p^\top \vec{Z} - \rho_0^*(p) \;\;\;\;\;\;\;\;\;\;\;\;\;\;({\rm due\;to\;} (\ref{eq:convconj2})) \nonumber \\
 && = \sup_{p \in {\cal C}} p^\top \vec{Z} - \left\{ \sup_{\vec{Z}} p^\top \vec{Z} - \rho_0(\vec{Z}) \right\} \nonumber \\
 && \leq \sup_{p \in {\cal C}} p^\top \vec{Z} - \left\{ \max_{j \in {\cal J}} \left\{ p^\top \vec{X}_j - \rho_0(\vec{X}_j) \right\} \right\} \nonumber \\
 && = \rho_{\delta}(\vec{Z}), \;\;\; \forall \vec{Z}, \label{eq:maineq}
\end{eqnarray}
and
\begin{eqnarray}
\rho_{\delta}(\vec{X}_{i}) && = \sup_{p \in {\cal C}} p^\top \vec{X}_{i} - \left\{ \max_{j \in {\cal J}} \left\{ p^\top \vec{X}_j - \rho_0(\vec{X}_j) \right\} \right\} \nonumber \\
&& \leq \sup_{p \in {\cal C}} p^\top \vec{X}_{i} - \left\{ p^\top \vec{X}_{i} - \rho_0(\vec{X}_i) \right\} \nonumber \\
 &&  = \rho_0(\vec{X}_i),\;\;\; \forall i \in {\cal J}. \label{eq:maineq2}
\end{eqnarray}

The inequalities (\ref{eq:maineq}) and (\ref{eq:maineq2}) imply, first, that 
$\rho_{\delta}(\vec{0}) = {\rho_0}(\vec{0}) = 0$  (i.e., $\rho_{\delta}$ satisfies the normalization condition).
One can confirm also that based on the representation theory (Theorem \ref{repconv}), $\rho_{\delta}$ must be monotonic, convex, and translation invariant by its construction. Hence, $\rho_{\delta} \in {\cal R}$. The inequalities (\ref{eq:maineq}) and (\ref{eq:maineq2}) imply also that $\rho_{\delta}$ satisfies the constraints (\ref{invv}) and (\ref{inv2}) because they lead to the following:
\[
\mbox{\ensuremath{\rho}}_{\delta}(\vec{L}_{k})\leq{\rho_0}(\vec{L}_{k})\leq{\rho_0}(\vec{U}_{k})\leq\rho_{\delta}(\vec{U}_{k}), \;\;\;\forall k \in {\cal K},
\]
\[
\rho_{\delta}(\vec{W}^{t})\leq{\rho_0}(\vec{W}^{t})\leq{\rho_0}(\vec{W})\leq\rho_{\delta}(\vec{W}), \;\;\;\forall\vec{W}\in{\cal W}^{t},\; \forall t \in {\cal T}.
\]
Finally, we verify that $\rho_{\delta}$ reaches the optimal value of $u^*$  (i.e., $||\rho_{\delta}-\tilde{\rho}|| = u^{*}$). We have 
\begin{eqnarray}
\rho_{\delta}(\vec{Z}) && = \sup_{y\in{\cal C}}y^{\top}\vec{Z}-\max_{j \in {\cal J}}\left\{y^{\top}\vec{X}_j-\rho_0(\vec{X}_j)\right\} \nonumber \\
&&  \leq  \sup_{y\in {\cal C}} y^\top \vec{Z} - \max_{j\in {\cal J}} \left\{ y^\top \vec{X}_j - (\tilde{\rho}(\vec{X}_{j})+u^{*})\right\} \;\;(\text{since }{\rho_0}(\vec{X}_j) \leq \tilde{\rho}(\vec{X}_j) + u^*,\;\forall j \in {\cal J})
\nonumber \\ 
&& \leq  \sup_{y\in {\cal C}} y^\top \vec{Z} - (y^\top \vec{0} - (\tilde{\rho}(\vec{0})+ u^{*})) = \tilde{\rho}(\vec{Z})+u^{*},\; \forall \vec{Z}, \label{inqq11} 
\end{eqnarray}
and also have
\begin{eqnarray}
\rho_{\delta}(\vec{Z}) && = \sup_{y\in{\cal C}}y^{\top}\vec{Z}-\max_{j \in {\cal J}} \left\{y^{\top}\vec{X}_j-\rho_0(\vec{X}_j)\right\} \nonumber \\
&& \geq   \sup_{y\in {\cal C}} y^\top \vec{Z} - \max_{j\in {\cal J}} \left\{ y^\top \vec{X}_j - (\tilde{\rho}(\vec{X}_{j}) -u^{*})\right\} \;\;(\text{since }{\rho_0}(\vec{X}_j) \geq \tilde{\rho}(\vec{X}_j) - u^*,\;\forall j \in {\cal J}) \nonumber \\
&& =  \sup_{y\in {\cal C}} y^\top \vec{Z} - u^*- \max_{j\in {\cal J}}\left\{ y^\top \vec{X}_j - \tilde{\rho}(\vec{X}_{j}) \right\} \;\; \nonumber \\
&&\geq \sup_{y\in {\cal C}} y^\top \vec{Z} -u^{*} = \tilde{\rho}(\vec{Z})-u^{*},\; \forall \vec{Z}, \label{inqq12}
\end{eqnarray}
where the second inequality follows the dual representation of $\tilde{\rho}$; namely 
for any $j \in {\cal J}$,
\begin{equation}
 \tilde{\rho}(\vec{X}_j) = \sup_{y\in {\cal C}} y^\top \vec{X}_j \Leftrightarrow  y^\top \vec{X}_j - \tilde{\rho}(\vec{X}_j) \leq 0,\; \forall y \in {\cal C}. \nonumber
\end{equation}
\Halmos
\end{proof}

\subsection*{Proof of Proposition \ref{pro2}}
\begin{proof}{Proof of Proposition \ref{pro2}}
Because $\rho_{\delta} \in {\cal R}$, we know from the theory of conjugate duality (Theorem \ref{rock} (1.)) that 
\[
\rho_{\delta}(\vec{Z}) = \rho_{\delta}^{**}(\vec{Z}) ,\;\forall \vec{Z}
\]
must hold. Obviously, the condition can be equivalently stated in terms of the following two inequalities: 
\begin{eqnarray}
  && \rho_{\delta}(\vec{Z}) \geq \rho_{\delta}^{**}(\vec{Z}),\; \forall \vec{Z},\nonumber\\
  && \rho_{\delta}(\vec{Z}) \leq \rho_{\delta}^{**}(\vec{Z}), \; \forall \vec{Z}. \nonumber
\end{eqnarray}
The first inequality, however, is always satisfied given that it is implied by the construction of bi-conjugate function (see Theorem \ref{rock}). Therefore, it suffices to proceed with the second inequality.

Consider the above second inequality with $\vec{Z}=\vec{Z}_j$, $j\in {\bar {\cal J}}$. By expanding the biconjugate $\rho_{\delta}^{**}$ based on its definition and using 
the conjugate function  $\rho_{\delta}^{*}$ derived in Lemma \ref{lemma1}, we have the following:
\begin{eqnarray*}
 && \rho_{\delta}^{**}(\vec{Z}_j)\geq\rho_{\delta}(\vec{Z}_j) ,\;\; \forall j\in {\bar {\cal J}}\\
\Leftrightarrow && \sup_{y} \left\{y^{\top}\vec{Z}_j-\rho_{\delta}^{*}(y) \right\}\geq\rho_{\delta}(\vec{Z}_j) ,\;\forall j\in {\bar {\cal J}}\\
\Leftrightarrow && \sup_{y\in \bar{{\cal C}}} \left\{y^{\top}\vec{Z}_j-\max_{i\in {\bar {\cal J}}} \left\{y^{\top}\vec{Z}_{i}-\rho_{\delta}(\vec{Z}_{i}) \right\} \right\} \geq\rho_{\delta}(\vec{Z}_j) ,\; \forall j\in {\bar {\cal J}}\\
\Leftrightarrow && \exists y_j \in \bar{{\cal C}} :\;\;y_j^{\top}\vec{Z}_j-\max_{i\in {\bar {\cal J}}}\left\{y_j^{\top}\vec{Z}_{i}-\rho_{\delta}(\vec{Z}_{i})\right\}\geq\rho_{\delta}(\vec{Z}_j) ,\;\forall j\in {\bar {\cal J}}\\
\Leftrightarrow && \exists y_j\in \bar{{\cal C}}:\;\;y_j^{\top}\vec{Z}_j-y_j^{\top}\vec{Z}_{i}+\rho_{\delta}(\vec{Z}_{i})\geq\rho_{\delta}(\vec{Z}_j),\; \forall i\neq j,
\end{eqnarray*}
where in the fourth line we apply the fact that $\bar{{\cal C}}$ is closed and bounded  (i.e., compact). This completes the first part of the proof.

To prove the other direction, note that for any feasible solution
$\{\delta_{j}^{*}\}_{j\in {\bar {\cal J}}}$ and $\{y_{j}^{*}\}_{j\in {\bar {\cal J}}}$ we can
always construct a $(\{\vec{Z}_j\}_{j\in {\bar {\cal J}}},\bar{{\cal C}})$-supported risk function (\ref{eq:polyreppeat}). It always satisfies $\rho_{\delta}(\vec{Z}_j)\leq\delta_{j}^{*}$ because  
$$\rho_{\delta}(\vec{Z}_j) \leq \sup_{y\in \bar{{\cal C}}}y^{\top}\vec{Z}_j- y^{\top}\vec{Z}_j+\delta_{j}^{*} = \delta_j^*,$$
and satisfies $\rho_{\delta}(\vec{Z}_j)\geq\delta_{j}^{*}$
because
$$\rho_{\delta}(\vec{Z}_j) \geq y_{j}^{*\top}\vec{Z}_j-\max_{i\in {\bar {\cal J}}} \left\{y_{j}^{*\top}\vec{Z}_{i}-\delta_{i}^{*} \right\} \geq  \delta_{j}^{*} \;\; \text{(due to the feasibility of \ensuremath{\{\delta_{j}^{*}\}_{j\in {\bar {\cal J}}}} and \{\ensuremath{y_{j}^{*}\}_{j\in {\bar {\cal J}}}}).} $$
This proves the existence of a risk function $\rho_{\delta} \in {\cal L}(\{\vec{Z}_j\}_{j\in \bar{{\cal J}}}, \bar{{\cal C}})$ that satisfies $\rho_{\delta}(\vec{Z}_j) = \delta_j^*$. \Halmos
\end{proof}

\subsection*{Proof of Proposition \ref{pro3}}
\begin{proof}{Proof of Proposition \ref{pro3}}
We seek a function $\rho_{\delta} \in {\cal L}(\{\vec{X}_j\}_{j\in {\cal J}}, {\cal C})$ that further satisfies the system below:
\begin{eqnarray}
 && x^{t} \in \arg\min_{x\in{\cal X}^{t}}\rho_{\delta}(\vec{Z}^t(x)), \; \forall t \in {\cal T}, \label{optcon}\\
 && \rho_{\delta}(\vec{L}_{k})\leq\rho_{\delta}(\vec{U}_{k}),\;\forall k\in {\cal K} \label{laweli}.
\end{eqnarray}

Due to the monotonicity of $\rho_{\delta} \in {\cal R}$, we can equivalently write the optimization problem in (\ref{optcon}) as 
$$\min_{(x,\vec{W})\in \Pi^{t}} \rho_{\delta}(\vec{W}), \text{ where } \Pi^{t} : = \left\{(x,\vec{W})\; \middle |\;\vec{W} \geq \vec{Z}^t(x),\; x \in {\cal X}^{t} \right\}.$$
Because $\Pi^{t}$ is convex and $\vec{W}^{t} \in \{\vec{X}_j\}_{j \in {\cal J}}$, we can characterize the optimality condition for the problem; namely,  there must exist a subgradient $y\in\partial\rho_{\delta}(\vec{W}^{t})$
such that
\[
y^{\top}(\vec{W}-\vec{W}^{t})\geq0,\;\forall (x,\vec{W}) \in \Pi^{t} \Leftrightarrow 
y^{\top}\vec{W}^{t} \leq\min_{x \in {\cal X}^{t}}y^{\top}\vec{Z}^t(x).
\]
Moreover, applying the theory of conjugate duality (Theorem \ref{rock}) (2.), which states that 
\[
\partial\rho_{\delta}(\vec{W}^{t})=\partial\rho_{\delta}^{**}(\vec{W}^{t})=\arg\max_{y} \left\{y^{\top}\vec{W}^{t}-\rho_{\delta}^{*}(y)\right\},
\]
we can further characterize the set of subgradients $\partial\rho_{\delta}(\vec{W}^{t})$ in terms of inequality constraints, given that 
\begin{eqnarray}
 && y \in \partial\rho_{\delta}^{**}(\vec{W}^{t}) \nonumber \\
\Leftrightarrow \{ y: && y^{\top}\vec{W}^{t}-\rho_{\delta}^{*}(y)\geq\rho_{\delta}(\vec{W}^{t}) \}  \nonumber\\
\Leftrightarrow \{ y: && y^{\top}\vec{W}^{t}-\max_{j\in {\cal J}} \left\{y^{\top}\vec{X}_{j}-\rho_{\delta}(\vec{X}_{j}) \right\}  \geq \rho_{\delta}(\vec{W}^{t}),\;\; y \in {\cal C} \}  \nonumber\\
\Leftrightarrow \{ y: && y^{\top}\vec{W}^{t}-y^{\top}\vec{X}_{j}+\rho_{\delta}(\vec{X}_{j})\geq\rho_{\delta}(\vec{W}^{t}),\; \forall j\in {\cal J},\;\;y \in {\cal C} \}, \label{lawper}
\end{eqnarray}
where in the third line the conjugate $\rho_{\delta}^{*}(y)$ derived in Lemma \ref{lemma1} is applied.

Finally, applying Proposition \ref{pro2}, we can equivalently describe the set of risk functions $\rho_{\delta}$ by the system of constraints (\ref{eq:linearsystema}). Thus by replacing 
$\rho_{\delta}(\vec{X}_{j})$ with $\delta_j$ in (\ref{laweli}) and (\ref{lawper}), we arrive at the final formulation. 
\Halmos
\end{proof}

\subsection*{Proof of Theorem \ref{2ndmain}}
\begin{proof}{Proof of Theorem \ref{2ndmain}}
To complete the proof, we need only to show that given that $\rho_{\delta} \in {\cal L}(\{\vec{X}_j\}_{j\in {\cal J}}, {\cal C})$ the objective function $||\rho_{\delta}-\tilde{\rho}||_{\infty}$ can be reduced to 
\begin{equation} \label{maxx}
\max_{j\in {\cal J}} \left\{|\rho_{\delta}(\vec{X}_{j})-\tilde{\rho}(\vec{X}_{j})| \right\}.
\end{equation}
To confirm this, let $u^* = (\ref{maxx})$. Obviously by the definition of $u^*$, we can also write
$$ -u^* \leq {\rho_{\delta}}(\vec{X}_j) - \tilde{\rho}(\vec{X}_j) \leq  u^*,\;\forall j \in {\cal J}. $$
The direction $||\rho_{\delta}-\tilde{\rho}||_{\infty} \geq u^*$ is clear. To prove the other direction  (i.e., $||\rho_{\delta}-\tilde{\rho}||_{\infty} \leq u^*$), note first that because $\rho_{\delta} \in {\cal R}$, we can apply the theory of conjugate duality (Theorem \ref{rock} (1.)) and use the conjugate $\rho_{\delta}^{*}$ derived in Lemma \ref{lemma1} to write $\rho_{\delta}$ as 
\[\rho_{\delta}(\vec{Z}) = \rho_{\delta}^{**}(\vec{Z})=\sup_{y\in{\cal C}}y^{\top}\vec{Z}-\max_{j \in {\cal J}} \left\{y^{\top}\vec{X}_j-\rho_{\delta}(\vec{X}_j)\right\}.\]
Based on this representation, we can then derive 
\begin{eqnarray}
\rho_{\delta}(\vec{Z}) && = \sup_{y\in{\cal C}}y^{\top}\vec{Z}-\max_{j \in {\cal J}} \left\{y^{\top}\vec{X}_j-\rho_{\delta}(\vec{X}_j)\right\} \nonumber \\
&&  \leq  \sup_{y\in {\cal C}} y^\top \vec{Z} - \max_{j\in {\cal J}} \left\{ y^\top \vec{X}_j - (\tilde{\rho}(\vec{X}_{j})+u^{*})\right\} \;\;(\text{since }{\rho_{\delta}}(\vec{X}_j) \leq \tilde{\rho}(\vec{X}_j) + u^*,\;\forall j \in {\cal J})
\nonumber \\ 
&& \leq  \sup_{y\in {\cal C}} y^\top \vec{Z} - (y^\top \vec{0} - (\tilde{\rho}(\vec{0})+ u^{*})) = \tilde{\rho}(\vec{Z})+u^{*},\; \forall \vec{Z}, \label{inqq11} 
\end{eqnarray}
and also have
\begin{eqnarray}
\rho_{\delta}(\vec{Z}) && = \sup_{y\in{\cal C}}y^{\top}\vec{Z}-\max_{j \in {\cal J}} \left\{y^{\top}\vec{X}_j-\rho_{\delta}(\vec{X}_j)\right\} \nonumber \\
&& \geq   \sup_{y\in {\cal C}} y^\top \vec{Z} - \max_{j\in {\cal J}} \left\{ y^\top \vec{X}_j - (\tilde{\rho}(\vec{X}_{j}) -u^{*})\right\} \;\;(\text{since }{\rho_{\delta}}(\vec{X}_j) \geq \tilde{\rho}(\vec{X}_j) - u^*,\;\forall j \in {\cal J}) \nonumber \\
&& =  \sup_{y\in {\cal C}} y^\top \vec{Z} - u^*- \max_{j\in {\cal J}} \left\{ y^\top \vec{X}_j - \tilde{\rho}(\vec{X}_{j})\right\} \;\; \nonumber \\
&&\geq \sup_{y\in {\cal C}} y^\top \vec{Z} -u^{*} = \tilde{\rho}(\vec{Z})-u^{*},\; \forall \vec{Z}, \label{inqq12}
\end{eqnarray}
where the second inequality follows the supremum representation of $\tilde{\rho}$; namely 
for any $j \in {\cal J}$,
\begin{equation}
 \tilde{\rho}(\vec{X}_j) = \sup_{y\in {\cal C}} y^\top \vec{X}_j \Leftrightarrow  y^\top \vec{X}_j - \tilde{\rho}(\vec{X}_j) \leq 0,\; \forall y \in {\cal C}. \nonumber
\end{equation}

We thus conclude that to determine the parameter $\delta$, it suffices to solve the problem (\ref{final}). Since the problem is a convex program, we can apply the famous result of \cite{M.-Grotschel:1981aa}, which states that for a convex program like (\ref{final}), it can be solved in polynomial time by using the ellipsoid method if and only if for any $z^* := (\delta^*, y_j^*)$ it takes polynomial time to either confirm that $z^*$ is in the feasible set ${\cal Z}$ or generate a hyperplane that separates $z^*$ from ${\cal Z}$. Hence, if the function $h_t(y_{t})$ can be evaluated in polynomial time  (i.e., the forward problem can be solved in polynomial time), and the oracle for the set ${\cal C}$ exists, it can be shown fairly straightforwardly that it also takes polynomial time to confirm $z^* \in {\cal Z}$ or  separate $z^*$ from ${\cal Z}$.  This completes the proof. \Halmos \end{proof}

\subsection*{Proof of Corollary \ref{linprog}}
\begin{proof}{Proof of Corollary \ref{linprog}}
By introducing a dummy variable $s$ that bounds from above $\max_{j\in{\cal J}}\left\{ p^\top \vec{X}_j - \delta_j\right\}$, we have the first formulation. The second formulation with the linear constraints is the dual of the first formulation that can be 
obtained by applying linear duality theory. Note that strong duality always holds here because,  obviously, there always exists a feasible solution to the first formulation (e.g., by setting $s^* =\max_{j\in{\cal J}}\left\{ {p^*}^\top \vec{X}_j - \delta_j\right\}$ for any feasible $p^* \in {\cal C}$).
\Halmos
\end{proof}

\subsubsection*{Proof of Proposition \ref{prop_final}}
As mentioned in Section \ref{permrisk}, in the case of permutation invariance one must deal with the issue that the optimality conditions are defined over non-convex sets $\left\{ \sigma(\vec{W})\;\middle | \sigma \in \Sigma,\; \vec{W} \in {\cal W}^t \right\}$ and the issue that the size of $\left\{ \sigma(\vec{X}_j) \right\}_{\sigma \in \Sigma, j \in {\cal J}}$ grows exponentially. 

We first show how the inverse problem (\ref{uniforminv})--(\ref{inv2}) with ${\cal R}:=\bars{{\cal R}}$ can still be formulated as a convex program by following closely the steps established in Section \ref{gg}. This bypasses the issue of non-convexity. We summarize this first step in the following proposition. After proving this proposition, we then continue to prove how to reduce the size of the problem presented in the proposition from exponentially many to polynomially many.

\begin{proposition} \label{pro9}
Given that Assumption \ref{asperm} holds and that the set of optimal solutions is non-empty, the inverse problem (\ref{uniforminv})--(\ref{inv2}) with ${\cal R}:=\bars{{\cal R}}$ can by solved a risk function $\bars{\rho}_{\delta} \in \bars{{\cal L}}(\{\vec{X}_j\}_{j \in {\cal J}}, \bars{{\cal C}})$, and the parameter $\delta$ is calculated by solving 
\begin{eqnarray}
\min_{u,\delta,y_{\sigma,j}} && u \nonumber\\
 {\rm subject\; to}&& -u\leq\delta_{j}-\tilde{\rho}(\vec{X}_{j})\leq u, \;\;\; \forall j\in {\cal J}, \nonumber \\
 && \delta_{j}+y_{\sigma,j}^{\top}(\sigma'(\vec{X}_{i})-\sigma(\vec{X}_{j}))\leq\delta_{i} , \;\;\; \forall (\sigma,j)\in\Sigma\times {\cal J},\; \forall i\ne j,\;\forall\sigma'\in\Sigma,\nonumber \\
 &&  y_{\sigma,j} \in \bars{{\cal C}} ,  \;\;\; \forall (\sigma,j)\in\Sigma\times {\cal J}, \label{eq:second} \\
 && y_{\sigma*, t}^{\top}\vec{X}_{t}\leq h_t(y_{\sigma^*, t}),  \;\;\; \forall t \in {\cal T}, \nonumber \\
 && \delta_{i}\leq\delta_{j} ,  \;\;\; \forall(i,j)\in{\cal B},\nonumber 
\end{eqnarray}
where $\sigma^*$ in $y_{\sigma*, t}$ corresponds to the permutation such that $\sigma^*(\vec{X}) = \vec{X}$; 
$u \in \Re$, $\delta \in \Re^{|{\cal J}|}$, $ y_{\sigma,j} \in \Re^{M}$; the set
${\cal B}:= \left\{(i,j)\in\{1,2,...,{\cal J}\}^{2}\;\middle |\;(\vec{X}_{i},\vec{X}_{j})\in\{(\vec{L}_{k},\vec{U}_{k})\}_{k\in {\cal K}}\right \}$;
and $h_t$ denotes the function $h_t(y):=\min_{x} \left\{ y^\top \vec{Z}^t(x) \;\middle |\;x\in{\cal X}^{t}\right\}$. \end{proposition}

\begin{proof}{Proof of Proposition \ref{pro9}}
{\bf (Step 1)} Like Proposition \ref{pro4}, we claim that the set $\bars{{\cal L}}(\{\vec{X}_j\}_{j\in {\cal J}}, \bars{{\cal C}})$ contains an optimal solution to the problem. We prove this by showing that if there exists a permutation-invariant risk function $\bars{\rho}_0$ that is optimal for the  problem with the optimal value $u^* < \infty$, there must exist a risk function $\bars{\rho}_{\delta} \in \bars{{\cal L}}(\{\vec{X}_j\}_{j\in {\cal J}}, \bars{{\cal C}})$ that is also optimal to the problem, namely by setting 
\[ \delta_j =  \bars{\rho}_0(\vec{X}_j). \]

This can be proved by following exactly the steps in Proposition \ref{pro4}. Namely, by letting  
$\{\vec{X}_j\}_{j\in {\cal J}}$ in ${\cal L}(\{\vec{X}_j\}_{j\in {\cal J}}, {\cal C})$ of Proposition \ref{pro4} now represent all random losses in $\{\sigma(\vec{X}_j)\}_{\sigma\in \Sigma, j\in {\cal J}}$ and letting ${\cal C}$ now take the form of $\bars{{\cal C}}$, we can conclude based on the result in Proposition \ref{pro4} that the following function $\rho_{\delta} \in {\cal L}(\{\sigma(\vec{X}_j)\}_{\sigma \in \Sigma j\in {\cal J}}, \bars{{\cal C}})$ must also be optimal: 
\[\rho_{\delta}(\vec{Z}) :=  \sup_{p \in \bars{{\cal C}}}  p^{\top}\vec{Z} - \max_{\sigma\in \Sigma, j\in {\cal J}} \left\{p^{\top}\sigma(\vec{X}_{j})- \bars{\rho}_0 (\sigma(\vec{X}_j)) \right\}. \]
Note that due to the fact that $\bars{\rho}_0$ is permutation invariant, the above function can be simplified to 
\[\rho_{\delta}(\vec{Z}) =  \sup_{p \in \bars{{\cal C}}}  p^{\top}\vec{Z} - \max_{\sigma\in \Sigma, j\in {\cal J}} \left\{p^{\top}\sigma(\vec{X}_{j})- \bars{\rho}_0 (\vec{X}_j) \right\} = \bars{\rho}_{\delta}(\vec{Z}). \]
This completes the first step of the proof. \\
\\
{\bf (Step 2)} Following Proposition \ref{pro2}, we can boil down the problem of search over the set $\bars{{\cal L}}(\{\vec{X}_j\}_{j\in {\cal J}}, \bars{{\cal C}})$ into identifying feasible solutions $y_{\sigma,j}$ and $\delta_j$ in the following systems:
\begin{eqnarray}
 && \delta_{j}+y_{\sigma,j}^{\top}(\sigma'(\vec{X}_{i})-\sigma(\vec{X}_{j}))\leq\delta_{i},\;  \forall i, j \in {\cal J}, \; \forall \sigma, \sigma' \in \Sigma, \label{eq:linearsystem2}\\
  && y_{\sigma,j} \in \bars{{\cal C}},\; \forall j\in {\cal J},\; \forall \sigma \in \Sigma. \nonumber 
\end{eqnarray} 
\\
{\bf (Step 3)} We will follow closely the steps in Proposition \ref{pro3}. However, one can see that because we seek only a function $\bars{\rho}_{\delta} \in \bars{{\cal L}}(\{\vec{X}_j\}_{j\in {\cal J}}, \bars{{\cal C}})$, which is permutation invariant, we can reduce the constraints 
\begin{eqnarray}
 && x^{t} \in \arg\min_{x\in{\cal X}^{t}}\bars{\rho}_{\delta}(\sigma(\vec{Z}^t(x))),\; \forall t \in {\cal T}, \; \forall \sigma \in \Sigma, \label{optcon11}\\
 && \bars{\rho}_{\delta}(\sigma'(\vec{L}_{k}))\leq\bars{\rho}_{\delta}(\sigma(\vec{U}_{k})),\; \forall k\in {\cal K},\; \forall \sigma' , \sigma \in \Sigma \label{laweli11}
\end{eqnarray}
into 
\begin{eqnarray}
 && x^{t} \in \arg\min_{x\in{\cal X}^{t}}\bars{\rho}_{\delta}(\vec{Z}^t(x)),\; \forall  t \in {\cal T},  \label{optcon22}\\
 && \bars{\rho}_{\delta}(\vec{L}_{k})\leq\bars{\rho}_{\delta}(\vec{U}_{k}),\; \forall k\in {\cal K}. \label{laweli22}
\end{eqnarray}

We will follow the same arguments used in Proposition \ref{pro3} to further reduce the constraint (\ref{optcon22}). Note that although it looks quite repetitive, one should be careful to derive it based on structure of $\bars{\rho}_{\delta}$. First, the monotonicity of $\bars{\rho}_{\delta}$ allows us to equivalently write the constraint (\ref{optcon22}) as 

$$\min_{(x,\vec{W})\in \Pi^{t}} \bars{\rho}_{\delta}(\vec{W}), \text{ where } \Pi^{t} : = \left\{(x,\vec{W})\;\middle |\;\vec{W} \geq \vec{Z}^t(x),\; x \in {\cal X}^{t}\right\}.$$

Because $\Pi^{t}$ here is convex and $\vec{W}^{t} \in \{\vec{X}_j\}_{j \in {\cal J}}$, we can now write down the optimality condition that there must exist a subgradient $y\in\partial\bars{\rho}_{\delta}(\vec{W}^{t})$ such that 
\[
y^{\top}(\vec{W}-\vec{W}^{t})\geq0,\;\forall (x,\vec{W}) \in \Pi^{t} \Leftrightarrow 
y^{\top}\vec{W}^{t} \leq\min_{x \in {\cal X}^{t}}y^{\top}\vec{Z}^t(x).
\]

Because $\bars{\rho}_{\delta}$ is obviously convex, we again apply the theory of conjugate duality (Theorem \ref{rock}) (2.), which states that 
\[
\partial\bars{\rho}_{\delta}(\vec{W}^{t})=\partial\bars{\rho}_{\delta}^{**}(\vec{W}^{t})=\arg\max_{y} \left\{y^{\top}\vec{W}^{t}-\bars{\rho}_{\delta}^{*}(y)\right\}.
\]
We now further characterize the set of subgradients $\partial\bars{\rho}_{\delta}(\vec{W}^{t})$ by
\begin{eqnarray}
 && y \in \partial\bars{\rho}_{\delta}^{**}(\vec{W}^{t}) \nonumber \\
\Leftrightarrow\{y: && y^{\top}\vec{W}^{t}-\bars{\rho}_{\delta}^{*}(y)\geq\bars{\rho}_{\delta}(\vec{W}^{t})\}  \nonumber\\
\Leftrightarrow\{y: && y^{\top}\vec{W}^{t}-\max_{\sigma \in \Sigma, j\in {\cal J}} \left\{y^{\top}\sigma(\vec{X}_{j})-\bars{\rho}_{\delta}(\sigma(\vec{X}_{j}))\right\}\geq \bars{\rho}_{\delta}(\vec{W}^{t}),\; y \in \bars{{\cal C}}\}  \nonumber\\
\Leftrightarrow\{y: && y^{\top}\vec{W}^{t}-y^{\top}\sigma(\vec{X}_{j})+\bars{\rho}_{\delta}(\sigma(\vec{X}_{j}))\geq\bars{\rho}_{\delta}(\vec{W}^{t}),\; \forall j\in {\cal J},\; \forall \sigma \in \Sigma, \; \forall y \in \bars{{\cal C}}\}, \label{lawper2}
\end{eqnarray}
where in the third line the conjugate $\bars{\rho}_{\delta}^{*}(y)$ derived in Lemma \ref{lemma1} is applied.

{\bf (Step 4)} Now, following the proof of Theorem \ref{2ndmain}, it is straightforward to see that the objective function in this case can be reduced to
\begin{equation} \label{maxxx}
\max_{\sigma \in \Sigma, j\in {\cal J}} \left\{|\bars{\rho}_{\delta}(\sigma(\vec{X}_{j}))-\tilde{\rho}(\vec{X}_{j})| \right\}.
\end{equation}

Moreover, in applying Step 2 we can equivalently describe the set of risk functions $\bars{\rho}_{\delta}$ by the system of constraints (\ref{eq:linearsystem2}). Thus by replacing 
$\bars{\rho}_{\delta}(\sigma(\vec{X}_{j}))$ with $\delta_j$ in (\ref{laweli22}), (\ref{lawper2}) and (\ref{maxxx}), we arrive at the final formulation. \Halmos
\end{proof}

\vspace{10pt}
{\bf (Continue the proof of Proposition \ref{prop_final})} Now, we show how to reduce the optimization problems presented Proposition \ref{pro9}
(including the problem in the definition of $\bars{{\cal L}}(\{\vec{X}_j\}_{j \in {\cal J}}, \bars{{\cal C}})$). We consider first the reduction of the problem (\ref{eq:second}). Note first that the constraints associated with
$y_{\sigma*, t}$ can be equivalently written as, with $y_{\sigma*, t}$ replaced by $y_t$,
\begin{eqnarray}
  && y_{t}^{\top}\vec{X}_{t}\leq h_t(y_{t}),\; \forall t \in {\cal T},  \nonumber \\
  && \delta_{t}+y_{t}^{\top}(\sigma(\vec{X}_{i})-\vec{X}_{t})\leq\delta_{i}, \; \forall t \in {\cal T},\; \forall i\in {\cal J}\setminus\{t\}, \; \forall\sigma\in\Sigma,  \label{dd1}\\
  && y_t\in  \bars{{\cal C}}, \; \forall t \in {\cal T}. \nonumber 
\end{eqnarray}

We show that the other constraints, namely the second and the third constraint in (\ref{eq:second}) in general can also be reduced to
\begin{eqnarray}
 && \delta_{j}+y_{j}^{\top}(\sigma'(\vec{X}_{i})-\vec{X}_{j})\leq\delta_{i},\;  \forall j\in {\cal J},\;\forall i\ne j,\;\forall\sigma'\in\Sigma,\label{eq:reduction1}\\
 && y_j \in \bars{{\cal C}},\;   \forall j\in {\cal J}. \nonumber 
\end{eqnarray}
We prove this by showing that given any feasible solution $(u^*, \delta^{*},y_{\sigma,j})$
of the problem (\ref{eq:second}) we can always construct a
feasible solution $(u^*, \delta^{*},\bar{y}_{\sigma'',j})$ with
$\bar{y}_{\sigma'',j}$ satisfying 
\[
\bar{y}_{\sigma'',j}:=\frac{1}{|\Sigma|}\sigma''(\sum_{\sigma\in\Sigma}\sigma{}^{-1}(y_{\sigma,j})),\;\forall\sigma''\in\Sigma,
\]
which gives the same optimal value $u^*$. 

To verify its feasibility for the second constraint of (\ref{eq:second}), by substitution we have 
$$
 \bar{y}_{\sigma'',j}^{\top}(\sigma'(\vec{X}_{i})-\sigma''(\vec{X}_{j}))
= \frac{1}{|\Sigma|}(\sum_{\sigma\in\Sigma}y_{\sigma,j}^{\top}(\sigma(\sigma''^{-1}(\sigma'(\vec{X}_{i})))-\vec{X}_{j})
\leq \frac{1}{|\Sigma|}(\sum_{\sigma\in\Sigma}(\delta_{i}-\delta_{j})) =\delta_{i}-\delta_{j},
$$
where the last inequality is due to the feasibility of $y_{\sigma,j}$. For the third constraint of (\ref{eq:second}), feasibility can be verified as follows. Since $y_{\sigma,j}\in \bars{{\cal C}}$, we have $\sigma^{-1}(y_{\sigma,j}) \in \bars{{\cal C}}$. We also have $\sum_{\sigma \in \Sigma} \frac{1}{|\Sigma|} \sigma^{-1}(y_{\sigma,j}) \in \bars{{\cal C}}$ because the summation is a convex combination and the set $\bars{{\cal C}}$ is convex. Given this, we also have $\bar{y}_{\sigma'',j} \in \bars{{\cal C}}$ by the definition of $\bars{{\cal C}}$. Hence, we can replace $y_{\sigma,j}$ by $\sigma(y_{j})$ for some
$y_{j}\in\Re^{M}$ in the second and third constraints in (\ref{eq:second})
and arrive at the reduction (\ref{eq:reduction1}).

Now, both the constraint (\ref{dd1}) and (\ref{eq:reduction1}) can be re-arranged into the following general
form 
\begin{equation}
y_{j}^{\top} \sigma(\vec{X}_{i})\leq\delta_{i}-\delta_{j}+y_{j}^{\top}\vec{X}_{j},\; \forall\sigma\in\Sigma,\; \forall  j \in {\cal J}. \label{star}
\end{equation}

We show in general how the constraint in the form of $y_{j}^{\top}\sigma(\vec{X})\leq b,\;\forall\sigma\in\Sigma$ for some $\vec{X}$ and $b$
can be reduced, which can then be applied to reduce (\ref{dd1}) and (\ref{eq:reduction1}). Recall first that a permutation matrix $Q_{\sigma}$ is a matrix that satisfies 
$\sigma(\vec{X}) = Q_\sigma \vec{X}$ and $(Q_{\sigma})_{m,n} \in \{0,1\}$ and $Q_{\sigma}^{\top}\vec{1}=\vec{1},\;Q_{\sigma}\vec{1}=\vec{1}$. Hence, $y_{j}^{\top}\sigma(\vec{X})\leq b,\;\forall\sigma\in\Sigma$ can be re-written as $\max_{\sigma\in\Sigma}y_{j}^{\top}Q_{\sigma}\vec{X}\leq b$ and also as 
$$\max_{ Q \in \text{Conv}(\{Q_{\sigma}\}_{\sigma\in \Sigma}) } y_{j}^{\top}Q \vec{X}\leq b.$$
Applying the result of \cite{birkhoff:tola}, we can reformulate the convex hull of all permutation matrices into linear constraints and arrive at the following formulation:
\[
\max_{Q}\; \left\{y_{j}^{\top}Q\vec{X}\; \middle |\;Q^{\top}\vec{1}=\vec{1},\;Q\vec{1}=\vec{1},\;Q \in \Re_+^{M\times M} \right\}\leq b.
\]

By deriving the dual problem of the above linear program, we have 
\begin{eqnarray*}
 & \min_{v,w} \left\{\vec{1}^{\top}v+\vec{1}^{\top}w\; \middle |\;\vec{X}y^{\top}-v\vec{1}^{\top}-\vec{1}w^{\top} \leq 0 \right\} &\leq b\\
\Leftrightarrow & \exists v,w\;:\;\vec{1}^{\top}v+\vec{1}^{\top}w\leq b,\;\vec{X}y^{\top}-v\vec{1}^{\top}-\vec{1}w^{\top}\leq0.
\end{eqnarray*}
Strong duality holds for the above linear programs because there always exists a permutation matrix satisfying the above constraints. We apply this dualization procedure to (\ref{star}), which leads to the final formulation of the problem (\ref{mod2}). 

Finally, we consider the reduction of the problem in the definition of $\bars{{\cal L}}(\{\vec{X}_j\}_{j \in {\cal J}}, \bars{{\cal C}})$. The optimization problem can be equivalently formulated as 
\begin{eqnarray}
     \sup_{p \in \bars{{\cal C}},t}  && p^\top \vec{Z} - t  \nonumber \\
     {\rm subject \ to} && p^\top \sigma(\vec{X}_j)  \leq t +  \delta_j ,\;\;\;\; \forall \sigma \in \Sigma, \; \forall j \in {\cal J}. \label{last2}
\end{eqnarray}
We can apply the same dualization procedure above to reduce again the constraint (\ref{last2}), which leads to the final formulation. \Halmos

\subsection*{Proof of Proposition \ref{lastpros}}
\begin{proof}{Proof of Proposition \ref{lastpros}}
Given that Assumption \ref{rational} holds, we can always convert the probability values specified in the distribution of $Z(x,\xi^t)$ (that satisfies Assumption \ref{ass_obj}) and the set of distributions $\{F_j\}_{j\in {\cal J}}$ to ratios in the form of $n/M$ for some fixed $M \in \mathbb{Z}^+$ and $n \in \{1,...,M\}$. By considering an outcome space with $M$ uniformly distributed outcomes, we can equivalently define the random loss $\vec{Z}(x)$ as a mapping from $\Omega:=\{\omega_i\}_{i=1}^M$ to $\Re$ that satisfies $Z(x,\xi(\omega_i)) \in \{Z(x,\xi_o)\}_{o=1}^{\tau_0}$ and $| \left\{\omega_i \;|\; Z(x,\xi(\omega_i))=Z(x,\xi_o) \right\} |=\bar{p}^{\xi}_o M$, and similarly $\vec{X}_j$ as a mapping $X_j : \Omega \rightarrow \Re$ that satisfies 
$X_j(\omega_i) \in \{(\vec{S}_j)_{o}\}_{o=1}^{\tau_j}$ and $|\{\omega_i\;|\; X_j(\omega_i)=(\vec{S}_j)_{o} \}|=\bar{p}^{j}_o M$. 

Suppose now that the optimization problems  (\ref{mod2}) and (\ref{mod111}) are formulated based on the above definition of random losses. We show in what follows how the problems can be further reduced. Note first that by replacing the objective function $p^\top \vec {Z} - t$ in (\ref{mod111}) with a new variable $s$, we can reformulate its first constraint into
$$ \vec{1}^{\top}v_j+\vec{1}^{\top}w_j - p^\top \vec {Z} \leq \delta_j -s ,\; \forall j\in {\cal J}. $$
We can then write all the constraints on the variable $p$ in (\ref{mod111}) by 
\begin{equation}
p  \in {\cal G}(\vec{Z},s, \{\delta_i\}_{i \in {\cal J}}) \label{mm3},
\end{equation}
where 
${\cal G}(\vec{Y}, t, \{\delta_i\}_{i\in {\cal J}})$ is a parameterized set represented by the following system of constraints on $y$: $\exists v_i, w_i$ such that 
\begin{eqnarray}
&& \vec{1}^{\top}v_i+\vec{1}^{\top}w_i - y^\top \vec {Y} \leq \delta_i - t, \; \forall i \in {\cal J} ,\label{ver1} \\
&& \vec{X_i}y^{\top}-v_i\vec{1}^{\top}-\vec{1}w_i^{\top}\leq 0, \; \forall i \in {\cal J}, \label{ver2} \\
&& y \in \bars{{\cal C}}. \label{ver3} 
\end{eqnarray}

Given any fixed $\{\delta_j^*\}_{j\in {\cal J}}$, the constraints in the optimization problem (\ref{mod2}) can also be equivalently written as 
\begin{eqnarray}
  && y_t \in  {\cal G}(\vec{X}_t, \delta_t^*, \{\delta_i^*\}_{i\in {\cal J}})  \cap  \left\{y\; \middle |\; y^\top \vec{X}_t \leq h_t(y) \right\}, \; \forall t \in {\cal T}, \label{mm} \\
  && y_j  \in {\cal G}(\vec{X}_j,\delta_j^*, \{\delta_i^*\}_{i\in {\cal J}}),   \; \forall j \in {\cal J} \setminus {\cal T}. \label{mm2} 
\end{eqnarray}

We present only the reduction of the constraints (\ref{mm}) with (\ref{ver1})--(\ref{ver3}), given that  the same steps can be applied to reduce the constraints (\ref{mm2}) (with (\ref{ver1})--(\ref{ver3})) and (\ref{mm3}) (with (\ref{ver1})--(\ref{ver3})). 

Because it suffices to consider (\ref{mm}) for any fixed $t$, from here on we consider only $t=1$ and drop the index $t$ for the variables to simplify the presentation.
Given a fixed set of $\{\delta_j^*\}_{j\in {\cal J}}$, let $y^*$, $v_i^*$, $w_i^*$ denote a feasible solution that satisfies (\ref{mm}) and (\ref{ver1})--(\ref{ver3}). For $o=1,...,\tau_0$, let ${\cal I}^{(1)}_o$ denote the set of indices $n$ of $\vec{X_1}$ such that $(\vec{X_1})_n = (\vec{S}_1)_o$, and therefore $|{\cal I}^{(1)}_o| = \bar{p}^{(1)}_oM$. 
We claim that the solution $v_i^*$ together with the following $y^{**} \in \Re^M$, $w_i^{**}\in \Re^M$ that satisfies for any $n \in {\cal I}^{(1)}_o$,
$$ (y^{**})_n = \frac{1}{|{\cal I}^{(1)}_o |} \sum_{a \in {\cal I}^{(1)}_o }  (y^*)_a, \;\;\;\; 
(w^{**}_i)_n = \frac{1}{|{\cal I}^{(1)}_o |} \sum_{a \in {\cal I}^{(1)}_o }  (w^*_i)_a, $$
$o=1,...,\tau_0$, will also satisfy (\ref{mm}) and (\ref{ver1})--(\ref{ver3}).

The constraint (\ref{mm}) and (\ref{ver1})--(\ref{ver2}) can be verified fairly straightforwardly by the direct substitution. To verify the third constraint (\ref{ver3}), we construct a sequence of solutions $y_*^{(1)},...,y_*^{(\tau_0)}$ that satisfies $y_*^{(\tau_0)} = y^{**}$ and that $y_*^{(\tau_0)} \in \bars{{\cal C}}$. For $o=1,...,\tau_0$, let $\Sigma_o$ denote the set of all permutation operators $\sigma$ that satisfy $(\sigma(y^*))_a = (y^*)_a,$ $\forall a \notin {\cal I}^{(1)}_o$. In other words, the set consists of all permutations that permute only the entries $a \in {\cal I}^{(1)}_o$. Set $o=1$ and we construct $y_*^{(1)}$ by $y_*^{(1)} := \sum_{\sigma \in \Sigma_1} \frac{1}{|{\cal I}^{(1)}_o|!} \sigma(y^*)$. One can confirm that $y_*^{(1)}$ satisfies
$$ (y_{*}^{(1)})_{\tilde{n}} = \frac{1}{|{\cal I}^{(1)}_1 |} \sum_{a \in {\cal I}^{(1)}_1 }  (y^*)_a$$
for ${\tilde{n}} \in {\cal I}^{(1)}_1$ and $(y_{*}^{(1)})_{\tilde{n}} = (y^*)_{\tilde{n}}$ otherwise. Given that $y^* \in \bars{{\cal C}}$, we must have $\sigma(y^*) \in \bars{{\cal C}}$ by the definition of $\bars{{\cal C}}$ and therefore $y_*^{(1)} \in \bars{{\cal C}}$ because the summation is a convex combination. For $o \geq 2$, we can construct $y_*^{(o)} = \sum_{\sigma \in \Sigma_{o}}  \frac{1}{|{\cal I}^{(1)}_o|!}\sigma(y_*^{(o-1)})$. If $y_*^{(o-1)} \in \bars{{\cal C}}$, $y_*^{(o)} \in \bars{{\cal C}}$ must hold and $y_*^{(o)}$ satisfies that for any ${\tilde{n}} \in {\cal I}_{o'}^{(1)}$, $o'=1,...,o$, 
$$ (y_{*}^{(o)})_{\tilde{n}} = \frac{1}{|{\cal I}^{(1)}_{o'} |} \sum_{a \in {\cal I}^{(1)}_{o'} }  (y^*)_a,$$
and $(y_{*}^{(o)})_{\tilde{n}} = (y^*)_{\tilde{n}}$ otherwise. By induction, $y_*^{(\tau_0)} \in \bars{{\cal C}}$ and $y_{*}^{(\tau_0)} = y^{**}$.

Hence, we can reduce the constraints (\ref{mm}) and (\ref{ver1})--(\ref{ver3}) by imposing for some $\tilde{y} \in \Re^{\tau_0}$, $\tilde{w} \in \Re^{\tau_0}$
that for $a \in {\cal I}^{(1)}_o$, $o=1,...,\tau_0$, $(y)_a = (\tilde{y})_o$ and $(w_i)_a = (\tilde{w}_i)_o$. This leads to the following constraints
\begin{eqnarray}
        &&\vec{1}^\top v_i + 1^\top (\lambda_1 \circ \tilde{w}_i) - (\lambda_1\circ \tilde{y})^\top \vec{S}_1 \leq \delta_i - \delta_1, \label{fir}\\
        && \vec{X}_i \tilde{y}^\top - v_i \vec{1}^\top - \vec{1}\tilde{w}_i^\top \leq 0, \label{ind}\label{finn}\\
        && {\cal H}_{F_1}(\tilde{y}) \in \bars{{\cal C}}, \label{finn2}\\
        && (\lambda_1\circ \tilde{y})^\top \vec{S}_1 \leq \min_{x\in {\cal X}} (\lambda_1\circ \tilde{y})^\top \vec{Z}'(x), \label{finn3}
\end{eqnarray}
where $\lambda_1:=(|{\cal I}^{(1)}_1|,...,|{\cal I}^{(1)}_{\tau_0}|)$ and $\vec{Z}'(x):=(Z(x,\xi_1),...,Z(x,\xi_{\tau_0}))^\top$.

We now show that the above four constraints can be further reduced. Let ${\cal I}_o^{(i)}$ denote the set of indices $n$ of $\vec{X}_i$ such that $(\vec{X}_i)_n = (\vec{S}_i)_o$,  $o=1,...,\tau_i$,  and therefore $|{\cal I}_o^{(i)}| = \bar{p}_o^i M$. 

It is not difficult to see that for any $(v_i)_a$ such that $a\in {\cal I}_o^{(i)}$ the constraints associated with $(v_i)_a$ are identical in (\ref{ind}). Because reducing $(v_i)_a$, for any $a$, is always feasible for (\ref{fir}), if there exists any $(v^*_i)_a \neq (v^*_i)_b$ for $a$, $b \in {\cal I}_o^{(i)}$, we can always make them equal by reducing the larger one (without violating any constraint). We can thus conclude that we can always impose for some $\tilde{v}_i \in \Re^{\tau_i}$ that $(v_i)_a = (\tilde{v}_i)_o$ for any $a \in {\cal I}_o^{(i)}$. This leads to the reformulation of the first constraint (\ref{fir}) into
$$ \vec{1}^\top (\lambda_i \circ \tilde{v}_i) + 1^\top (\lambda_1 \circ \tilde{w}_i) - (\lambda_1 \circ \tilde{y})^\top \vec{S}_1 \leq \delta_i - \delta_1, $$
where $\lambda_i:=(|{\cal I}_1^{(i)}|,...,|{\cal I}_{\tau_i}^{(i)}|)$ and (\ref{finn}) into 
$$\vec{S}_i \tilde{y}^\top - \tilde{v}_i \vec{1}^\top - \vec{1}\tilde{w}_i^\top \leq 0. $$

Letting $(\lambda_i \circ \tilde{v}_i) = \hat{v}_i$, $(\lambda_1 \circ \tilde{w}_i) = \hat{w}_i$, and $(\lambda_1\circ \tilde{y}) = \hat{y}$, we have (\ref{finn}) become
\begin{equation}
\vec{S}_i ((\lambda_1)^{-1}\circ \hat{y})^\top - ((\lambda_i)^{-1}\circ \hat{v}_i) \vec{1}^\top - \vec{1}((\lambda_1)^{-1} \circ \hat{w}_i)^\top \leq 0, \label{last}
\end{equation}
and  (\ref{finn2}) become
$$ {\cal H}_{F_1}((\lambda_1)^{-1} \circ \hat{y})\in \bars{{\cal C}},$$
and (\ref{fir}), (\ref{finn3}) reduce respectively to (\ref{refff}) (with $j=1$) and (\ref{refff2}) (with $t=1$). 

Finally, multiplying (\ref{last}) with $(\vec{1} \lambda_1^\top)$, we have
$$ (\vec{1} \lambda_1^\top) \circ (\vec{S}_i ((\lambda_1)^{-1}\circ \hat{y})^\top - ((\lambda_i)^{-1}\circ \hat{v}_i) \vec{1}^\top - \vec{1}((\lambda_1)^{-1} \circ \hat{w}_i)^\top) \leq 0, $$ which leads to the final formulation for (\ref{refff3}) (with $j=1$).
\halmos
\end{proof}

\subsection*{Proof of Proposition \ref{prodec}}
\begin{proof}{Proof of Proposition \ref{prodec}}
Following Proposition \ref{pro3}, we can equivalently formulate the problem
as
\begin{eqnarray*}
\min_{x',\delta_{Z},y_{0},y_{Z}} && ||x'-x^{T}||\\
{\rm subject\;to\;} && y_{0}^{\top}Zx'\leq\delta_{Z},\\
 && \delta_{Z}-y_{Z}^{\top}Zx'\leq0,\\
 && y_{Z}^{\top}Zx'\leq\min_{x} \left\{y_{Z}^{\top}Zx\;|\;Ax\geq b \right\},\\
 && y_{0},\; y_{Z}\in {\cal C}.
\end{eqnarray*}
Observe first that given any solution $x'$,  $y_{Z}$, one can always set
$y_{0}=y_{Z}$ and $\delta_{Z}=y_{Z}^{\top}Zx'$ to satisfy the first
two constraints. The third constraint, by definition, is equivalent
to the constraint (\ref{invlin}). Hence, the above problem is indeed equivalent
to the first optimization problem in Proposition \ref{prodec}.

We can equivalently state the constraint (\ref{invlin}) in terms of the KKT
condition for linear programs, which is
\begin{align}
 & Ax'-b\geq0, \label{eq:1-1}\\
 & A^{\top}u=Z^{\top}y, \nonumber \\
 & u\geq0, \label{eq:12}\\
 & (Ax'-b)^{\top}u=0.\label{eq:13}
\end{align}

It is known that constraints (\ref{eq:1-1}), (\ref{eq:12}), and
(\ref{eq:13}) can be equivalently stated as the linear complementarity
constraints $(Ax'-b)_{i}\cdot u_{i}=0,\;\forall i$ (\cite{Z-Q:1996aa}).
By introducing binary variables $\eta_{i}\in\{0,1\}$, we can apply
the Big-M method to equivalently formulate these constraints as the
constraints \eqref{dec0}--\eqref{dec4} provided that the constant $M$ is sufficiently
large.

Finally, because $\delta_{Z}=y_{Z}^{\top}Zx'$, we can obtain the risk
function $\rho_{\delta}$ by setting $\vec{X}_{1}:=Zx^{*}$, $\vec{X}_{2}=\vec{0}$, 
and $\delta_1:= \delta_{Z}^*=y_{Z}^{*\top}Zx^{*}, \; \delta_{2}=0$ with the optimal
solution $x^{*}$ and $y_{Z}^{*}$. \Halmos
\end{proof}

\subsection*{Proof of Proposition \ref{perminv}}
\begin{proof}{Proof of Proposition \ref{perminv}}
Following the proof in Proposition \ref{prop_final} (in particular, \eqref{eq:reduction1}), we can equivalently formulate the problem as
\begin{eqnarray}
\min_{x',\delta_{Z},y_{0},y_{Z}} && ||x'-x^{T}||\nonumber \\
{\rm subject\;to\;} && y_{0}^{\top}\sigma(Zx')\leq\delta_{Z}, \;\;\; \forall\sigma\in\Sigma,  \nonumber \\
 && \delta_{Z}-y_{Z}^{\top}Zx'\leq0, \nonumber \\
 && y_{Z}^{\top}(\sigma(Zx')-Zx')\leq0, \;\;\; \forall\sigma\in\Sigma, \label{eq:perpe}\\
 && y_{Z}^{\top}Zx'\leq\min_{x}\left\{y_{Z}^{\top}Zx\;|\;Ax\geq b\right\}, \nonumber \\
 && y_{0,}\; y_{Z}\in \bars{{\cal C}}.\nonumber 
\end{eqnarray}

To show how one may reduce the problem, let us start by focusing on
the constraint (\ref{eq:perpe}):
\[
y_{Z}^{\top}\sigma(Zx')\leq y_{Z}^{\top}Zx',\;\forall\sigma\in\Sigma.
\]
By taking a closer look at this inequality, we can recognize that
it is an instance of the rearrangement inequality, which states that
$y_{Z}$ is feasible to the constraints if and only if its ordering
matches the ordering of $Zx'$:
\[
(Zx')_{i}\leq(Zx')_{j}\Leftrightarrow(y_{Z})_{i}\leq(y_{Z})_{j},\;\forall i\neq j.
\]
We can thus apply the Big-M method to reformulate the above condition
as the constraints (\ref{invperm1}), (\ref{invperm2}), (\ref{invperm}), and (\ref{invperm4}) with sufficiently large M. 

Moreover, for any $y_{Z} \in \bars{{\cal C}}$ satisfying the above matching condition, we can observe that 
\[
y_{Z}^{\top}Zx' \geq (\frac{1}{|\Omega|}{\bf 1})^{\top}Zx'
\]
must hold. To see this, one may assume without loss of generality that 
$(Zx')_1 \leq \cdots \leq (Zx')_{|\Omega|}$ and to minimize $y_{Z}^{\top}Zx'$
over $y_Z \in \bars{{\cal C}}$ that satisfies $(y_Z)_1 \leq \cdots \leq (y_Z)_{|\Omega|}$, one can confirm that $y_Z = \frac{1}{|\Omega|}{\bf 1}$ gives the minimum.
This implies that given any $y_{Z}$ satisfying (\ref{eq:perpe}), 
there always exists a feasible $y_{0}$ and $\delta_{Z}$ for the
first two constraints, namely by setting $y_{0}=\frac{1}{|\Omega|}{\bf 1}$
and $\delta_{Z}=y_{Z}^{\top}Zx'$. Hence, we
can remove the first two constraints and arrive at the final formulation.

Finally, we can obtain the risk function $\bars{\rho}_{\delta}$ by setting $\vec{X}_{1}:=Zx^{*}$, 
$\vec{X}_{2}=\vec{0}$, and $\delta_1:=\delta_{Z}^*=y_{Z}^{* \top}Zx^{*}$ and $\delta_{2}=0$ with the optimal solution $x^{*}$ and  $y_{Z}^{*}$. \Halmos
\end{proof}

\subsection*{Proof of Proposition \ref{perminv2}}
\begin{proof}{Proof of Proposition \ref{perminv2}}
Like the proof of Proposition \ref{lastpros}, we first construct random variables
in a outcome space $\Omega$ endowed with a uniform probability measure.
We can then apply Proposition \ref{perminv}, and thereafter we show how the problem
can be further reduced.

Given the finite supports $\{\xi_{o}^{\top}x\}_{o=1}^{\tau_{0}}$
and the associated probability values $\{\bar{p}_{o}^{\xi}\}_{o=1}^{\tau_{0}}$,
we define random loss $Zx$ as a mapping from $\Omega$ to $\mathbb{R}$
that satisfies $(Zx)_i \in\{\xi_{o}^{\top}x\}_{o=1}^{\tau_{0}}$
and $| \left\{i\;|\;(Zx)_i =\xi_{o}^{\top}x \right\}|=\bar{p}_{o}^{\xi}M$.
Let ${\cal I}_{o}$ denote the set of indices $i$ of $(Zx)_i$ such
that $(Zx)_i=\xi_{o}^{\top}x$ and therefore $|{\cal I}_{o}|=\bar{p}_{o}^{\xi}M$.
Using this definition of $Zx$, we can apply Proposition \ref{perminv} to formulate
the MIP model (\ref{invperm}). We claim that given any feasible solution $({x'}^{*}$, $y^{*}$, $u^{*}$, $\eta^{*}$, $\nu_{i,j}^{*})$ to (\ref{invperm}), one can always retain feasibility after replacing $y^{*}$ by an alternative
$y^{**}$ that satisfies for any $i\in{\cal I}_{o}$,
\[
(y^{**})_{i}=\frac{1}{|{\cal I}_{o}|}\sum_{a\in{\cal I}_{o}}(y^{*})_{a},
\]
$o=1,...,\tau_{0}$. 

First, given the definition of ${\cal I}_{o}$, clearly for any $y^{*}$
that are ordered the same way as $Z{x'}^{*}$  (i.e., $(Z{x'}^{*})_i \leq (Z{x'}^{*})_j\Leftrightarrow y_{i}^{*}\leq y_{j}^{*}$),
it must hold also that $(Z{x'}^{*})_i\leq (Z{x'}^{*})_j\Leftrightarrow y_{i}^{**}\leq y_{j}^{**}$. It is also clear that $A^{\top}u^{*}=Z^{\top}y^{*}=Z^{\top}y^{**}$.
Finally, to check $y^{**}\in \bars{{\cal C}}$, 
one can find the necessary arguments to prove it in the proof of Proposition  \ref{lastpros}.

Hence, this implies that without loss of generality, we can impose
for some $\tilde{y}\in\mathbb{R}^{\tau_{0}}$ that for $a\in{\cal I}_{o}$, $o=1,...,\tau_{0}$,
$(y)_{a}=(\tilde{y})_{o}$ must hold. This leads to the constraints (\ref{lawperm}) and (\ref{lawperm2}) in $\tilde{y}$ and the constraint $A^{\top}u=\Xi^{\top}(\bar{p}_{o}^{\xi}M\circ\tilde{y})$
and ${\cal H}_{F_{\xi}}(\tilde{y})\in \bars{{\cal C}}$.
By setting $y=M\tilde{y}$, we arrive at the final formulation. \Halmos
\end{proof}

\subsection*{Proof of Example \ref{spec2} (5. Spectral risk measures)}
\begin{proof}{Proof of Example \ref{spec2} (5. Spectral risk measures)}
Let $M \in \mathbb{Z}^+$ be a constant such that $p_k$, $k=1,...,K$, and $\bar{p}_o$, $o=1,...,\tau_j$, can be expressed in the form of $n/M$, $n \in \{1,...,M\}$. First, given the stepwise spectrum  $\phi^{-}(p)$, to apply the representation ${\cal C}$ in Example \ref{specex} (i.e., $\bars{{\cal C}}$), we have $\phi_j = \int_{\frac{j-1}{M}}^{\frac{j}{M}}\phi^{-}(t)dt = \frac{\bar{\phi}_k}{M}$ for any $j \in \{1,...,M\}$ such that $p_{k-1} \leq \frac{j-1}{M} < \frac{j}{M} \leq p_k$, and therefore $| \left\{ j \; \middle |\; \phi_j = \frac{\bar{\phi}_k}{M} \right\} | = (p_k - p_{k-1})M$. To see how the constraint ${\cal H}_{F_j}((\lambda_{F_j})^{-1} \circ y) \in \bars{{\cal C}}$ can be reduced, we apply first the result of \cite{birkhoff:tola} to reformulate $\bars{{\cal C}}$ into 
$$\bars{{\cal C}}= \left\{ q\; \middle | q = Q\phi,\; Q\vec{1} = \vec{1},\; Q^\top \vec{1} = \vec{1},\; Q \geq 0 \right\}.$$ 

Let ${\cal I}_o^{(j)}$,  $o=1,...,\tau_j$,  denote the set of indices $n$ of ${\cal H}_{F_j}((\lambda_{F_j})^{-1} \circ y)$ such that $({\cal H}_{F_j}((\lambda_{F_j})^{-1} \circ y))_n = ((\lambda_{F_j})^{-1} \circ y)_o$ and therefore $|{\cal I}_o^{(j)}| = (\lambda_{F_j})_o$.  It is obvious that the constraint ${\cal H}_{F_j}((\lambda_{F_j})^{-1} \circ y) \in \bars{{\cal C}}$ has a feasible solution if and only if the following set of constraints 
\begin{equation} \label{cc}
q_i = q_j, \;\forall i,j \in {\cal I}_o^{(j)},\; o=1,...,\tau_j,\;q = Q\phi,\; Q\vec{1} = \vec{1},\; Q^\top \vec{1} = \vec{1},\; Q \geq 0
\end{equation}
has a feasible solution. We show first how (\ref{cc}) can be reduced. Let $(q^*,Q^*)$ denote a feasible solution for the above constraints. We claim that the solution $q^*$ together with the following construction of $Q^{**}$
$$ (Q^{**})_{(\tilde{n},:)} = \frac{1}{|{\cal I}_o^{(j)}|} \sum_{n \in {\cal I}_o^{(j)} } (Q^*)_{(n,:)},\;\tilde{n} \in {\cal I}_o^{(j)},\; o=1,...,\tau_j$$
will also be feasible. The notation $(V)_{(k,:)}$ (respectively $(V)_{(:,k)}$)
refers to the $k$th-row (respectively $k$th-column) of the matrix $V$. The claim can be fairly straightforward to verify by direct substitution, which gives $q^*=Q^{**}\phi$, $Q^{**}\vec{1}=\vec{1}$, and ${Q^{**}}^\top\vec{1}=\vec{1}$.

Hence, we can reduce (\ref{cc}) by imposing that for any $\tilde{n} \in {\cal I}_o^{(j)}$, $ o=1,...,\tau_j$, 
 $(Q)_{(\tilde{n},:)} = (\bar{Q})_{(o,:)}$ for some $\bar{Q}\in \Re^{\tau_j\times M}$, which leads to 
\begin{equation} \label{qq}
q_i = \bar{q}_o, \;\forall i\in {\cal I}_o^{(j)},\; o=1,...,\tau_j,\;
\bar{q} = \bar{Q}\phi,\; \bar{Q}\vec{1} = \vec{1},\; ((\lambda_{F_j}\vec{1}^\top) \circ \bar{Q})^\top \vec{1} = \vec{1},\; \bar{Q} \geq 0,
\end{equation}
where $\bar{q} \in \Re^{\tau_j}$. Moreover, the constraint ${\cal H}_{F_j}((\lambda_{F_j})^{-1} \circ y) \in \bars{{\cal C}}$ can be equivalently written as $y\in \Re^{\tau_j}_+ \cap {\cal C}$, where
$$ {\cal C} := \left \{ \bar{q} \; \middle |\; \bar{q} = ((\lambda_{F_j}\vec{1}^\top)\circ\bar{Q})\phi, \;\bar{Q}\vec{1} = \vec{1}, \; ((\lambda_{F_j}\vec{1}^\top) \circ \bar{Q})^\top \vec{1} = \vec{1}, \;\bar{Q}\geq 0\right\}.$$ 

Letting $\hat{Q} = (\lambda_{F_j}\vec{1}^\top) \circ \bar{Q}$, we have 
\begin{equation}
 {\cal C} = \left\{ \bar{q} \; \middle |\; \bar{q} = \hat{Q}\phi,\; \hat{Q}\vec{1} = \lambda_{F_j},\; \hat{Q}^\top \vec{1} = \vec{1},\; \hat{Q}\geq 0\right\}. \label{qq2}
\end{equation}

Next, let ${\cal I}_k^{(\phi)}$ denote the set of indices $j$ of $\phi$ such that $\phi_j =  \frac{\bar{\phi}_k}{M}$ for $k=1,...,K$  and therefore $|{\cal I}_k^{(\phi)}| = (p_k-p_{k-1})M$. Given this, we show that the constraints in (\ref{qq2}) can be further reduced. Let $\bar{q}^*$, $\hat{Q}^*$ be a feasible solution for (\ref{qq2}). We claim that $\bar{q}^*$ together with the following construction of $\hat{Q}^{**}$ 
$$ \hat{Q}^{**}_{(:, \tilde{n})} : = \frac{1}{|{\cal I}_k^{(\phi)}|} \sum_{n \in {\cal I}_k^{(\phi)}} \hat{Q}^*_{(:,n)},\; \tilde{n} \in {\cal I}_k^{(\phi)},\; k=1,...,K
$$
is also feasible for the constraints. Similarly, the claim can be verified by the direct substitution, which gives
$ \hat{Q}^{**}\phi = \bar{q}^*$, $\hat{Q}^{**}\vec{1} = \lambda_{F_j}$, and $
(\hat{Q}^{**})^\top \vec{1} = \vec{1}. $

Hence, we can also impose that for any $\tilde{n} \in {\cal I}_k^{(\phi)}$, $k=1,..,K$, $\hat{Q}(:,\tilde{n}) = \tilde{Q}(:,k)$ for some $\tilde{Q} \in \Re^{\tau_j \times K}$ in (\ref{qq2}), which leads to the reformulation of the first and second constraint in ${\cal C}$ into 
$$\bar{q} = ((\vec{1} \lambda_{\phi}^\top)\circ \tilde{Q})(\frac{1}{M}\bar{\phi}),\;
\text{and } ((\vec{1}\lambda_{\phi}^\top)\circ \tilde{Q})\vec{1} = \lambda_{F_j},$$
where $(\lambda_\phi)_k := (p_k-p_{k-1})M$, and therefore also the set ${\cal C}$ into
$${\cal C} = \left\{ \bar{q} \; \middle |\; \bar{q} =  (\frac{1}{M})(\vec{1}\lambda_{\phi}^\top)\circ \tilde{Q}  \bar{\phi},\; ((\vec{1}\lambda_{\phi}^\top)\circ \tilde{Q})\vec{1} = \lambda_{F_j},\; \tilde{Q}^\top \vec{1} = \vec{1},\; \tilde{Q}\geq 0 \right\}.$$
Letting $\dot{Q} = (\frac{1}{M})(\vec{1}\lambda_{\phi}^\top)\circ \tilde{Q}$, we arrive at the final reduced form. \halmos
\end{proof}

\newpage
\section{Further discussion about the issue of constraint misspecification raised in Remark \ref{rm4}} \label{apee}
As mentioned in Remark \ref{rm4}, it is possible that the forward problem assumed in our inverse models does not well represent the true problem that the decision maker solved. 
In this section, we attempt to discuss this issue more formally and provide some ideas that might help (partially) resolve the issue.

The forward problem in this paper is completely characterized by the feasible sets ${\cal X}^t$, $t\in {\cal T}$, and from this point on we assume that the sets take the form of ${\cal X}^t := \left\{ x\; \middle | \; g^t_j(x)\leq 0,\; j=1,...,J^t \right\}$, $t \in {\cal T}$. To differentiate the feasible sets ${\cal X}^t$ assumed in our inverse models from the ``true" feasible sets  (i.e., the ones based on which the past decisions $x^t$, $ t\in {\cal T}$ were optimized), we denote by ${\bar {\cal X}}^t$, $t\in {\cal T}$ the true feasible sets and thus $x^t \in {\bar{\cal X}}^t$, $t \in {\cal T}$ follows. To facilitate our discussion,  we assume that the decision makers are rational
\footnote{This assumption is in fact necessary; otherwise in the case where there does not exist a risk function $\rho$ that fits all the observed decisions, one cannot tell if it is because of the sub-optimality of the observed decisions or the misspecification of constraints.}
  (i.e., the past decisions were optimally made with respect to the true risk function $\rho^*$). We can thus write down the following optimality condition that must hold for the observed decisions $x^t$, $t \in {\cal T}$ with respect to the true risk function $\rho^*$:
\begin{equation} \label{1eq}
\rho^*(\vec{Z}^t(x^t))\leq \rho^*(\vec{Z}^t(x)),\; \forall x \in {\bar{\cal X}}^t, \;t \in {\cal T}. 
\end{equation}

Recall that our inverse models take as input the observed decisions $x^t$, $t\in {\cal T}$ and the feasible sets ${\cal X}^t$, $t\in {\cal T}$ and that they seek a risk function $\rho$ from the following set:  
\begin{equation} \label{2eq}
{\cal R}_{inv} = \left\{ \rho \;\middle |\; 
\rho(\vec{Z}^t(x^t))\leq \rho(\vec{Z}^t(x)),\;\forall x \in {\cal X}^t, \;t \in {\cal T} \right\}.
\end{equation}
The question here is how the discrepancy between ${\cal X}^t$ and $\bar{{\cal X}}^t$ would affect the risk function $\rho$ generated from the inverse models and if the models can actually detect such a discrepancy. A trivial case where we can easily draw the conclusion, without even running the inverse models, is that if there are some observed decisions that are simply not feasible with respect to ${\cal X}^t$  (i.e., $x^t \not\in {\cal X}^t$), then one can easily tell which constraint is misspecified by checking which one is violated  (i.e., $g^t_j(x^t) > 0$ for some $j \in \{1,...,J^t\}$ and $t \in {\cal T}$). 

We thus focus primarily on the case where all the observed decisions are feasible with respect to ${\cal X}^t$  (i.e., $x^t \in {\cal X}^t$, $t \in {\cal T}$) but the feasible sets ${\cal X}^t$ may be misspecified  (i.e., ${\cal X}^t \neq {\bar {\cal X}}^t$). The inverse models in this case could either return a message of infeasibility  (i.e., ${\cal R}_{inv} = \emptyset$) or a feasible risk function $\rho$ that can fit all the observed decisions with respect to ${\cal X}^t$. The case of infeasibility would allow one to detect the sets ${\cal X}^t$ being misspecified, because we know from \eqref{1eq} and \eqref{2eq} that if the sets ${\cal X}^t$ are correctly specified  (i.e., ${\cal X}^t = {\bar {\cal X}}^t$), then $\rho^* \in {\cal R}_{inv}$  (i.e., ${\cal R}_{inv} \neq \emptyset$). In the case where the inverse models do return a feasible risk function $\rho$, it appears, however, as discussed below, generally not possible to detect any misspecification of ${\cal X}^t$.

In particular, let us consider the case motivated by our portfolio management example where the client might not necessarily take into account the long-only constraint. In this case, the feasible set ${\cal X}^t$ is only a subset of the true feasible set ${\bar {\cal X}}^t$  (i.e., ${\cal X}^t \subset {\bar{\cal X}}^t$), and we know from \eqref{1eq} that the following must hold:
\begin{equation} \label{insight}
  \rho^*(\vec{Z}^t (x^t)) \leq \rho^*(\vec{Z}^t(x)), \;\forall x \in {\cal X}^t, \;t\in {\cal T}.
\end{equation}
That is, we have $\rho^* \in {\cal R}_{inv}$. In other words, despite ${\cal X}^t \neq {\bar {\cal X}}^t$, in this case the inverse models can still narrow down a smaller set of risk functions that includes the true risk function $\rho^*$ as a feasible candidate. This also demonstrates why one would not be able to detect the discrepancy between ${\cal X}^t$ and ${\bar {\cal X}}^t$, given that it is not even  possible to rule out $\rho^*$ from the set ${\cal R}_{inv}$. The only difference in this case between applying the set ${\cal X}^t$ and the true feasible set ${\bar {\cal X}}^t$ in \eqref{2eq} is that the latter might allow for a faster rate of convergence of the set ${\cal R}_{inv}$ to the true risk function $\rho^*$. But it appears not possible to detect such a difference without knowing what the true risk function is. In other cases where ${\cal X}^t \neq {\bar {\cal X}}^t$ but ${\cal X}^t  \not\subset {\bar{\cal X}}^t$, it is possible that the true risk function $\rho^*$ may no longer be feasible  (i.e., $\rho^* \notin {\cal R}_{inv}$). It remains unclear, however, how one can detect such infeasibility, given that all one knows from the output of the inverse models is that there exists a risk function $\rho$ that perfectly fits all the observed decisions.

Although in general it appears not possible to detect $\rho^* \notin {\cal R}_{inv}$ (and thus also ${\cal X}^t \neq {\bar {\cal X}}^t$), we provide here some idea as to how one may quantify potential risk underestimation because of constraint misspecification. The idea here can also be applied to the case where one is able to confirm $\rho^* \notin {\cal R}_{inv}$  (e.g., the case ${\cal R}_{inv} = \emptyset$ as mentioned earlier) and tries to fix the issue. This idea comes from the observation made earlier that if ${\cal X}^t \subset {\bar{\cal X}}^t$, our inverse models would necessarily still include the true risk function $\rho^{*}$ as a feasible candidate. So, given the sets of potentially misspecified feasible sets ${\cal X}^t$, $t \in {\cal T}$ and observed decisions $x^t$, $t\in {\cal T}$, what one can do is to construct alternative feasible sets ${\cal X}^t_{*}$ that satisfy 
\begin{equation} \label{inclu}
x^t \in {\cal X}^t_{*} \subset { {\cal X}^t \cap \bar {\cal X}}^t, \;t \in {\cal T},
\end{equation}
(i.e., each ${\cal X}^t_*$ is a subset of both the potentially misspecified set and the true feasible set). By replacing the sets ${\cal X}^t$, $t\in {\cal T}$ with the subsets ${\cal X}^t_*$, $t \in {\cal T}$ in \eqref{2eq}, one can apply instead the following set in the inverse models:
\begin{equation} \label{3eq}
{\cal R}_{inv}^* = \left\{ \rho \;\middle |\; 
\rho(\vec{Z}^t(x^t))\leq \rho(\vec{Z}^t(x)),\;\forall x \in {\cal X}^t_*, \;t \in {\cal T} \right\}.
\end{equation}
The condition  ${\cal X}^t_{*}\subset {\cal X}^t$ implies that ${\cal R}_{inv} \subset {\cal R}_{inv}^*$, and the condition ${\cal X}^t_{*} \subset \bar{{\cal X}}^t$ implies that the above set always contains the true risk function $\rho^*$  (i.e., $\rho^* \in {\cal R}_{inv}^*$).  
With these guarantees, one can apply our inverse model \eqref{eq:inv3} to generate a risk function $\rho^{\uparrow}$ that bounds from above the true risk function $\rho^*$  (i.e., $\rho^{\uparrow} \geq \rho^*$) and any risk function generated from the original set ${\cal R}_{inv}$  (i.e., $\rho^{\uparrow} \geq \rho$,  $\forall \rho \in {\cal R}_{inv}$). Hence, $\rho^{\uparrow}$ provides a means to measure potential risk underestimation from constraint misspecification. More specifically, one may compare the risk estimates obtained from a risk function that is generated with respect to potentially misspecified constraints against $\rho^{\uparrow}$. If the difference is small, this may be regarded as a signal that the impact of constraint misspecification is less of a concern. Otherwise, the difference provides the upper bound on the potential risk underestimation from constraint misspecification. Ideally, one should seek the largest $
{\cal X}^t_{*}$ that satisfies \eqref{inclu}, given that the larger the constructed set 
${\cal X}^t_{*}$ is, the tighter the bound $\rho^{\uparrow}$ is (and hence the more informative the bound is).

Here, we give two examples of how the set ${\cal X}_t^*$ may be constructed. In the first one, we assume the form of the constraints $g_j^t(x)$ is known but it is only some of the ``right-hand-side" parameters that are prone to misspecification: $g_j^t(x) := h_j^t (x) - b_j^t \leq 0,\; j=1,...,J^t$, where $h_j^t(x)$ is known but $b_j^t$ may be misspecified for some $j$ and $t$. For example, it is common also in portfolio management that upper bounds must be imposed over the amount of investment for each asset  (i.e., $g_j^t(x) = x_j - b_j^t \leq 0$, $ j=1,...,{\rm dim}(x),\; t \in {\cal T}$) for diversification purposes,  but the values of the bounds may not be known exactly. To construct the set ${\cal X}_t^*$ in this case, one can lower the values of the parameters $b_j^t $ that are prone to misspecification, so that they guarantee $x^t \in {\cal X}^t_{*} \subset {\bar {\cal X}}^t$ (e.g., by setting $b_j^t := h_j^t (x^t)$). In the second example, we assume that the true feasible set $\bar{{\cal X}}^t$ does not change over time  (i.e.,  $\bar{{\cal X}}^t = \bar{{\cal X}}$, $\forall t \in {\cal T}$). One can then always construct the convex hull of the observed decisions ${\cal X}_{*} := {\rm Conv}(\{x^t\}_{t\in {\cal T}})$ so that ${\cal X}_* \subset {\bar {\cal X}}$. These examples are meant for demonstrating the general principle that one may apply for constructing the set ${\cal X}_t^*$.

\end{document}